\documentclass[10pt,reqno,final]{amsart}
\usepackage{epsfig,amssymb,amsmath,version}
\usepackage{amssymb,version,graphicx,fancybox,mathrsfs}
\usepackage[notcite,notref]{showkeys}
\usepackage{url,hyperref}
\usepackage{subfigure}
\usepackage{color}
\usepackage{stmaryrd}
\usepackage{multirow}
\usepackage{booktabs,siunitx}
\usepackage{bm} 
\usepackage {verbatim}
\usepackage{booktabs}
\usepackage{multirow}
\usepackage[title]{appendix}
\usepackage[ruled,vlined]{algorithm2e}
\allowdisplaybreaks[4]
\catcode`\@=11 \theoremstyle{plain}
\@addtoreset{equation}{section}   
\renewcommand{\theequation}{\arabic{section}.\arabic{equation}}
\@addtoreset{figure}{section}
\renewcommand\thefigure{\thesection.\@arabic\c@figure}
\newtheorem{thm}{\bf Theorem}

\newenvironment{theorem}{\begin{thm}} {\end{thm}}
\newtheorem{cor}{\bf Corollary}
\newtheorem{prop}{Proposition}[section]

\newtheorem{lmm}{\bf Lemma}

\theoremstyle{remark}
\newtheorem{rem}{\bf Remark}[section]

\def \epsilon {{\varepsilon}}

\definecolor{bgblue}{rgb}{0.04,0.39,0.54}
\definecolor{lired}{rgb}{0.3, 0.0, 0.0}
\definecolor{ligreen}{rgb}{0.0, 0.3, 0.0}
\definecolor{liblue}{rgb}{0.9, 1.0, 1.0}
\definecolor{gray}{rgb}{0.6, 0.6, 0.6}
\definecolor{sky}{rgb}{0.3, 1.0, 1.0}
\definecolor{bunhong}{rgb}{1.0, 0.3, 1.0}
\definecolor{yellow}{rgb}{0.97, 1, 0.0}
\definecolor{liyellow}{rgb}{0.9, 0.8, 0.0}
\definecolor{cengse}{rgb}{0.00,0.40,0.29}

\newcommand{\bs}[1]{\boldsymbol{#1}}

\renewcommand \wedge \times

\begin{document}
\bibliographystyle{plain}

{\title[Rotational discrete gradient method for biaxial elasticity gradient flow] {A second-order $SO(3)$-preserving and energy-stable scheme for orthonormal frame gradient flow model of biaxial nematic liquid crystals}
\author[
  H. Wang,\,    J. Xu\,  $\&$\,  Z. Yang
  ]{
	  \;\; Hanbin Wang${}^{1}$,   \;\;  Jie Xu${}^{2}$ \;\; and\;\; Zhiguo Yang${}^{*,3}$
	  }
\thanks{
\noindent${}^{1}$ School of Mathematical Sciences, Shanghai Jiao Tong University, Shanghai 200240, China. Email: hb.wang@sjtu.edu.cn (H. Wang).
  \\
  \indent${}^{2}$ LSEC and NCMIS, Institute of Computational Mathematics and Scientific/Engineering Computing (ICMSEC), Academy of Mathematics and Systems Science (AMSS), Chinese Academy of Sciences, Beijing, China. Emails: xujie@lsec.cc.ac.cn (J. Xu).
  \\
    \indent ${}^{3*}$ Corresponding author. School of Mathematical Sciences, MOE-LSC and CMA-Shanghai, Shanghai Jiao Tong University, Shanghai 200240, China. Email: yangzhiguo@sjtu.edu.edu (Z. Yang)
	  }
	  }

\begin{abstract}
  In this paper, we present a novel second-order generalised rotational discrete gradient scheme for numerically approximating the orthonormal frame gradient flow of biaxial nematic liquid crystals.
  This scheme relies on reformulating the original gradient flow system into an equivalent generalised ``rotational'' form.
  A second-order discrete gradient approximation of the energy variation is then devised such that it satisfies an energy difference relation.
  The proposed numerical scheme has two remarkable properties: (i) it strictly obeys the orthonormal property of the tensor field and (ii) it satisfies the energy dissipation law at the discrete level, regardless of the time step sizes.
  We provide ample numerical results to validate the accuracy, efficiency, unconditional stability and $SO(3)$-preserving property of this scheme.
  In addition, comparisons of the simulation results between the biaxial orthonormal frame gradient flow model and uniaxial Oseen-Frank gradient flow are made to demonstrate the ability of the former to characterize non-axisymmetric local anisotropy. 
\end{abstract}
\maketitle

\section{Introduction}\label{sect: Int}
Biaxial nematic liquid crystals are featured by spontaneously non-axisymmetric local anisotropy.
It was possibly proposed very early in 1970 \cite{freiser1970ordered} and confirmed experimentally decades later \cite{madsen2004thermotropic}, and is believed to have prospective applications in many areas such as semiconductors because of its distinct properties from uniaxial nematics \cite{berardi2008field, han2018highly, liu2016biaxial, mundoor2018hybrid}. 


In mathematical theories, similar to uniaxial nematics, models involving biaxial nematic phases are constructed at different levels.
When phase transitions are addressed, molecular model or tensor model needs to be adopted, which is constructed for spatially homogeneous cases \cite{bisi2006universal, sonnet2003dielectric, straley1974ordered,xu2014microscopic,xu2017transmission} and later further developed to include elasticity \cite{lubensky2002theory,xu2022symmetry,xu2020general,xu2018tensor}. 
The study of biaxial hydrodynamics is restricted to spatially homogeneous cases under steady shear flow \cite{li2010biaxial,sircar2008shear,sircar2009dynamics, xu2018onsager}. 
If one would focus on elasticity, a shortcut description is to assume that the local orientational order is exactly the ideal biaxial nematics. 
In this way, the local orientation is simplified to an orthonormal frame, and we refer to models by orthonormal frame field as frame models.
The elastic energy of orthonormal frame field \cite{alben1973phase,saupe1981elastic,stallinga1994theory} and the corresponding hydrodynamics \cite{brand1981hydrodynamics,brand1982number,brand1982theory,govers1985fluid,kini1984isothermal,liu1981hydrodynamic,saslow1982hydrodynamics} have been developed. 
Recently, the connection between molecular-theory-based tensor hydrodynamics and biaxial frame hydrodynamics is rigorously established \cite{li2023frame,li2023rigorous}. 
Computational investigation on dynamics of biaxial nematics in spatially inhomogeneous cases would be desirable but is not addressed in existing works. 

In this work, we focus on the numerical approximation of frame models. 
Let us denote the orthonormal frame field as $\mathbf{p}=\left(\bs{n}_1, \bs{n}_2, \bs{n}_3\right) \in S O(3)$.
The full biaxial frame hydrodynamics is a coupling system between $\mathbf{p}$ and the fluid velocity field $\bs{v}$.
Here, we consider a simplification of small velocity by assuming that $\bs{v}=0$.
In this case, the frame dynamics is written as 
\begin{subequations}\label{eq: gfmodelori}
\begin{align}
& \chi_1\frac{\partial \bs n_2}{\partial t} \cdot \bs{n}_3+\mathscr{L}_1 \mathcal{F}_{B i}=0, \\
& \chi_2\frac{\partial \bs n_3}{\partial t} \cdot \bs{n}_1+\mathscr{L}_2 \mathcal{F}_{B i}=0, \\
& \chi_3\frac{\partial \bs n_1}{\partial t} \cdot \bs{n}_2+\mathscr{L}_3 \mathcal{F}_{B i}=0, \\
& \mathbf{p}=\left(\bs{n}_1, \bs{n}_2, \bs{n}_3\right) \in S O(3).
\end{align}
\end{subequations}
In the above, 
$\mathscr{L}_i,\, i=1,2,3$ are the (variational) differential operators of infinitesimal rotations, defined as
\begin{equation}\label{eq: Li}
\left\{\begin{array}{l}
\mathscr{L}_1:=\boldsymbol{n}_3 \cdot \frac{\delta}{\delta {\boldsymbol{n}_2}}-\boldsymbol{n}_2 \cdot \frac{\delta}{\delta{\boldsymbol{n}_3}}, \\[4pt]
\mathscr{L}_2:=\boldsymbol{n}_1 \cdot \frac{\delta}{\delta{\boldsymbol{n}_3}}-\boldsymbol{n}_3 \cdot \frac{\delta}{\delta{\boldsymbol{n}_1}}, \\[4pt]
\mathscr{L}_3:=\boldsymbol{n}_2 \cdot \frac{\delta}{\delta{\boldsymbol{n}_1}}-\boldsymbol{n}_1 \cdot \frac{\delta}{\delta{\boldsymbol{n}_2}} .
\end{array}\right.
\end{equation}
The biaxial orientational elasticity, 
\begin{equation}\label{eq: bienergy}
\mathcal{F}_{B i}[\mathbf{p}]=\int f_{B i}(\mathbf{p}, \nabla \mathbf{p}) \mathrm{d} {\bs x}
\end{equation}
takes the form 
\begin{equation}\label{eq: density}
  \begin{aligned}
  f_{B i}(\mathbf{p}, \nabla \mathbf{p})= 
  & \frac{1}{2} \Big\{   K_1\left(\nabla \cdot \bs{n}_1\right)^2+K_2\left(\nabla \cdot \bs{n}_2\right)^2+K_3\left(\nabla \cdot \bs{n}_3\right)^2    \\
  & +K_4\left(\bs{n}_1 \cdot \nabla \times \bs{n}_1\right)^2+K_5\left(\bs{n}_2 \cdot \nabla \times \bs{n}_2\right)^2+K_6\left(\bs{n}_3 \cdot \nabla \times \bs{n}_3\right)^2 \\
  & +K_7\left(\bs{n}_3 \cdot \nabla \times \bs{n}_1\right)^2+K_8\left(\bs{n}_1 \cdot \nabla \times \bs{n}_2\right)^2+K_9\left(\bs{n}_2 \cdot \nabla \times \bs{n}_3\right)^2 \\
  & +K_{10}\left(\bs{n}_2 \cdot \nabla \times \bs{n}_1\right)^2+K_{11}\left(\bs{n}_3 \cdot \nabla \times \bs{n}_2\right)^2+K_{12}\left(\bs{n}_1 \cdot \nabla \times \bs{n}_3\right)^2 \\
  & +\gamma_1 \nabla \cdot \big[\left(\bs{n}_1 \cdot \nabla\right) \bs{n}_1-\left(\nabla \cdot \bs{n}_1\right) \bs{n}_1\big]+\gamma_2 \nabla \cdot\big[\left(\bs{n}_2 \cdot \nabla\right) \bs{n}_2-\left(\nabla \cdot \bs{n}_2\right) \bs{n}_2\big] \\
  & +\gamma_3 \nabla \cdot\big[\left(\bs{n}_3 \cdot \nabla\right) \bs{n}_3-\left(\nabla \cdot \bs{n}_3\right) \bs{n}_3\big]\Big\}. 
  \end{aligned}
\end{equation}	
The three $\gamma_i$ terms can be rewritten as surface integrals that vanish under suitable boundary conditions. 
The coefficients in the elastic energy, as well as three diffusion coefficients $\chi_i$, can be calculated from molecular parameters \cite{li2023frame,xu2018calculating}. 
We will demonstrate that system \eqref{eq: gfmodelori} can be viewed as a constrained gradient flow  for the biaxial elastic energy \eqref{eq: bienergy} later.

To the knowledge of the authors, there is no attempt to numerically solve the above system, maybe due to several challenges. 
First of all, it is a highly nonlinear and coupled system involving three unknown variables $\bs n_1,\bs n_2, \bs n_3$ and six point-wise constraints. The nonlinearity and coupling come from both the nonlinear terms in the energy variation of the biaxial elastic energy functional and the explicit orthonormality constraint. A possible way to eliminate the constraint is to reparameterize the local frame $\mathbf{p}(\bs x,t)$ in terms of Euler angles $\theta(\bs x,t)$, $\phi(\bs x,t)$ and $\psi(\bs x,t)$, 
\begin{equation*}
\mathbf{p}(\bs x,t)=
\begin{pmatrix}
\cos\theta \cos\phi \cos\psi-\sin\phi\sin\psi   & -\cos\theta \cos\phi \sin\psi-\sin\phi \cos\psi & \sin\theta\cos\phi    \\
\cos \theta \sin\phi \cos\psi+\cos\phi\sin\psi & -\cos\theta  \sin\phi \sin\psi+\cos\phi\cos\psi & \sin\theta \sin\phi \\
-\sin\theta \cos\psi  & \sin\theta \sin\psi & \cos\theta
\end{pmatrix}.
\end{equation*}
Actually, earlier work writes down the model as PDEs w.r.t. the three Euler angles \cite{govers1985fluid}. 
However, the free energy and its derivatives then turn into expressions extremely complicated. 
Secondly, suggested by the derivation from molecular parameters \cite{xu2018calculating}, the elastic coefficients in the free energy density \eqref{eq: density} can become highly disparate with the change of molecular architecture, 
leading to highly anisotropic derivative terms that are difficult to compute numerically. 
This requires efficient and robust numerical methods which can deal with highly anisotropic derivatives, strong nonlinearity and coupling. 
Last, and most importantly, this system obeys two fundamental physical properties: the point-wise orthonormality constraint of $\mathbf{p}(\bs x,t)$ and the energy dissipation law. 
It is highly desirable for a numerical method to preserve these two properties simultaneously at the discrete level. 
On one hand, the form of biaxial elastic energy density $ f_{B i}(\mathbf{p}, \nabla \mathbf{p})$ relies upon the assumption that $\mathbf{p}\in SO(3)$, and a violation of this constraint may lead to unpredictable computational results that are not characterized by biaxial elastic energy. 
On the other hand, a naive correction or projection of $\mathbf{p}$ onto $SO(3)$ may induce numerical instability and cause energy of the gradient flow system to increase unphysically.

The aim of this paper is to propose an efficient and robust numerical scheme for the orthonormal frame gradient flow system \eqref{eq: gfmodelori}--\eqref{eq: density} that can preserve the properties of orthonormality and energy dissipative law simultaneously at the discrete level. Among these two properties, preservation of the orthonormality constraint discretely plays a crucial role in producing physically meaningful simulation results. 
If focusing on this sole target, one may be tempted to use Euler angles. 
However, it is tricky to deal with the discontinuity across periods for these Euler angles and the energy functional becomes too cumbersome to design unconditionally energy-stable schemes based upon this formulation.
Another possible technique is to introduce Lagrange multipliers to reformulate the $SO(3)$-constrained gradient flow system into a saddle point problem \cite{tschumperle2002orthonormal}.
But systems of this kind are generally harder to solve numerically. 
Meanwhile, since \eqref{eq: gfmodelori}--\eqref{eq: density} is a dissipative system, it is favorable to solve it using energy-stable schemes. 
Classical method for developing energy stable schemes include, but not limited to convex splitting \cite{Elliott_and_Stuart-1993,Eyre-1998}, linear stablization \cite{Cai_Shen_and_Xu-2017,Shen_Yang-2010}, and auxiliary variable methods such as IEQ \cite{Yang-2016, Zhao_Yang_Li_and_Wang-2016} and SAV \cite{Shen_Xu_Yang-2018,Shen_Xu_Yang-2019,yang2020gpav}. 
However, preserving energy dissipation and orthonormal constraint simultaneously using classical methods is a highly nontrivial task. This challenge necessitates the development of new numerical methods. 

In this paper, we propose a second-order unconditionally energy-stable and $SO(3)$-preserving generalised rotational discrete gradient (gRdg) scheme for the orthonormal frame gradient flow system of biaxial nematic liquid crystal. Two key steps are essential for preserving these properties at the same time. One is to reformulate the gradient flow system with orthonormality constraint into an equivalent system, which can be viewed as a generalised ``rotational'' form. The other is to develop a discrete gradient approximation for the variational derivative of biaxial elastic energy functional such that it satisfies a certain energy difference relation.
These two ingredients, together with a time-centered Crank-Nicolson scheme, constitute the proposed method.
We present rigorous proof that the gRdg scheme strictly preserves the orthonormality constraint and satisfies a discrete energy dissipative law.
An inexact Newton-Krylov iterative method is employed to solve the nonlinear equation resulted from the gRdg method.
We also propose a time-adaptive strategy to the proposed scheme, which proves to be quite efficient for accelerating the nonlinear solution algorithm by ample numerical experiments. 
We examine the evolution of the orthonormal frames with various elastic parameters, compare them with those obtained from Oseen-Frank gradient flows of uniaxial nematics, and report a few featured evolution patterns of biaxial nematic phases. 
In particular, we investigate the dynamics with scale-different elastic coefficients obtained from molecular parameters.
In simulating these challenging numerical examples, the gRdg scheme demonstrates its robustness and accuracy, and intriguing dynamics is revealed that could draw further research interests. 

The rest of the paper is organized as follows. 
In Section \ref{sect: Rdg}, we reformulate the original orthonormal frame gradient flow system with explicit orthonormality constraint into an equivalent rotational form. We then present the gRdg scheme for temporal discretization of the reformulated system, and develop a discrete gradient approximation for the variation derivative of biaxial elastic energy satisfying the energy difference relation. 
We rigorously prove that the gRdg scheme is strictly $SO(3)$-preserving and unconditionally energy-stable. 
Techniques for implementation, including the iterative solver and time-adaptive strategy are also discussed.
In Section \ref{sect: num}, we conduct ample numerical experiments of orthonormal frame gradient flow with various elastic coefficients to demonstrate the accuracy, efficiency, exact orthonormality preservation and energy stability of gRdg method.
Concluding remarks are given in Section \ref{sect: remark}.

\section{The generalised rotational  discrete gradient Method}\label{sect: Rdg}

\subsection{Notations} 
We briefly introduce some notations and conventions that will be used throughout the paper. We use $\bs n=(n_i)$ and $\mathbf{p}=(p_{ij})$ to denote vectors and matrices respectively.
Define
\begin{equation}\label{eq: SO3}
SO(3):=\big\{\Lambda \in \mathbb{R}^{3\times 3}\big|  \Lambda^{\intercal} \Lambda=\mathbf{I},\; {\rm det}(\Lambda)=1  \big\},\qquad so(3):=\big\{ \Lambda \in \mathbb{R}^{3\times 3}\big| \Lambda^{\intercal}+\Lambda=\mathbf{0}   \big\},
\end{equation} 
are the set of noncommutative Lie group of orthonormal frames and the set of all skew symmetric tensors, respectively. The tangential space of $\mathbf{p}=(\bs n_1, \bs n_2, \bs n_3)\in SO(3)$ is given by 
\begin{equation}\label{eq: TpSO(3)}
T_{\mathbf{p}} S O(3):=\left\{{\mathbf{\Theta}}_{\mathbf{p}}:=\mathbf{p} {\mathbf{\Theta}},\; \forall \mathbf{\Theta} \in so(3)  \right\},
\end{equation}
for which we choose a basis as 
  \begin{equation}\label{eq: Tpbasis}
  \mathbf{V}_1=\left(0, \bs{n}_3,-\bs{n}_2\right), \quad \mathbf{V}_2=\left(-\bs{n}_3, 0, \bs{n}_1\right), \quad \mathbf{V}_3=\left(\bs{n}_2,-\bs{n}_1, 0\right).
  \end{equation} 
A brief derivation for the expression of $T_{\mathbf{p}} S O(3)$ is provided in Appendix \ref{sect: app}. For further details on $SO(3)$ group, we refer readers to \cite[p. 181-194]{dewitt1978analysis}. 


\subsection{Reformulated equivalent system}
In this section, we reformulate the original system with $SO(3)$ constraint  into an equivalent system in generalised ``rotational'' form, on which the gRdg scheme is constructed.
\begin{prop}
The  system \eqref{eq: gfmodelori} is equivalent to the following generalised rotational form
  \begin{equation}\label{eq: rotform}
  \begin{aligned}
 &  \frac{\partial \mathbf{p}}{\partial t}=\mathbf{p}  \mathbf{A},\\
 &  \mathbf{p}=(\bs n_1, \bs n_2, \bs n_3)\in SO(3), 
  \end{aligned}
  \end{equation}
with $\mathbf{A}\in so(3)$ taking the form
  \begin{equation}\label{eq: A}
  \mathbf{A}=
\begin{pmatrix}
  0 & \frac{1}{\chi_3}\mathscr{L}_3 \mathcal {F}_{B_i}[\mathbf{p}] & -\frac{1}{\chi_2}\mathscr{L}_2 \mathcal {F}_{B_i}[\mathbf{p}] \\
  -\frac{1}{\chi_3}\mathscr{L}_3 \mathcal {F}_{B_i}[\mathbf{p}] & 0 & \frac{1}{\chi_1}\mathscr{L}_1 \mathcal {F}_{B_i}[\mathbf{p}] \\
  \frac{1}{\chi_2}\mathscr{L}_2 \mathcal {F}_{B_i}[\mathbf{p}] & -\frac{1}{\chi_1}\mathscr{L}_1 \mathcal {F}_{B_i}[\mathbf{p}] & 0
\end{pmatrix}.
  \end{equation}

\end{prop}
 \begin{proof}
For $\mathbf{p}=(\bs n_1, \bs n_2, \bs n_3)\in SO(3)$, since $\{\bs n_i,i=1, 2, 3\}$ is a set of the normal orthogonal basis  of $\mathbb{R}^3$. For $i,j=1, 2, 3$, the following equations hold:
\begin{subequations}\label{eq: dn}
\begin{align}
&\frac{\partial \bs n_1}{\partial t}=\big(\frac{\partial \bs n_1}{\partial t} \cdot \bs{n}_2 \big) \bs{n}_2+\big(\frac{\partial \bs n_1}{\partial t} \cdot \bs{n}_3 \big) \bs{n}_3, \label{eq: dn1}\\
&\frac{\partial \bs n_2}{\partial t}=\big(\frac{\partial \bs n_2}{\partial t} \cdot \bs{n}_1 \big) \bs{n}_1+\big(\frac{\partial \bs n_2}{\partial t} \cdot \bs{n}_3 \big) \bs{n}_3, \label{eq: dn2}\\
&\frac{\partial \bs n_3}{\partial t}=\big(\frac{\partial \bs n_3}{\partial t} \cdot \bs{n}_1 \big) \bs{n}_1+\big(\frac{\partial \bs n_3}{\partial t} \cdot \bs{n}_2 \big) \bs{n}_2, \label{eq: dn3}\\
&\frac{\partial(\bs n_i \cdot \bs n_j)}{\partial t}=\frac{\partial \bs n_i}{\partial t} \cdot {\bs n}_j+\frac{\partial \bs n_j}{\partial t} \cdot {\bs n}_i=0.  \label{eq: dn4}
\end{align}
\end{subequations}
By replacing ${\partial_t \bs n_i} \cdot {\bs n}_j$ terms on the right hand side of equations \eqref{eq: dn1}-\eqref{eq: dn3} using equations \eqref{eq: gfmodelori}, we can directly obtain \eqref{eq: rotform}. The formulation \eqref{eq: rotform} can be regarded as a generalised rotational form since for every $\mathbf{p} \in SO(3),\,\mathbf{A} \in so(3)$, because $\mathbf{p}\mathbf{A}   \in T_{\mathbf{p}}SO(3)$.

On the other hand, by taking dot product respectively between ${\partial_t \bs n_2},{\partial_t \bs n_3},{\partial_t \bs n_1}$  and $\bs n_3,\bs n_1,\bs n_2$ in equation \eqref{eq: rotform}, we can recover the system \eqref{eq: gfmodelori}. 
\end{proof}

\begin{theorem}
Given  $\mathbf{p}\in SO(3)$, the system \eqref{eq: rotform} satisfies the following energy dissipation law, 
\begin{equation}\label{eq: energylaw}
  \frac{d\mathcal{F}_{B_i}[\mathbf{p}]}{dt}=-\sum_{k=1}^3\frac{1}{\chi_k} \Big\|\mathscr{L}_k \mathcal{F}_{B i}[\mathbf{p}]\Big\|_{L^2}^2 \leq 0.
\end{equation}
Moreover, this system can be recast into an constrained gradient flow of the following form
\begin{equation}\label{eq: gf2}
\begin{aligned}
&\frac{\partial \mathbf{p}(\bs x,t)}{\partial t}=  \mathbf{p}\Big(\boldsymbol{\hat{\chi}}\odot \mathcal{S}\Big[\mathbf{p}^{\intercal}\frac{\delta \mathcal{F}_{Bi}[\mathbf{p}]}{\delta \mathbf{p}}\Big]\Big)\\
 &  \mathbf{p}=(\bs n_1, \bs n_2, \bs n_3)\in SO(3)
\end{aligned},\quad {\rm with}\;\;
\boldsymbol{\hat{\chi}}=-\begin{pmatrix}
 0 & 1/{\chi_3} & 1/{\chi_2}\\
 1/{\chi_3} & 0 & 1/{\chi_1}\\
 1/{\chi_2}& 1/{\chi_1} & 0
\end{pmatrix},\;\; \mathcal{S}[\mathbf{a}]=\mathbf{a}-\mathbf{a}^{\intercal},
\end{equation}
which satisfies an equivalent energy dissipative law
\begin{equation}\label{eq: eqlaw}
\frac{d \mathcal{F}_{Bi}[\mathbf{p}]}{d t}=\frac{1}{2}\int_{\Omega}\Big(\mathcal{S}\Big[\mathbf{p}^{\intercal}\frac{\delta \mathcal{F}_{Bi}[\mathbf{p}]}{\delta \mathbf{p}}\Big]\Big) \cdot \Big(\boldsymbol{\hat{\chi}}\odot (\mathcal{S}\Big[\mathbf{p}^{\intercal}\frac{\delta \mathcal{F}_{Bi}[\mathbf{p}]}{\delta \mathbf{p}}\Big]\Big)dV\leq 0.
\end{equation}
Here, $\odot$ is the Hadamard product.
\end{theorem}
\begin{proof}
Let us first prove the dissipative law \eqref{eq: energylaw}. By taking inner product between equation \eqref{eq: rotform} and $ {\delta \mathcal{F}[\mathbf{p}]}/{\delta \mathbf{p}}$, one can obtain that
\begin{equation}\label{eq: ener1}
 \frac{d \mathcal{F}_{B_i}[\mathbf{p}]}{d t}=\int_{\Omega} \frac{\delta \mathcal{F}_{B_i}[\mathbf{p}]}{\delta \mathbf{p}} \cdot \frac{\partial \mathbf{p}(\boldsymbol{x}, \boldsymbol{t})}{\partial t} d V=\int_{\Omega} \frac{\delta \mathcal{F}_{B_i}[\mathbf{p}]}{\delta \mathbf{p}} \cdot (\mathbf{p}  \mathbf{A})\, d V.
\end{equation}
With the help of the expression of $\mathbf{A}$ in equation \eqref{eq: A}, equation \eqref{eq: ener1} can be expanded as  
\begin{equation*}
  \begin{aligned}
	 \frac{d \mathcal{F}_{B_i}[\mathbf{p}]}{d t}  & = \int_{\Omega}\Big\{ \frac{\delta \mathcal{F}_{B_i}[\mathbf{p}]}{\delta \boldsymbol{n}_1} \cdot\Big(\frac{1}{\chi_3}\big(-\frac{\delta \mathcal{F}_{B_i}[\mathbf{p}]}{\delta \boldsymbol{n}_1} \cdot \boldsymbol{n}_2+\frac{\delta \mathcal{F}_{B_i}[\mathbf{p}]}{\delta \boldsymbol{n}_2} \cdot \boldsymbol{n}_1\big) \boldsymbol{n}_2+\frac{1}{\chi_2}\big(\frac{\delta \mathcal{F}_{B_i}[\mathbf{p}]}{\delta \boldsymbol{n}_3} \cdot \boldsymbol{n}_1-\frac{\delta \mathcal{F}_{B_i}[\mathbf{p}]}{\delta \boldsymbol{n}_1} \cdot \boldsymbol{n}_3\big) \boldsymbol{n}_3\Big) \\
	   &\qquad+\frac{\delta \mathcal{F}_{B_i}[\mathbf{p}]}{\delta \boldsymbol{n}_2} \cdot\Big(\frac{1}{\chi_3}\big(\frac{\delta \mathcal{F}_{B_i}[\mathbf{p}]}{\delta \boldsymbol{n}_1} \cdot \boldsymbol{n}_2-\frac{\delta \mathcal{F}_{B_i}[\mathbf{p}]}{\delta \boldsymbol{n}_2} \cdot \boldsymbol{n}_1\big) \boldsymbol{n}_1+\frac{1}{\chi_1}\big(-\frac{\delta \mathcal{F}_{B_i}[\mathbf{p}]}{\delta \boldsymbol{n}_2} \cdot \boldsymbol{n}_3+\frac{\delta \mathcal{F}_{B_i}[\mathbf{p}]}{\delta \boldsymbol{n}_3} \cdot \boldsymbol{n}_2\big) \boldsymbol{n}_3\Big) \\
	  &\qquad +\frac{\delta \mathcal{F}_{B_i}[\mathbf{p}]}{\delta \boldsymbol{n}_2} \cdot\Big(\frac{1}{\chi_2}\big(\frac{\delta \mathcal{F}_{B_i}[\mathbf{p}]}{\delta \boldsymbol{n}_3} \cdot \boldsymbol{n}_1-\frac{\delta \mathcal{F}_{B_i}[\mathbf{p}]}{\delta \boldsymbol{n}_1} \cdot \boldsymbol{n}_3\big) \boldsymbol{n}_1+\frac{1}{\chi_1}\big(-\frac{\delta \mathcal{F}_{B_i}[\mathbf{p}]}{\delta \boldsymbol{n}_2} \cdot \boldsymbol{n}_3+\frac{\delta \mathcal{F}_{B_i}[\mathbf{p}]}{\delta \boldsymbol{n}_3} \cdot \boldsymbol{n}_2\big) \boldsymbol{n}_3\Big) \Big\}d V.
  \end{aligned}
\end{equation*}
By collecting terms in the above equation, one obtain
\begin{equation}\label{eq: FBi}
\begin{aligned}
\frac{d \mathcal{F}_{B_i}[\mathbf{p}]}{d t}  &=- \int_{\Omega}\left\{\frac{1}{\chi_1}\left(\frac{\delta \mathcal{F}_{B_i}[\mathbf{p}]}{\delta n_2} \cdot \boldsymbol{n}_3-\frac{\delta \mathcal{F}_{B_i}[\mathbf{p}]}{\delta \boldsymbol{n}_3} \cdot \boldsymbol{n}_2\right)^2+\frac{1}{\chi_2}\left(\frac{\delta \mathcal{F}_{B_i}[\mathbf{p}]}{\delta \boldsymbol{n}_3} \cdot \boldsymbol{n}_1-\frac{\delta \mathcal{F}_{B_i}[\mathbf{p}]}{\delta \boldsymbol{n}_1} \cdot \boldsymbol{n}_3\right)^2\right. \\
	  &\qquad \qquad \;\;\; \left.+\frac{1}{\chi_3}\left(\frac{\delta \mathcal{F}_{B_i}[\mathbf{p}]}{\delta \boldsymbol{n}_1} \cdot \boldsymbol{n}_2-\frac{\delta \mathcal{F}_{B_i}[\mathbf{p}]}{\delta \boldsymbol{n}_2} \cdot \boldsymbol{n}_1\right)^2\right\} d V. \\
\end{aligned}
\end{equation}
Aware of the expressions of $\mathscr{L}_i,\, i=1,2,3$ in equation \eqref{eq: Li}, we have 
\begin{equation*}
\frac{d \mathcal{F}_{B_i}[\mathbf{p}]}{d t}=- \sum_{k=1}^3\dfrac{1}{\chi_k}\left\|\mathscr{L}_k \mathcal{F}_{B i}\right\|_{L^2}^2\leq 0.
\end{equation*}

Next, we show that system \eqref{eq: rotform} is indeed a constrained gradient flow of the form \eqref{eq: gf2}.
Considering the constrained variational derivation $ {\delta \mathcal{F}_{Bi}[\mathbf{p}]}/{\delta \mathbf{p}}\big|_{\mathbf{p} \in SO(3)}$,  it is equivalent to projecting ${\delta \mathcal{F}_{Bi}[\mathbf{p}]}/{\delta \mathbf{p}}$ onto the tangential plane of $SO(3)$ at $\mathbf{p}$.
Thus, it can be expressed by a linear combination of the complete orthogonal basis $\{\mathbf{V}_i \}_{i=1}^3$ of $T_{\mathbf{p}} S O(3)$ in equation \eqref{eq: Tpbasis}
  \begin{equation}\label{eq: expand}
\frac{\delta \mathcal{F}_{Bi}[\mathbf{p}]}{\delta \mathbf{p}}\bigg|_{\mathbf{p} \in SO(3)}=\sum_{i=1}^3  c_i \mathbf{V}_i,
  \end{equation}
  where $c_1, c_2, c_3$ are the expansion coefficients. One can take dot product between the above equation and  $\mathbf{V}_i$ to obtain
  \begin{equation}\label{eq: ci}
  c_i=\frac{1}{2}\frac{\delta \mathcal{F}_{Bi}[\mathbf{p}]}{\delta \mathbf{p}} \cdot \mathbf{V}_i,\;\; i=1,2,3.
  \end{equation}
With the help of  equation \eqref{eq: ci} and  equation \eqref{eq: expand} and a direct calculation, we obtain
\begin{equation}\label{eq: condf}
\frac{\delta \mathcal{F}_{Bi}[\mathbf{p}]}{\delta \mathbf{p}}\bigg|_{\mathbf{p} \in SO(3)}=\frac{1}{2}\mathbf{p}\mathcal{S}\Big[\mathbf{p}^{\intercal}\frac{\delta \mathcal{F}_{Bi}[\mathbf{p}]}{\delta \mathbf{p}}\Big].
 \end{equation}
Inserting the above equation into
\begin{equation}\label{eq: gfmodel}
\frac{\partial \mathbf{p}}{\partial t} =2\mathbf{p}\Big(\boldsymbol{\hat{\chi}}\odot\mathbf{p}^{\intercal}\frac{\delta \mathcal{F}_{Bi}[\mathbf{p}]}{\delta \mathbf{p}}\bigg|_{\mathbf{p} \in SO(3)}\Big)
\end{equation}
we arrive at the desired constrained gradient flow \eqref{eq: gf2}.

Finally, we turn to the equivalent dissipative law \eqref{eq: eqlaw}.  By taking the inner product between the variational derivation ${\delta \mathcal{F}_{Bi}[\mathbf{p}]}/{\delta \mathbf{p}}$ with equation \eqref{eq: gf2}, we obtain
\begin{equation}\label{eq: enderi}
\begin{aligned}
\frac{d \mathcal{F}_{Bi}[\mathbf{p}]}{d t}=&\int_{\Omega} \frac{\delta \mathcal{F}_{Bi}[\mathbf{p}]}{\delta \mathbf{p}} \cdot\Big(\mathbf{p}\big(\boldsymbol{\hat{\chi}}\odot (\mathcal{S}\big[\mathbf{p}^{\intercal}\frac{\delta \mathcal{F}_{Bi}[\mathbf{p}]}{\delta \mathbf{p}}\big]\big)\Big)dV\\
=&\int_{\Omega}\Big(\mathbf{p}^{\intercal}\frac{\delta \mathcal{F}_{Bi}[\mathbf{p}]}{\delta \mathbf{p}}\Big) \cdot \Big(\boldsymbol{\hat{\chi}}\odot (\mathcal{S}\big[\mathbf{p}^{\intercal}\frac{\delta \mathcal{F}_{Bi}[\mathbf{p}]}{\delta \mathbf{p}}\big]\Big)dV\\
=&\frac{1}{2}\int_{\Omega}\Big(\mathbf{p}^{\intercal}\frac{\delta \mathcal{F}_{Bi}[\mathbf{p}]}{\delta \mathbf{p}}\Big) \cdot \Big(\boldsymbol{\hat{\chi}}\odot (\mathcal{S}\big[\mathbf{p}^{\intercal}\frac{\delta \mathcal{F}_{Bi}[\mathbf{p}]}{\delta \mathbf{p}}\big]\Big)\\
&+\Big(\mathbf{p}^{\intercal}\frac{\delta \mathcal{F}_{Bi}[\mathbf{p}]}{\delta \mathbf{p}}\Big)^{\intercal} \cdot \Big(\boldsymbol{\hat{\chi}}\odot (\mathcal{S}\big[\mathbf{p}^{\intercal}\frac{\delta \mathcal{F}_{Bi}[\mathbf{p}]}{\delta \mathbf{p}}\big]\Big)^{\intercal}dV\\
=&\frac{1}{2}\int_{\Omega}\Big(\mathbf{p}^{\intercal}\frac{\delta \mathcal{F}_{Bi}[\mathbf{p}]}{\delta \mathbf{p}}\Big) \cdot \Big(\boldsymbol{\hat{\chi}}\odot (\mathcal{S}\big[\mathbf{p}^{\intercal}\frac{\delta \mathcal{F}_{Bi}[\mathbf{p}]}{\delta \mathbf{p}}\big]\Big)\\
&-\Big(\mathbf{p}^{\intercal}\frac{\delta \mathcal{F}_{Bi}[\mathbf{p}]}{\delta \mathbf{p}}\Big)^{\intercal} \cdot \Big(\boldsymbol{\hat{\chi}}\odot (\mathcal{S}\big[\mathbf{p}^{\intercal}\frac{\delta \mathcal{F}_{Bi}[\mathbf{p}]}{\delta \mathbf{p}}\big]\Big)dV\\
=&\frac{1}{2}\int_{\Omega}\Big(\mathcal{S}\big[\mathbf{p}^{\intercal}\frac{\delta \mathcal{F}_{Bi}[\mathbf{p}]}{\delta \mathbf{p}}\big]\Big) \cdot \Big(\boldsymbol{\hat{\chi}}\odot (\mathcal{S}\big[\mathbf{p}^{\intercal}\frac{\delta \mathcal{F}_{Bi}[\mathbf{p}]}{\delta \mathbf{p}}\big]\Big)dV.
\end{aligned}
\end{equation}
Since $\hat{\chi}$ is symmetric and every entry of it is negative, it is direct to verify that $\mathbf{A} \cdot \hat{\chi}\odot \mathbf{A}\leq 0 $ for arbitrary $\mathbf{A}$, thus right-hand side of equation \eqref{eq: enderi} is guaranteed to be non-positive, which ends the proof. 
\end{proof}

\subsection{An energy-stable and $SO(3)$-preserving scheme}
With the help of the generalised rotational form \eqref{eq: rotform}, we shall propose a second-order  unconditionally energy-stable and $SO(3)$-preserving scheme. Let $\tau$ be the time step, and the  superscript $n$ is used to represent variables at $t=n\tau$.

The gRdg scheme takes the form
\begin{equation}\label{eq: rotform2}
\frac{\mathbf{p}^{n+1}-\mathbf{p}^n}{\tau}=\mathbf{p}^{n+\frac{1}{2}} \mathbf{A}^{n+\frac{1}{2}},
\end{equation}
with $\mathbf{p}^{n+\frac{1}{2}}:={(\mathbf{p}^{n+1}+\mathbf{p}^{n})}/{2}$ and 
\begin{equation}\label{eq: disA}
A^{n+\frac{1}{2}}=\begin{pmatrix}
0 & \frac{1}{\chi_3}(\mathscr{L}_3 \mathcal {F}_{B_i}[\mathbf{p}])^{n+\frac{1}{2}} & -\frac{1}{\chi_2}(\mathscr{L}_2 \mathcal {F}_{B_i}[\mathbf{p}])^{n+\frac{1}{2}} \\
-\frac{1}{\chi_3}(\mathscr{L}_3 \mathcal {F}_{B_i}[\mathbf{p}])^{n+\frac{1}{2}} & 0 & \frac{1}{\chi_1}(\mathscr{L}_1 \mathcal {F}_{B_i}[\mathbf{p}])^{n+\frac{1}{2}} \\
\frac{1}{\chi_2}(\mathscr{L}_2 \mathcal {F}_{B_i}[\mathbf{p}])^{n+\frac{1}{2}} & -\frac{1}{\chi_1}(\mathscr{L}_1 \mathcal {F}_{B_i}[\mathbf{p}])^{n+\frac{1}{2}} & 0
\end{pmatrix}.
\end{equation} 
Let us first show that the above scheme is $SO(3)$-preserving.
From equation \eqref{eq: rotform2}, we have 
\begin{equation}\label{eq: peq}
 \mathbf{p}^{n+1}\Big(\mathbf{I}-\frac{\tau}{2}\mathbf{A}^{n+\frac{1}{2}}\Big)=\mathbf{p}^{n}\Big(\mathbf{I}+\frac{\tau}{2}\mathbf{A}^{n+\frac{1}{2}}\Big).
\end{equation}
For $\mathbf{p}^n \in SO(3)$, since $\mathbf{A}^{n+\frac{1}{2}}$ is $3\times 3$ and skew-symmetric, its eigenvalues are $\lambda,-\lambda,0$ for $\lambda$ being pure imaginary. Thus, it is direct to obtain that $\mathbf{I}\pm ({\tau}/2)\mathbf{A}^{n+\frac{1}{2}}$ is invertible and ${\rm det} \big(\mathbf{I}\pm ({\tau}/{2})\mathbf{A}^{n+\frac{1}{2}} \big) > 0.$
Taking right multiplication of $\big(\mathbf{I}-({\tau}/{2}) \mathbf{A}^{n+\frac{1}{2}}\big)^{-1}$ on equation \eqref{eq: peq}, one obtains 
\begin{equation}\label{eq: peq2}
  \mathbf{p}^{n+1}=\mathbf{p}^{n}\Big(\mathbf{I}+\frac{\tau}{2}\mathbf{A}^{n+\frac{1}{2}} \Big)\Big(\mathbf{I}-\frac{\tau}{2} \mathbf{A}^{n+\frac{1}{2}}\Big)^{-1}.
\end{equation}
Again, using the fact that $\mathbf{A}^{n+\frac{1}{2}}$ is skew-symmetric, one can take transpose of equation \eqref{eq: peq2} to obtain 
  \begin{equation}\label{eq: ptrans}
  \begin{aligned}
  (\mathbf{p}^{n+1})^{\intercal}=&\Big[\mathbf{p}^{n}\Big(\mathbf{I}+\frac{\tau}{2}\mathbf{A}^{n+\frac{1}{2}} \Big)\Big(\mathbf{I}-\frac{\tau}{2}\mathbf{A}^{n+\frac{1}{2}}\Big)^{-1} \Big]^{\intercal}\\
  =&\Big(\mathbf{I}+\frac{\tau}{2}\mathbf{A}^{n+\frac{1}{2}}\Big)^{-1} \Big(\mathbf{I}-\frac{\tau}{2}\mathbf{A}^{n+\frac{1}{2}}\Big)(\mathbf{p}^{n})^{\intercal}.
  \end{aligned}
  \end{equation}
 Combining equations \eqref{eq: peq2} and \eqref{eq: ptrans}, we arrive at 
 \begin{equation}\label{eq: pptrans}
 \begin{aligned}
  (\mathbf{p}^{n+1})^{\intercal}\mathbf{p}^{n+1}&=\Big(\mathbf{I}+\frac{\tau}{2}\mathbf{A}^{n+\frac{1}{2}}\Big)^{-1} \Big(\mathbf{I}-\frac{\tau}{2}\mathbf{A}^{n+\frac{1}{2}}\Big)(\mathbf{p}^{n})^{\intercal} \mathbf{p}^{n}\Big(\mathbf{I}+\frac{\tau}{2}\mathbf{A}^{n+\frac{1}{2}} \Big)\Big(\mathbf{I}-\frac{\tau}{2} \mathbf{A}^{n+\frac{1}{2}}\Big)^{-1}\\
  &=\Big(\mathbf{I}+\frac{\tau}{2}\mathbf{A}^{n+\frac{1}{2}}\Big)^{-1} \Big(\mathbf{I}+\frac{\tau}{2}\mathbf{A}^{n+\frac{1}{2}}\Big)\Big(\mathbf{I}-\frac{\tau}{2}\mathbf{A}^{n+\frac{1}{2}} \Big)\Big(\mathbf{I}-\frac{\tau}{2} \mathbf{A}^{n+\frac{1}{2}}\Big)^{-1}=\mathbf{I},
 \end{aligned}
 \end{equation} 
 where the fact that  $\mathbf{p}^{n} \in SO(3)$ and the identity $(\mathbf{I}-\mathbf{A})(\mathbf{I}+\mathbf{A})=\big(\mathbf{I}+\mathbf{A})(\mathbf{I}-\mathbf{A} \big)$ have been used.  Moreover, from the above equation, we have $\text{det}(\mathbf{p}^{n+1})=\pm 1$. Together with equation \eqref{eq: peq2} and $\text{det}(\mathbf{p}^{n})=1$, we can deduce $\text{det}(\mathbf{p}^{n+1})=1$. Thus, we conclude $\mathbf{p}^{n+1}\in SO(3)$.

Now we can focus on the design of $\mathbf{A}^{n+1/2}$ such that the proposed gRdg scheme satisfies the discrete energy dissipation law. The $(\mathscr{L}_i \mathcal {F}_{B_i}[\mathbf{p}])^{n+\frac{1}{2}},\; i=1,2,3$ terms in equation \eqref{eq: disA} are discretized by
\begin{equation}
(\mathscr{L}_i \mathcal {F}_{B_i}[\mathbf{p}])^{n+\frac{1}{2}}=\mathbf{V}_i^{n+\frac{1}{2}}\cdot \left.D_{\mathcal{F}}(\mathbf{p})\right|^{n+\frac{1}{2}},
\end{equation}
with $\mathbf{V}_i^{n+\frac{1}{2}}=({\mathbf{V}_i^{n}+\mathbf{V}_i^{n+1}})/{2} $ and $D_{\mathcal{F}}(\mathbf{p})\Big|^{n+\frac{1}{2}}$ is a second-order approximation of ${\delta \mathcal{F}_{B_i}[\mathbf{p}]}/{\delta \mathbf{p}}$ satisfying the energy difference relation
\begin{equation}\label{eq: DDprop}
\int_{\Omega} D_{\mathcal{F}}(\mathbf{p})\Big|^{n+\frac{1}{2}}\cdot (\mathbf{p}^{n+1}-\mathbf{p}^n) dV=\mathcal{F}_{B_i}[ \mathbf{p}^{n+1}]-\mathcal{F}_{B_i}[ \mathbf{p}^{n}].
\end{equation}

One can take the inner product between equation \eqref{eq: rotform2} and  $D_{\mathcal{F}}(\mathbf{p})\Big|^{n+\frac{1}{2}}$ and utilizing the energy difference relation \eqref{eq: DDprop} to obtain the desired energy dissipation law by exactly the same procedure as the continuous case. Hence, it suffices to construct the discrete gradient satisfying the discrete difference relation \eqref{eq: DDprop}.

\subsection{The biaxial discrete gradient}
In what follows, we shall propose a second-order discrete gradient approximation of ${\delta \mathcal{F}_{B_i}[\mathbf{p}]}/{\delta \mathbf{p}}$, which satisfies the energy difference relation \eqref{eq: DDprop}.
To fix the idea, we assume periodic boundary conditions, and will explain how to deal with non-periodic boundary conditions afterwards. 
In this case, the three $\gamma$ terms vanish after integrated, so that $\gamma_i$ can be arbitrarily chosen. 
To ensure that $\mathcal{F}_{B i}[\mathbf{p}]$ is bounded from below, the twelve elastic coefficients $K_i$ are assumed positive.
We begin with reformulating the biaxial energy density function into an equivalent form to fully utilize its structure. 
\begin{prop}
The biaxial elasticity $f_{B i}(\mathbf{p}, \nabla \mathbf{p})$ can be reformulated as 
\begin{equation}\label{eq: density2}
f_{B i}(\mathbf{p}, \nabla \mathbf{p})=\frac{1}{2} \sum_{i=1}^3 \gamma_i\left|\nabla \bs{n}_i\right|^2+W(\mathbf{p}, \nabla \mathbf{p}),
\end{equation}
where 
\begin{align}
W(\mathbf{p}, \nabla \mathbf{p})=&\frac{1}{2}\Big(\sum_{i=1}^3 k_i\left(\nabla \cdot \bs{n}_i\right)^2+\sum_{i, j=1}^3 k_{i j}\left(\bs{n}_i \cdot \nabla \times \bs{n}_j\right)^2\Big),
\end{align}
and the coefficients are given by 
\begin{align}\label{co:density}
  &\gamma_1  =\min \left\{K_1, K_4, K_7, K_{10}\right\}>0, \gamma_2=\min \left\{K_2, K_5, K_8, K_{11}\right\}>0 \notag\\
  &\gamma_3  =\min \left\{K_3, K_6, K_9, K_{12}\right\}>0\notag\notag\\
  &k_1=K_1-\gamma_1, \quad k_2=K_2-\gamma_2, \quad k_3=K_3-\gamma_3\notag\notag\notag\\
  & k_{11}=K_4-\gamma_1, \quad k_{22}=K_5-\gamma_2, \quad k_{33}=K_6-\gamma_3, \notag\notag\notag\notag\\
   &k_{31}=K_7-\gamma_1, \quad k_{12}=K_8-\gamma_2, \quad k_{23}=K_9-\gamma_3, \notag\notag\notag\notag\notag\\
  & k_{21}=K_{10}-\gamma_1, \quad k_{32}=K_{11}-\gamma_2, \quad k_{13}=K_{12}-\gamma_3.\notag\\
  &k_{i} \geq 0,k_{ij} \geq 0 \quad \quad (i,j =1,2,3).
  \end{align}
\end{prop}
\begin{proof}
Under  periodic boundary conditions, the last three term of \eqref{eq: density} do not contribute to the gradient, so we can just let $\gamma_i$ take the value in \eqref{co:density}.
Since $\mathbf{p}=(\bs {n}_1,\bs {n}_2,\bs {n}_3)^{\intercal} \in  SO(3)$, one has
\begin{align}
\left|\bs{n}_1 \times\left(\nabla \times \bs{n}_1\right)\right|^2&=\left(\bs{n}_2 \cdot \nabla \times \bs{n}_1\right)^2+\left(\bs{n}_3 \cdot \nabla \times \bs{n}_1\right)^2, \notag\\
\left|\bs{n}_2 \times\left(\nabla \times \bs{n}_2\right)\right|^2&=\left(\bs{n}_1 \cdot \nabla \times \bs{n}_2\right)^2+\left(\bs{n}_3 \cdot \nabla \times \bs{n}_2\right)^2, \notag\\
\left|\bs{n}_3 \times\left(\nabla \times \bs{n}_3\right)\right|^2&=\left(\bs{n}_1 \cdot \nabla \times \bs{n}_3\right)^2+\left(\bs{n}_2 \cdot \nabla \times \bs{n}_3\right)^2, \notag\\
\left|\nabla \bs{n}_i\right|^2&=\left(\nabla \cdot \bs{n}_i\right)^2+\left(\bs{n}_i \cdot \nabla \times \bs{n}_i\right)^2+\left|\bs{n}_i \times\left(\nabla \times \bs{n}_i\right)\right|^2 \notag\\
&\quad +\nabla \cdot\big [\left(\bs{n}_i \cdot \nabla\right) \bs{n}_i-\left(\nabla \cdot \bs{n}_i\right) \bs{n}_i\big], \quad i=1,2,3.
\end{align}
Inserting the above identities into equation \eqref{eq: density}, we arrive at the reformulated form \eqref{eq: density2}.
\end{proof}

With the help of the reformulated form \eqref{eq: density2} and under periodic boundary conditions, one has the following form of energy variation
\begin{equation}\label{variation}
  \begin{aligned}
  \frac{\delta \mathcal{F}[\mathbf{p}]}{\delta \bs n_i}& =-\gamma_i \Delta \bs{n}_i-k_i \nabla (\nabla \cdot \bs{n}_i)+\sum_{j=1}^3 k_{j i} \nabla \times\big[\left(\bs{n}_j \cdot \nabla \times \bs{n}_i \right)\bs{n}_j\big]\\
  &\quad +\sum_{j=1}^3 k_{i j}\left(\bs{n}_i \cdot \nabla \times \bs{n}_j\right)\left(\nabla \times \bs{n}_j\right),\quad i=1,2,3.
  \end{aligned}
\end{equation}

Now, we are ready to propose the second-order discrete gradient approximation for the energy variation in equation \eqref{variation}. 
\begin{prop}
Define
\begin{equation}\label{eq:OFDG}
D^B_{\mathcal{F}}(\mathbf{p})\Big|^{n+\frac{1}{2}}=\left(D^B_{\mathcal{F}}(\mathbf{p})\Big|_1^{n+\frac{1}{2}},D^B_{\mathcal{F}}(\mathbf{p})\Big|_2^{n+\frac{1}{2}},D^B_{\mathcal{F}}(\mathbf{p})\Big|_3^{n+\frac{1}{2}}\right),
\end{equation}
where $D^B_{\mathcal{F}}(\mathbf{p})\Big|_i^{n+\frac{1}{2}}$ is given by
\begin{equation}\label{variationdis}
\begin{aligned}
D^B_{\mathcal{F}}(\mathbf{p})\Big|_i^{n+\frac{1}{2}}=&-\gamma_i \nabla \cdot( \nabla  \bs n_i^{n+1/2})-k_i \nabla ( \nabla \cdot \bs n_i^{n+1/2})+\sum_{j=1}^3 k_{j i} \nabla \times\left(\beta_{ji}^{n+1/2}\bs{n}_j^{n+\frac{1}{2}}\right)\\
&+\sum_{j=1}^3 k_{i j}\beta_{ij}^{n+1/2}\left(\nabla \times \bs{n}_j^{n+\frac{1}{2}}\right),
\end{aligned}
\end{equation}
with $\beta_{ij}^{n+1/2}$ take the form
\begin{equation}\label{eq: omega}
\beta_{ij}^{n+1/2}=\big(\bs n_{i}^{n+1}\cdot \nabla \times \bs n_{j}^{n+1}+\bs n_{i}^{n}\cdot \nabla \times \bs n_{j}^{n} \big)/2.
\end{equation}
The approximation $D^B_{\mathcal{F}}(\mathbf{p})\Big|^{n+\frac{1}{2}}$ serves as a second-order discrete gradient approximation of the biaxial elastic energy variation  $  {\delta \mathcal{F}[\mathbf{p}]}/{\delta \mathbf{p}}$ in equation \eqref{variation} and it satisfies the energy difference relation
\begin{equation*}
\int_{\Omega} D_{\mathcal{F}}(\mathbf{p})\Big|^{n+\frac{1}{2}}\cdot (\mathbf{p}^{n+1}-\mathbf{p}^n) dV=\mathcal{F}_{B_i}[ \mathbf{p}^{n+1}]-\mathcal{F}_{B_i}[ \mathbf{p}^{n}].
\end{equation*}
\end{prop}
\begin{proof}
Taking no account of boundary terms, the biaxial elasticity can be divided into three parts:
\begin{equation}
\mathcal{F}[\mathbf{p}]=\frac{1}{2}\mathcal{F}_1[\mathbf{p}]+\frac{1}{2}\mathcal{F}_2[\mathbf{p}]  +\frac{1}{2}\mathcal{F}_3[\mathbf{p}],
\end{equation}
and these three parts are
\begin{equation}
\begin{aligned}
&\mathcal{F}_1[\mathbf{p}]=\sum_{i=1}^3 \gamma_i\int_{\Omega} \left|\nabla \bs{n}_i\right|^2dV,\\
&\mathcal{F}_2[\mathbf{p}]=\sum_{i=1}^3 k_i\int_{\Omega}\left(\nabla \cdot \bs{n}_i\right)^2dV, \\
&\mathcal{F}_3[\mathbf{p}]=\sum_{i, j=1}^3 k_{i j}\int_{\Omega}\left(\bs{n}_i \cdot \nabla \times \bs{n}_j\right)^2dV.
\end{aligned}
\end{equation}
For the first term, using integration by parts and the periodic  boundary conditions, we can derive that
\begin{equation*}
\begin{aligned}
\mathcal{F}_1[\mathbf{p}^{n+1}] -\mathcal{F}_1[\mathbf{p}^n]=&\sum_{i=1}^3\gamma_i \int (|\nabla   \bs{n}_i^{n+1}|^{2}-|\nabla  \bs{n}_i^{n}|^{2} )dV\\
=&\sum_{i=1}^3\gamma_i \int (\nabla   \bs{n}_i^{n+1}+\nabla  \bs{n}_i^{n})(\nabla   \bs{n}_i^{n+1}-\nabla   \bs{n}_i^{n}) dV\\
=&2 \sum_{i=1}^3\gamma_i\int \nabla  ( \bs{n}_i^{n+1/2}) \nabla   (\bs{n}_i^{n+1}-\bs{n}_i^{n}) dV\\
=&-2\sum_{i=1}^3\gamma_i\int \nabla \cdot( \nabla  \bs n_i^{n+1/2})\cdot (\bs n_i^{n+1}-\bs n_i^n) dV\\
=&-\int 2\left(\gamma_1\nabla \cdot( \nabla  \bs n_1^{n+1/2}),\gamma_2 \nabla \cdot( \nabla  \bs n_2^{n+1/2}),\gamma_3\nabla \cdot( \nabla \bs n_3^{n+1/2}) \right)\cdot(\mathbf{p}^{n+1}-\mathbf{p}^{n})dV.
\end{aligned}
\end{equation*}
Define
\begin{equation}\label{eq: df1}
D^B_{\mathcal{F}_{1}}(\mathbf{p})\Big|^{n+\frac{1}{2}}=-2\left(\gamma_1\nabla \cdot( \nabla  \bs n_1^{n+1/2}),\gamma_2 \nabla \cdot( \nabla  \bs n_2^{n+1/2}),\gamma_3\nabla \cdot( \nabla \bs n_3^{n+1/2}) \right),
\end{equation}
and one readily observes that 
\begin{equation}\label{eq: df1rela}
\int D^B_{\mathcal{F}_{1}}(\mathbf{p})\Big|^{n+\frac{1}{2}}\cdot (\mathbf{p}^{n+1}-\mathbf{p}^n)  dV=\mathcal{F}_1[\mathbf{p}^{n+1}] -\mathcal{F}_1[\mathbf{p}^n].
\end{equation}

For the second term, using  integration by parts and the periodic  boundary condition, it is direct to calculate that
\begin{equation*}
\begin{aligned}
\mathcal{F}_2[\mathbf{p}^{n+1}] -\mathcal{F}_2[\mathbf{p}^n]=&\sum_{i=1}^3 k_i \int (|\nabla \cdot  \bs{n}_i^{n+1}|^{2}-|\nabla \cdot \bs{n}_i^{n}|^{2} )dV \notag\\
=&\sum_{i=1}^3 k_i \int (\nabla \cdot  \bs{n}_i^{n+1}+\nabla \cdot  \bs{n}_i^{n})(\nabla \cdot  \bs{n}_i^{n+1}-\nabla \cdot  \bs{n}_i^{n}) dV \notag\\
=&2 \sum_{i=1}^3 k_i\int \nabla \cdot \big( \bs{n}_i^{n+1/2} \big) \nabla \cdot  \big(\bs{n}_i^{n+1}-\bs{n}_i^{n} \big) dV \notag\\
=&-2\sum_{i=1}^3 k_i\int \nabla \big( \nabla \cdot \bs n_i^{n+1/2} \big)\cdot  \big(\bs n_i^{n+1}-\bs n_i^n \big) dV \notag\\
=&\int -2\left[ k_1\nabla \big( \nabla \cdot \bs n_1^{n+1/2} \big), k_2\nabla \big( \nabla \cdot \bs n_2^{n+1/2} \big),k_3\nabla \big( \nabla \cdot \bs n_3^{n+1/2} \big) \right] \cdot(\mathbf{p}^{n+1}-\mathbf{p}^{n})dV.
\end{aligned}
\end{equation*}
Hence, we can define  
\begin{equation}\label{eq: df2}
D^B_{\mathcal{F}_{2}}(\mathbf{p})\Big|^{n+\frac{1}{2}}=-2\left(k_1\nabla \big( \nabla \cdot \bs n_1^{n+1/2} \big), k_2\nabla \big( \nabla \cdot \bs n_2^{n+1/2}\big),k_3 \nabla \big( \nabla \cdot \bs n_3^{n+1/2}\big) \right)
\end{equation}
to obtain 
\begin{equation}\label{eq: df2rela}
\int D^B_{\mathcal{F}_{2}}(\mathbf{p})\Big|^{n+\frac{1}{2}}\cdot (\mathbf{p}^{n+1}-\mathbf{p}^n)  dV=\mathcal{F}_2[\mathbf{p}^{n+1}] -\mathcal{F}_2[\mathbf{p}^n].
\end{equation}

For the third term, we need to resort to the following identity
\begin{equation*}
\begin{aligned}
&\big|\bs{n}_i^{n+1}\cdot \nabla \times \bs{n}_j^{n+1} \big|^{2}- \big|\bs{n}_i^{n}\cdot \nabla \times \bs{n}_j^{n} \big|^{2}\\
&=2\big((\bs{n}_i^{n+1}-\bs{n}_i^{n})\cdot \nabla \times \bs{n}_j^{n+1/2} +\bs{n}_i^{n+1/2}\cdot \nabla\times (\bs{n}_j^{n+1}-\bs{n}_j^{n})\big)\beta_{ij}^{n+1/2},
\end{aligned}
\end{equation*}
where the definition of $\beta_{ij}$ in equation \eqref{eq: omega} has been used. 
Thus, we can follow a similar manner as the first two terms to obtain
\begin{equation}\label{eq: df3dis}
\begin{aligned}
\mathcal{F}_3[\mathbf{p}^{n+1}] -\mathcal{F}_3[\mathbf{p}^n]=&2 \sum_{i, j=1}^3 k_{ij} \int \Big(\beta_{ij}^{n+1/2}  \nabla \times \bs n_j^{n+1/2} \Big) \cdot (\bs n_i^{n+1}-\bs n_i^n) \\
&+\big(\nabla \times (\beta_{ij}^{n+1/2}  \bs n_{i}^{n+1/2})\big)\cdot (\bs n_j^{n+1}-\bs n_j^n) dV.
\end{aligned}
\end{equation}
As a consequence, we can define the discrete gradient approximation for the last term  as
\begin{equation}\label{eq: df3}
  \begin{aligned}
 D_{\mathcal{F}_3}(\mathbf{p})\Big|^{n+1 / 2}=2 &\Big(  \sum_{j=1}^3 k_{j 1} \nabla \times\big(\beta_{j 1}^{n+1 / 2} \bs n_j^{n+\frac{1}{2}}\big)+\sum_{j=1}^3 k_{1 j} \beta_{1 j}^{n+1 / 2}\big(\nabla \times \bs n_j^{n+\frac{1}{2}}\big), \\
	  & \sum_{j=1}^3 k_{j 2} \nabla \times\big(\beta_{j 2}^{n+1 / 2} \bs n_j^{n+\frac{1}{2}}\big)+\sum_{j=1}^3 k_{2 j} \beta_{2 j}^{n+1 / 2}\big(\nabla \times \bs n_j^{n+\frac{1}{2}}\big), \\
	  & \sum_{j=1}^3 k_{j 3} \nabla \times\big(\beta_{j 3}^{n+1 / 2} \bs n_j^{n+\frac{1}{2}}\big)+\sum_{j=1}^3 k_{3 j} \beta_{3 j}^{n+1 / 2}\big(\nabla \times \bs  n_j^{n+\frac{1}{2}}\big)\Big),
  \end{aligned}
\end{equation}
which satisfies the energy difference relation
\begin{equation}\label{eq: df3rela}
  \int D^B_{\mathcal{F}_{3}}(\mathbf{p})\Big|^{n+\frac{1}{2}}\cdot (\mathbf{p}^{n+1}-\mathbf{p}^n)  dV=\mathcal{F}_3[\mathbf{p}^{n+1}] -\mathcal{F}_3[\mathbf{p}^n].
\end{equation}
Combining the equations \eqref{eq: df1rela}, \eqref{eq: df2rela} and \eqref{eq: df3rela} we can arrive at the discrete energy relation \eqref{eq: DDprop}

Moreover, by comparing the form of the discrete gradient approximation \eqref{variationdis} and the continuous energy variation \eqref{variation}, one can directly see that the terms except $\beta_{ij}^{n+1/2}$ in \eqref{variationdis}  are second order central difference  approximations of the corresponding terms in \eqref{variation}. Nevertheless, we can directly calculate that 
\begin{equation*}
\begin{aligned}
2\beta_{ij}^{n+1/2}=&\big(\bs{n}_i(t_{n+1 / 2})+\partial_t \bs{n}_i(t_{n+1 / 2})\Delta t+O({\Delta t}^2) \big) \cdot \nabla \times \big(\bs{n}_j(t_{n+1 / 2})+\partial_t \bs{n}_j(t_{n+1 / 2})\Delta t+O({\Delta t}^2) \big)\\
+&\big(\bs{n}_i(t_{n+1 / 2})-\partial_t \bs{n}_i(t_{n+1 / 2})\Delta t+O({\Delta t}^2) \big) \cdot \nabla \times \big(\bs{n}_j(t_{n+1 / 2})-\partial_t \bs{n}_j(t_{n+1 / 2})\Delta t+O({\Delta t}^2) \big)\\
=&\bs{n}_i(t_{n+1 / 2}) \cdot \nabla \times \bs{n}_j(t_{n+1 / 2})+O({\Delta t}^2).
\end{aligned}
\end{equation*}
Thus, we conclude that $D^B_{\mathcal{F}}(\mathbf{p})\Big|^{n+\frac{1}{2}}$ is  a second-order discrete gradient approximation of the energy variation  $  {\delta \mathcal{F}[\mathbf{p}]}/{\delta \mathbf{p}}$, which ends the proof.
\end{proof}
\begin{rem}
  Let us now explain how to deal with non-periodic boundary conditions. Let $\mathbf{p}$ perturb along $\delta \mathbf{p}$.
  Still, we use the form \eqref{eq: density2} as an example. 
  We can obtain the following variational form:
\begin{equation*}
\begin{aligned}
 & \lim_{s \rightarrow 0} \frac{\mathcal{F}[\mathbf{p}+s \delta \mathbf{p}]-\mathcal{F}[\mathbf{p}]}{s}=\sum_{i=1}^3 \gamma_i\int_{\Omega} (\nabla \bs n_i)\cdot(\nabla \delta \bs n_i)dV + \sum_{i=1}^3 k_i\int_{\Omega} (\nabla \cdot \bs n_i)(\nabla \cdot \delta \bs n_i)dV\\
&+\sum_{i,j=1}^3 k_{ij}\int_{\Omega} (\bs n_i \cdot \nabla \times \bs n_j)(\delta \bs n_i \cdot \nabla \times \bs n_j + \bs n_i \cdot  \nabla \times \delta\bs n_j)dV\\
=&-\sum_{i=1}^3 \gamma_i\int_{\Omega} (\Delta \bs n_i)\cdot (\delta \bs n_i)dV - \sum_{i=1}^3 k_i\int_{\Omega} (\nabla (\nabla \cdot \bs n_i))\cdot (\delta \bs n_i)dV\\
&+\sum_{i,j=1}^3 k_{ij}\int_{\Omega} (\bs n_i \cdot \nabla \bs n_j)\nabla \times \bs n_j \cdot \delta \bs n_i dV+\sum_{i,j=1}^3 k_{ji}\int_{\Omega} \nabla \times\left(\left(\bs{n}_j \cdot \nabla \times \bs{n}_i\right) \bs{n}_j\right) \cdot \delta \bs n_i  \,dV \\
&+\sum_{i=1}^3 \gamma_i\int_{\partial \Omega} (\nabla \bs n_i)^{\intercal} \bs \nu \cdot (\delta \bs n_i)\,dS + \sum_{i=1}^3 k_i\int_{\partial \Omega} (\nabla \cdot \bs n_i) \bs \nu  \cdot \delta \bs n_i \,dS\\
&+\sum_{i,j=1}^3 k_{ji}\int_{\partial \Omega} (\bs n_j \cdot \nabla \times \bs n_i)\bs n_j \times \bs \nu \cdot \delta \bs n_i \,dS,
\end{aligned}
\end{equation*}
where $\bs \nu$ is the outward unit normal vector.  
We recognize that the volume integrals yield the same terms as in the case of periodic boundary conditions, while the surface integrals might give extra conditions. 
When the Dirichlet boundary condition is adopted, we have $\delta \bs n_i=0$ on $\partial \Omega$, we can directly to observe that the surface integral vanishes. Since $\delta \bs n_i$ is perpendicular to $\bs n_i$ for $i=1, 2, 3$, the Neumann boundary condition shall be given by
\begin{equation}\label{eq:bound3}
\bs n_i \times \big[\gamma_i (\nabla \bs n_i)^{\intercal}\bs \nu +k_i \nabla  \cdot \bs  n_i+\sum_{j=1}^3 k_{ji}(\bs n_j \cdot \nabla \times \bs n_i)\bs n_j \times \bs \nu \big]=0, \quad i=1, 2, 3.
\end{equation}
We now need to equip the discrete gradient with appropriate approximation of the boundary condition, which turns out to be 
\begin{equation}\label{eq:bound3dis}
\bs n_i^{n+\frac{1}{2}} \times [\gamma_i (\nabla \bs n_i^{n+\frac{1}{2}})^{\intercal}\bs \nu +k_i \nabla  \cdot \bs n_i^{ n+\frac{1}{2}}+\sum_{j=1}^3 k_{ji} \beta_{ji}^{ n+\frac{1}{2}}\bs n_j^{n+\frac{1}{2}} \times \bs \nu ]=0, \quad i=1, 2, 3.
\end{equation}
If additional boundary energy is present, we use the same way to get the suitable discretization. 
\end{rem}

\begin{rem}
Besides the proposed discrete gradient approximation, other discrete gradient approximations  that satisfy the energy difference relation \eqref{eq: DDprop} can be adopted, e.g. the mean-value discrete gradient \cite{celledoni2012preserving,harten1983upstream}:
\begin{equation}\label{eq:mv}
D^M_{\mathcal{F}}(\mathbf{p})\Big|^{n+\frac{1}{2}}=\int_0^1 \frac{\delta \mathcal{F}_{B_i}}{\delta \mathbf{p}}\big[ (1-s)\mathbf{p}^{n+1}+s\mathbf{p}^n \big]ds,
\end{equation}
and the Gonzalez discrete gradient \cite{gonzalez2000time}:
\begin{equation}\label{eq: Gondg}
\begin{aligned}
D^G_{\mathcal{F}}(\mathbf{p})\Big|^{n+\frac{1}{2}}=&\frac{\delta \mathcal{F}_{B_i}}{\delta \mathbf{p}}[\mathbf{p}^{n+\frac{1}{2}}]\\
&+\frac{ \mathcal{F}_{B_i}[\mathbf{p}^{n+1}]- \mathcal{F}_{B_i}[\mathbf{p}^{n}]-\int_{\Omega} \frac{\delta \mathcal{F}_{B_i}}{\delta \mathbf{p}}[\mathbf{p}^{n+\frac{1}{2}}] \cdot (\mathbf{p}^{n+1}-\mathbf{p}^n) dV  }{\int_{\Omega} \big(\mathbf{p}^{n+1}-\mathbf{p}^n\big) \cdot \big(\mathbf{p}^{n+1}-\mathbf{p}^n \big)dV }(\mathbf{p}^{n+1}-\mathbf{p}^n).
\end{aligned}
\end{equation}
However, these two approximations have their drawbacks. The implementation of the mean-value discrete gradient relies on numerical integration, while it preserves the discrete energy dissipation law up to the numerical quadrature error.  Thus, it results in a substantial increase in the computational cost in each iteration when we solve the nonlinear scheme. As for the Gonzalez discrete gradient, it is sensitive to roundoff errors and suffers from convergence issue when the gradient flow system tends to equilibrium (when the denominator $\int_{\Omega} \big(\mathbf{p}^{n+1}-\mathbf{p}^n\big) \cdot \big(\mathbf{p}^{n+1}-\mathbf{p}^n \big)dV$ approaches $0$).   
\end{rem}

\subsection{Time adaptivity strategy}\label{sect: timeadap}
The proposed discrete scheme is nonlinear and one needs to adopt nonlinear solvers such as inexact Newton-Krylov (INK) method with proper line search technique to solve it efficiently (see more details in \cite{kelley1995iterative}). It is worthwhile to note that due to the sensitivity of convergence of nonlinear solver to initial guess, suitable choice of time step can overcome this issue and improve the accuracy and efficiency of the computation. Moreover, as is shown in numerical experiments in the next section, the evolution of liquid crystal driven by biaxial elastic energy may exhibit multiple stages. It may evolve dramatically in relatively short time and keep stable or evolve slowly for the rest of the time. Thus, it is necessary to adopt appropriate time-adaptivity strategy to make efficient and accurate simulations. Inspired by the method proposed in \cite{qiao2011adaptive}, we adopt the following time-adaptive strategy
\begin{equation}\label{eq: adap}
\tau_{n+1}={\rm max}\bigg(\tau_{\rm min},  \frac{\tau_{\rm max}}{ \sqrt{1+\alpha| (\mathcal{F}[\mathbf{p}^{n}]-\mathcal{F}[\mathbf{p}^{n-1}])/\tau_{n}        |^2 }}   \bigg),
\end{equation} 
where $\tau_{n}$ is the $n$-th time step, $\alpha$ is a constant chosen to adjust the speed of change of the step size, and $\tau_{\rm max}$ and $\tau_{\rm min}$ are the upper and lower bounds of the time step sizes, respectively.

\section{Representative numerical experiments}\label{sect: num} 
In this section, we conduct several numerical experiments to demonstrate the accuracy and efficiency of the proposed gRdg method for orthonormal frame gradient flow system, as well as its ability to preserve the orthonormality constraint and energy dissipation law at the discrete level.
The periodic boundary conditions are assumed, discretized using the Fourier spectral method \eqref{eq: gfmodel}.
An inexact Newton-Krylov (INK) solver is adopted to solve the resultant nonlinear equation with the stopping tolerance $10^{-8}$.
We first present the spatial and temporal convergence rates using manufactured solutions. Then for a prescribed initial profile of the frame field, we examine the discrete orthonormality error and energy curve as functions of time.
Moreover, we compare the simulation results obtained by biaxial orthonormal frame and uniaxial Oseen-Frank gradient flows for special choice of elastic coefficients such that the biaxial elasticity reduces to the uniaxial Oseen-Frank energy.
Last, we investigate the dynamics of frame field with highly anisotropic elastic coefficients derived from molecular parameters.
In what follows, let us denote the elastic coefficients $\hat{\bs K}=\left(K_1, K_2, \cdots, K_{12}\right)^{\intercal}$ for notational convenience. And without loss of generality, we take $\chi_1=\chi_2=\chi_3=2$.

\subsection{Convergence tests} The aim of this subsection is to show the convergence rates of the method developed herein using a manufactured solution. The computational domain is fixed with $\Omega=\left[0,2\pi\right]^3.$ We assume the following analytic expression for the manufactured solution $\mathbf{p}(\bs x,t)=(\bs n_1, \bs n_2, \bs n_3)$ of the orthonormal frame gradient flow system \eqref{eq: gfmodel}
\begin{equation}\label{eq:conver}
  \begin{aligned}
	  n_{11}(\bs x,t)=&\sin{\big(\sin{(x_1+t)}\cos{(x_2)}\sin{(x_3)}\big)}\cos{\big(\cos{(x_1)}\sin{(x_2+t)}\cos{(x_3)}\big)},\\
	  n_{12}(\bs x,t)=&\sin{\big(\sin{(x_1+t)}\cos{(x_2)}\sin{(x_3)}\big)}\sin{\big(\cos{(x_1)}\sin{(x_2+t)}\cos{(x_3)}\big)},\\
	  n_{13}(\bs x,t)=&\cos{\big(\sin{(x_1+t)}\cos{(x_2)}\sin{(x_3)}\big)},\\
	  n_{21}(\bs x,t)=&\cos{\big(\sin{(x_1+t)}\cos{(x_2)}\sin{(x_3)}\big)}\cos{\big(\cos{(x_1)}\sin{(x_2+t)}\cos{(x_3)}\big)},\\
	  n_{22}(\bs x,t)=&\cos{\big(\sin{(x_1+t)}\cos{(x_2)}\sin{(x_3)}\big)}\sin{\big(\cos{(x_1)}\sin{(x_2+t)}\cos{(x_3)}\big)},\\
	  n_{23}(\bs x,t)=&-\sin{\big(\sin{(x_1+t)}\cos{(x_2)}\sin{(x_3)}\big)},\\
	  n_{31}(\bs x,t)=&-\sin{\big(\cos{(x_1)}\sin{(x_2+t)}\cos{(x_3)}\big)},\\
	  n_{32}(\bs x,t)=&\cos{\big(\cos{(x_1)}\sin{(x_2+t)}\cos{(x_3)}\big)},\quad n_{33}(\bs x,t)=0,\\
  \end{aligned}
\end{equation}
such that $\mathbf{p}$ is orthonormal and periodic in $\Omega$. In order for the above analytic expression to satisfy system \eqref{eq: rotform}, a forcing term $\mathbf{f}(\bs x,t)$ is added to \eqref{eq: rotform} such that
\begin{equation}\label{eq: forceterm}
\mathbf{f}=\mathbf{p}_t-\mathbf{p}A.
\end{equation}
We prescribe the initial condition to be $\mathbf{p}_0(\bs x):=\mathbf{p}(\bs x,t=0)$ and choose the elastic coefficients to be $\hat{\bs K}=(1,0.01,0.01,1,0.01,0.01,1,$ $0.01,0.01,1,0.01,0.01)^{\intercal}$.

Let us start with the spatial convergence test. We fix the time step size $\tau=10^{-4}$ and increase the number of Fourier collocation points in one dimension $N$ from 6 to 34. In Figure \ref{figs: spatialtest}, we depict the $L^{\infty}$-errors of $\bs n_1$,$\bs n_2$, $\bs n_3$ at $t=0.2$ as functions of $N$. It can be observed that when $N$ is below $30$,  with increased $N$, the errors decrease exponentially. While for $N\geq 16$, the errors curves level off at $10^{-8}$, showing a saturation due to the temporal discretization error and the tolerance of INK solver.

Next, we proceed to the temporal convergence test with a fixed $N=40$. The time step size $\tau$ is decreased by half  from 0.1 to 0.003125 and the $L^{\infty}$-errors of $\bs n_1$, $\bs n_2$, $\bs n_3$ are recorded at $t=0.2$. Figure \ref{figs: temporaltest} shows the numerical errors as functions of $\tau$. It can be seen clearly that a second-order convergence rate in time is achieved by using the gRdg scheme.

\begin{figure}[tbp]
\begin{center}
	\subfigure[$L^{\infty}$-error of $\bs n_{1}$]{ \includegraphics[scale=.36]{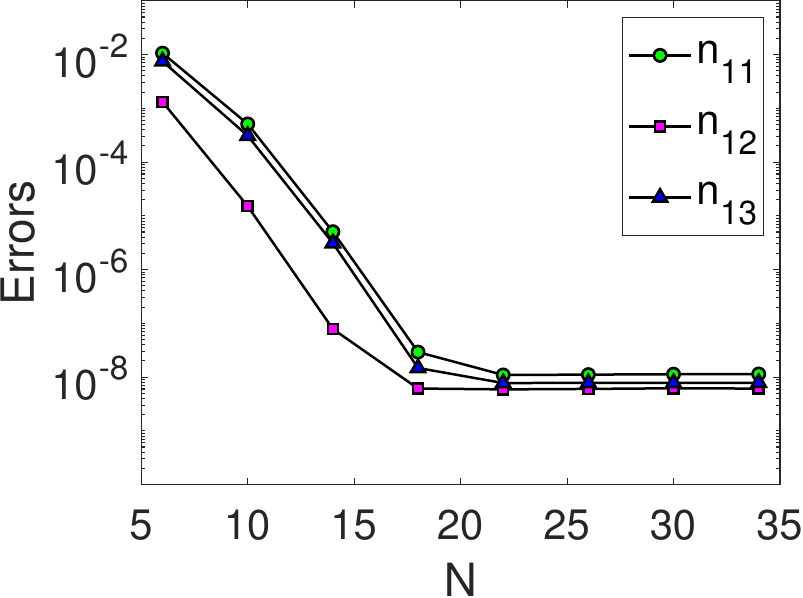}}
	\subfigure[$L^{\infty}$-error of $\bs n_{2}$ ]{ \includegraphics[scale=.36]{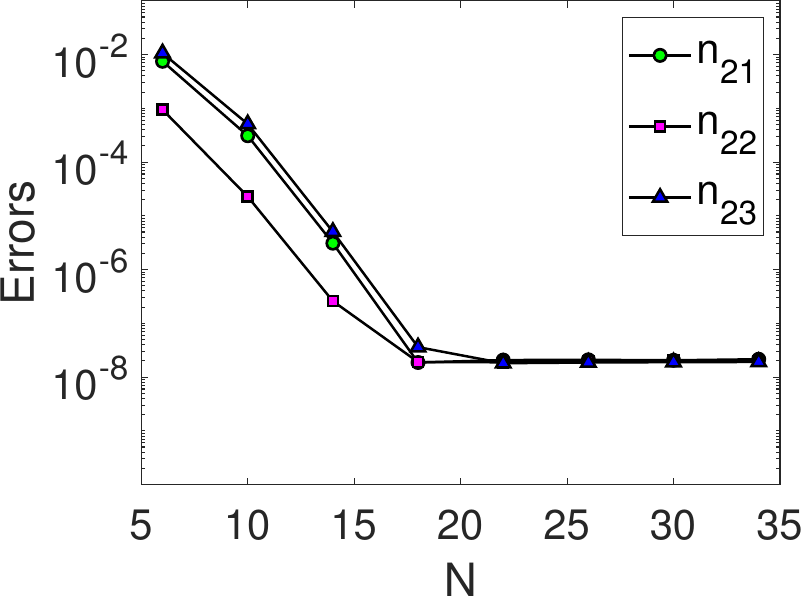}}
	\subfigure[ $L^{\infty}$-error of $\bs n_{3}$]{ \includegraphics[scale=.36]{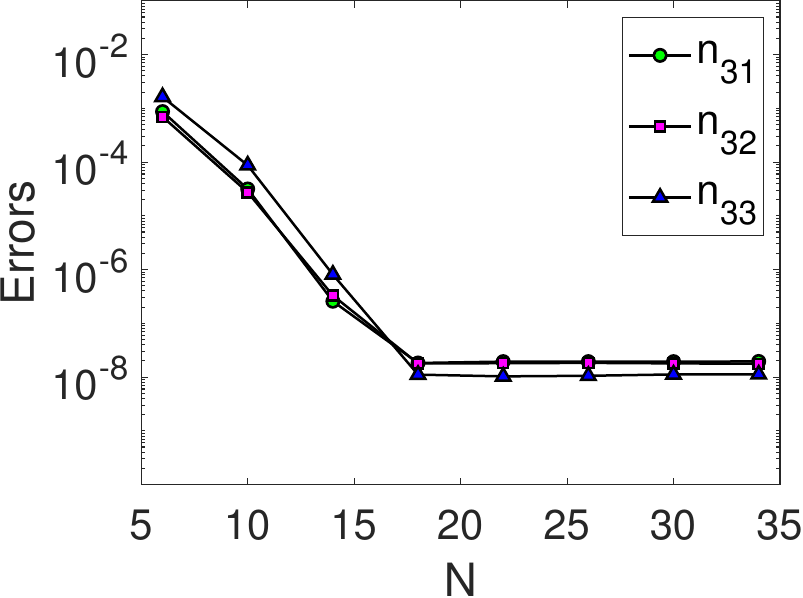}} 
   \caption{\small Spatial convergence test: $L^{\infty}$-errors of the gRdg method as a function of the number of Fourier collocation points $N$ in each dimension  for  (a) $\bs n_{1}$,  (b) $\bs n_{2}$, (c) $\bs n_{3}$.} 
	 \label{figs: spatialtest}
\end{center}
\end{figure}

\begin{figure}[tbp]
  \begin{center}
	  \subfigure[$L^{\infty}$-error of $\bs n_{1}$]{ \includegraphics[scale=.37]{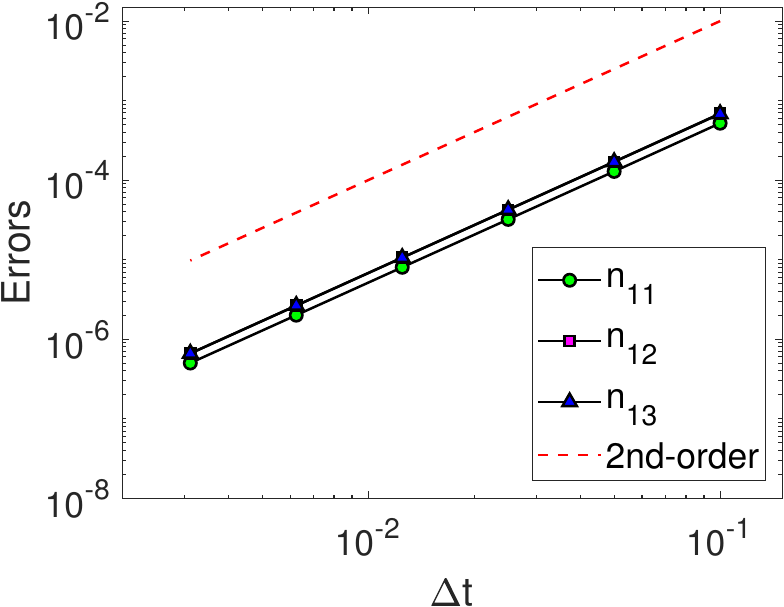}}
	  \subfigure[$L^{\infty}$-error of $\bs n_{2}$ ]{ \includegraphics[scale=.37]{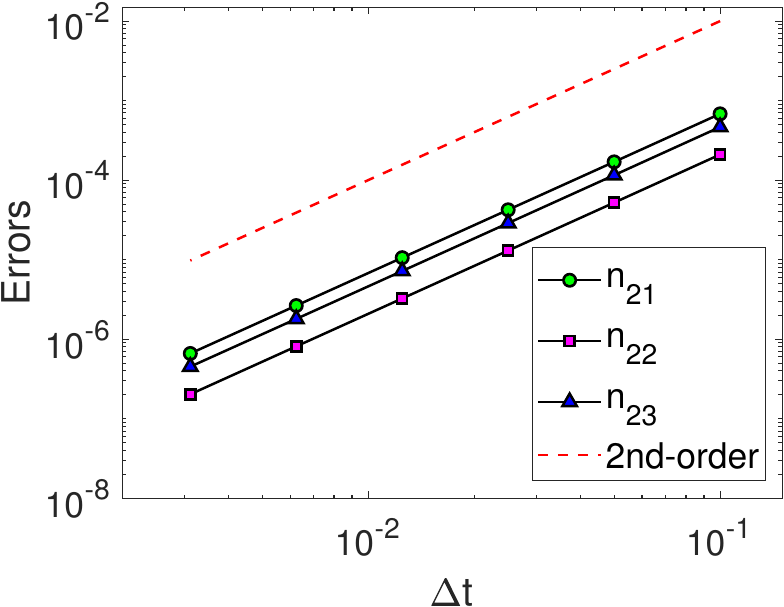}}
	  \subfigure[ $L^{\infty}$-error of $\bs n_{3}$]{ \includegraphics[scale=.37]{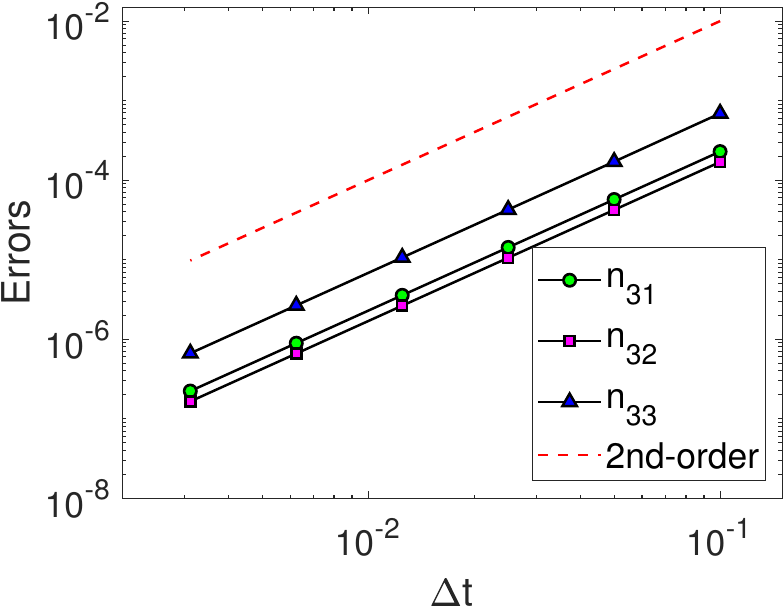}} 
  \caption{\small Temporal convergence test: $L^{\infty}$-errors of the gRdg method as a function of the time step size $\tau$ for (a) $\bs n_{1}$,  (b) $\bs n_{2}$, (c) $\bs n_{3}$. }
	 \label{figs: temporaltest}
  \end{center}
  \end{figure}

\subsection{Property test}\label{sect: ProTest}
In this subsection, we present several numerical simulations of the evolution of the frame field governed by the orthonormal frame gradient flow system to demonstrate that the proposed gRdg scheme preserves the orthonormality constraint and energy dissipation law at the discrete level.  For the sake of clear visualization, let us represent the frame field in each physical point by colored bricks. In Figure \ref{figs:direction} (a)-(b), we depict the front and back view of the colored brick, with the longest, middle and the shortest sides represent $\bs n_1$, $\bs n_2$, $\bs n_3$, respectively.
To further distinguish these three directions,  every corners of the brick are prescribed with different colors, see Figure \ref{figs:direction}.
We consider the computational domain to be a  cube $\Omega = \left[-1,1\right]^3$, and assume that the frame field is homogeneous along the $x_3$ direction.
The number of Fourier modes is taken to be $40$ in $x_1$ and $x_2$ directions, while that takes 1 for $x_3$ direction.
The time-adaptive strategy proposed in Subsection \ref{sect: timeadap} is employed to examine its effectiveness for accelerating the numerical simulations.
Hereafter, we fix the parameters $\tau_{\rm max}=2 \times 10^{-3}$, $\tau_{\rm min}=10^{-5}$ and $\alpha=10^{-3}$ in the equation \eqref{eq: adap} for the adaptivity strategy. 

First, we intentionally take the elastic coefficients $\hat{\bs K}=(1, 0, 0, 1, 0, 0, 1, 0, 0, 1, 0, 0)^{\intercal}$ that are degenerate. since the elastic energy $\mathcal{F}_{B i}[\mathbf{p}]$ reduces to the classical uniaxial Oseen-Frank energy functional (see Appendix \ref{sect: app2}). 
The following initial profile of the orthonormal frame field is chosen, 
\begin{equation}\label{eq:utest1}
  \begin{aligned}
	  \bs n_{1}(\bs x,t=0)=&\left(\sin{\big(2\sin(\pi x_1) \big)}\cos{(2\pi x_2)}, \sin{\big(2\sin(\pi x_1)\big)}\sin{(2\pi x_2)}, \cos{(2\sin(\pi x_1))}\right)^{\intercal},\\
	  \bs n_{2}(\bs x,t=0)=&\left(\cos{\big(2\sin(\pi x_1) \big)}\cos{(2\pi x_2)}, \cos{\big(2\sin(\pi x_1)\big)}\sin{(2\pi x_2)}, -\sin{\big(2\sin(\pi x_1)\big)}\right)^{\intercal},\\
	  \bs n_{3}(\bs x,t=0)=&\left(-\sin{(2\pi x_2)},\cos{(2\pi x_2)},0\right)^{\intercal}. 
  \end{aligned}
\end{equation}
\begin{figure}[tbp]
  \begin{center}  
  \subfigure[Front view]{  \includegraphics[scale=.23] {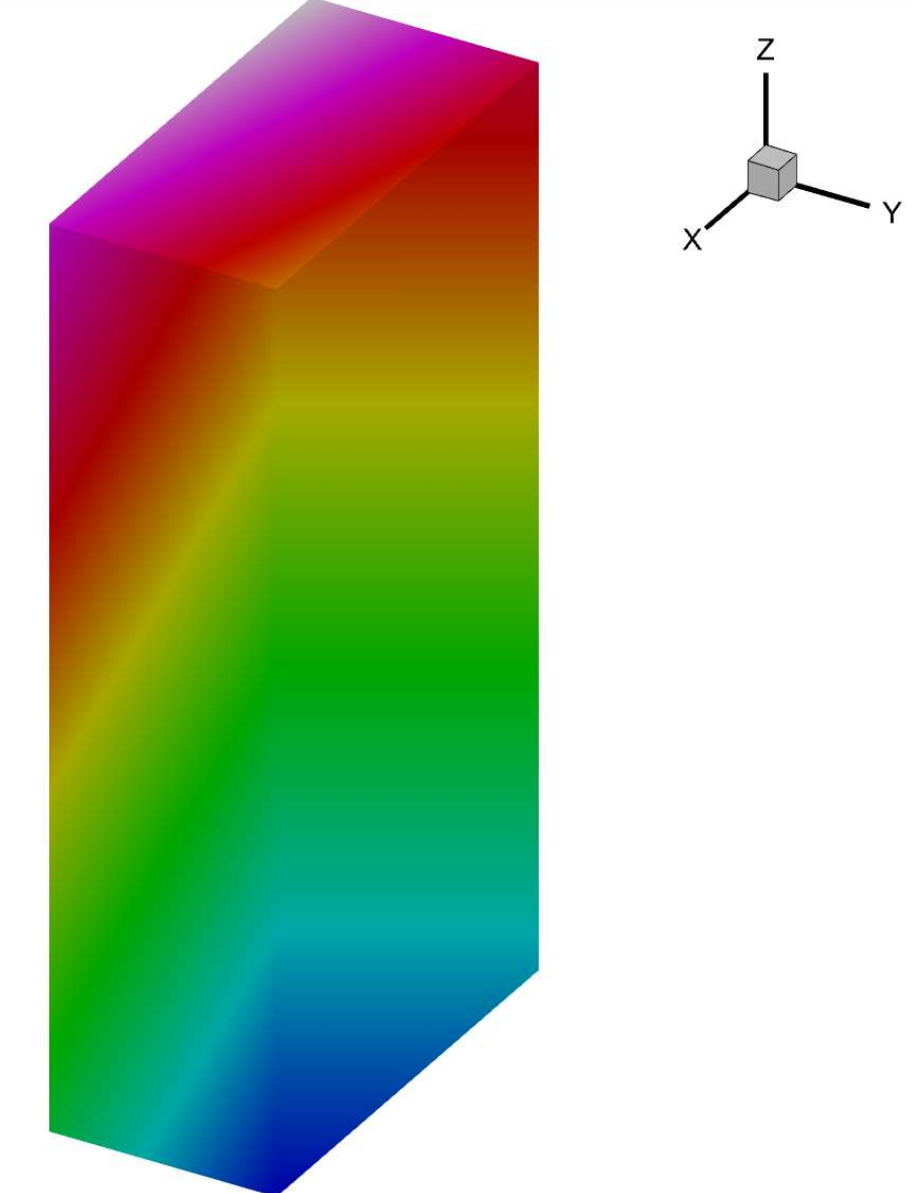}}  \qquad \quad
  \subfigure[Back view]{  \includegraphics[scale=.23] {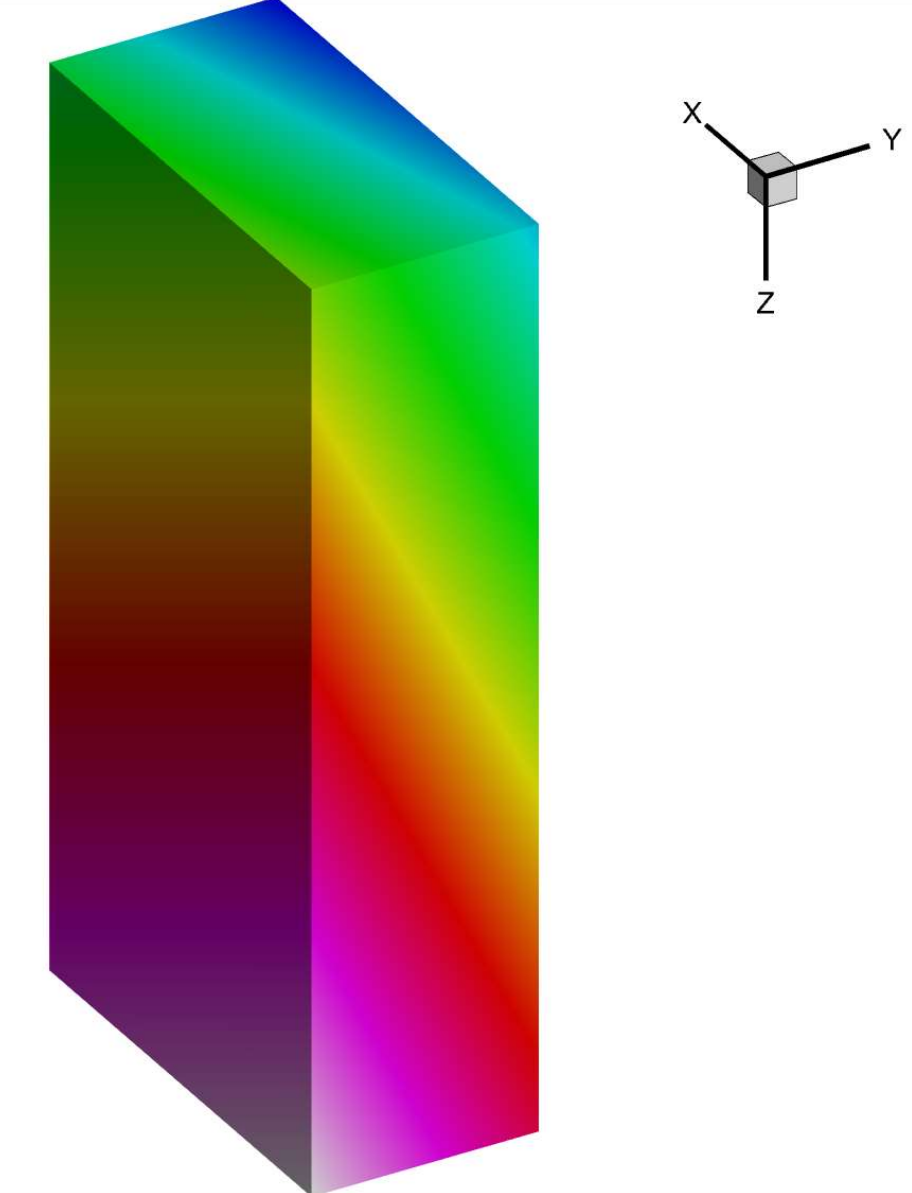}}  
  \caption{\small A colored brick representing the orthonormal frame $\mathbf{p}$}  
  \label{figs:direction}
  \end{center} 
\end{figure}
\begin{figure}[tbp]
  \begin{center}  
  \subfigure[Initial profile]{  \includegraphics[scale=.33] {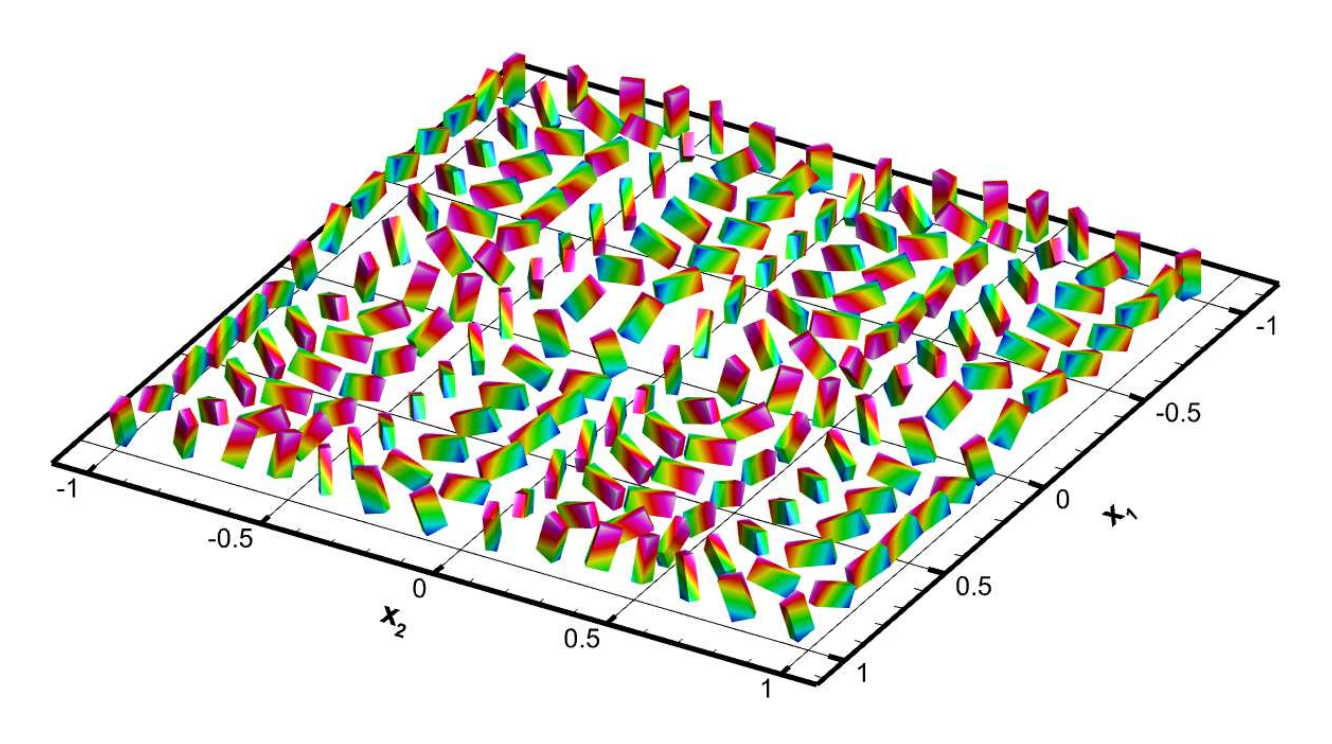}} 
 \subfigure[Final profile]{  \includegraphics[scale=.33] {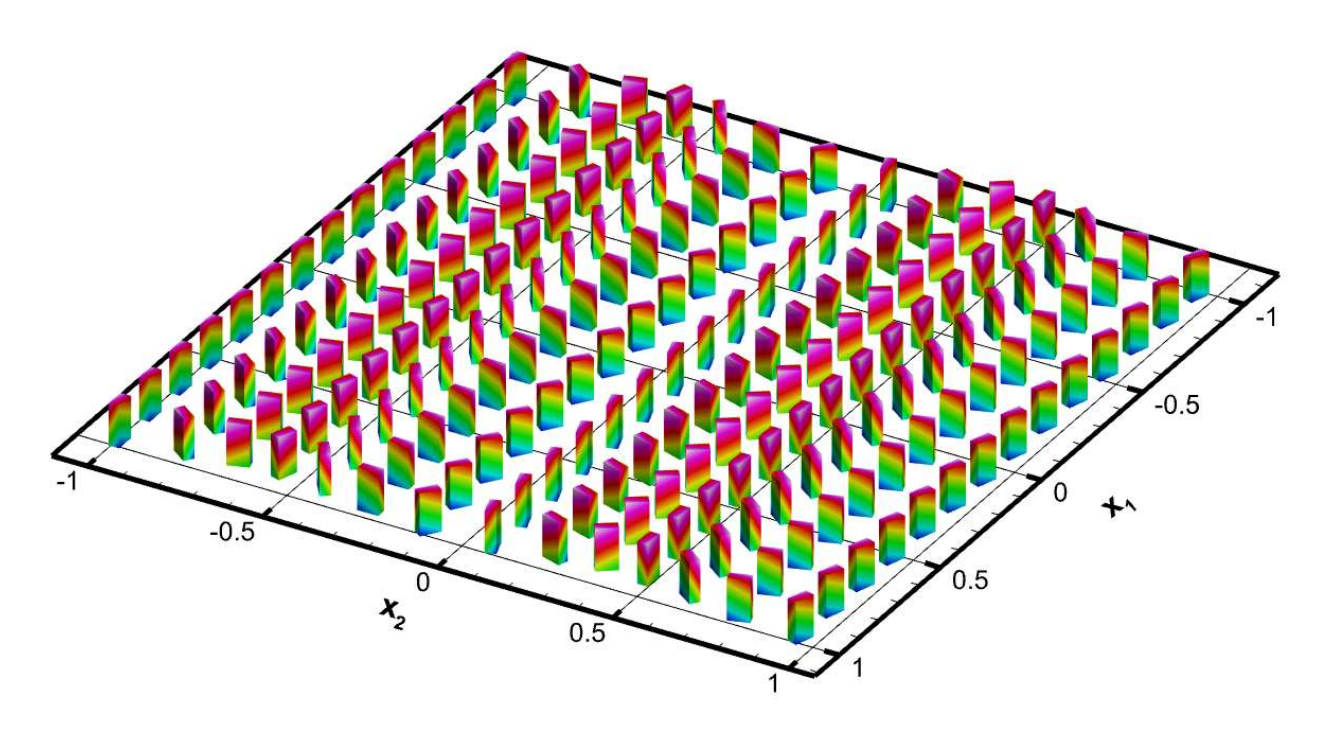}}  \\
      \subfigure[Energy vs Time]{  \includegraphics[scale=.48]{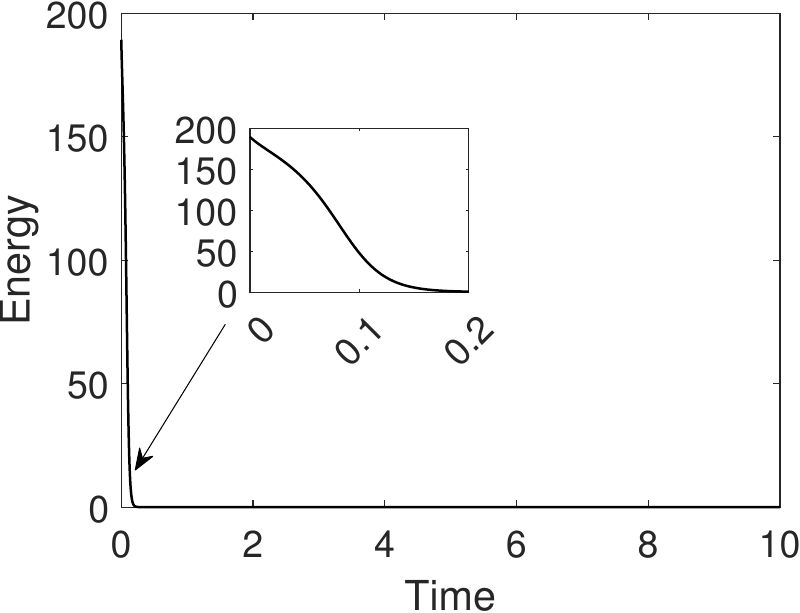}} \quad
  \subfigure[Orthonormal frame error vs Time]{  \includegraphics[scale=.48]{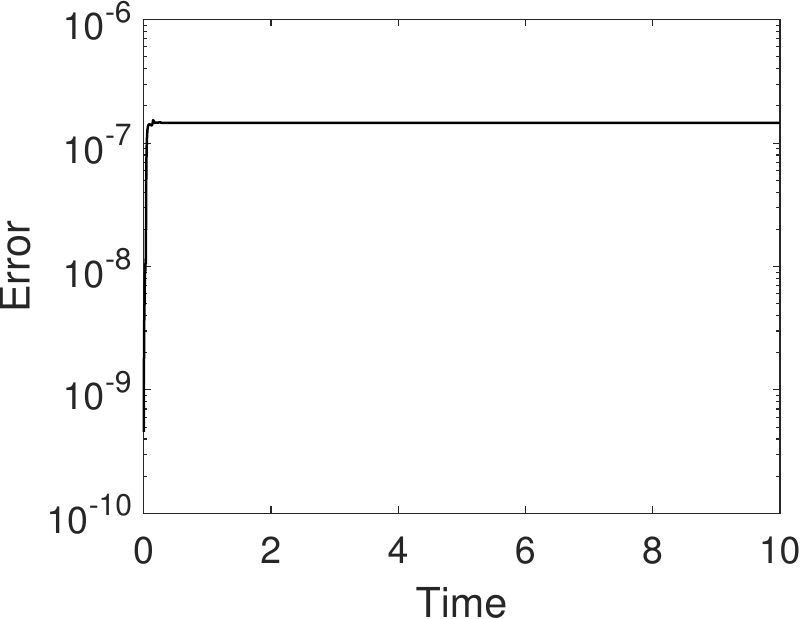}
  }\\
  \subfigure[Step size vs Time]{  \includegraphics[scale=.48]{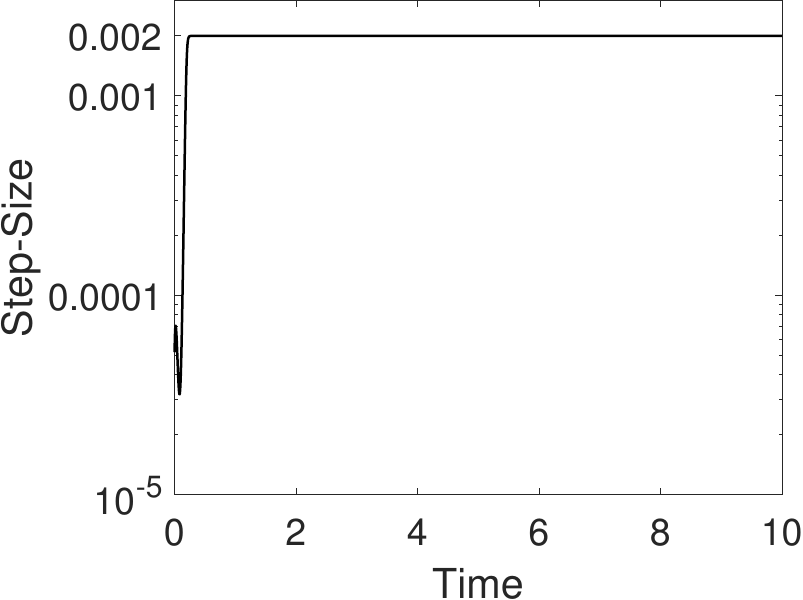}} \quad \quad \quad 
  \subfigure[Computational cost vs Time]{  \includegraphics[scale=.48]{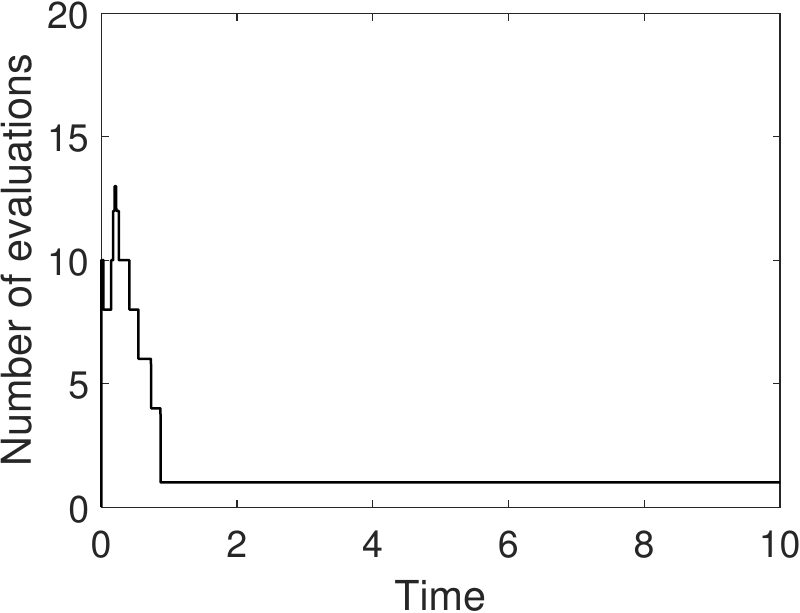}}	
  \caption{\small Evolution of the orthonormal frame field $\mathbf{p}$ with initial frame field in equation \eqref{eq:utest1} and the elastic coefficients $\hat{\bs K}=(1, 0, 0, 1, 0, 0, 1, 0, 0, 1, 0, 0)^{\intercal}$. (a) Initial frame field distribution of equation \eqref{eq:utest1}; (b) profile of the frame field at equilibrium plotted at $t=10$;  (c) time history of energy $\mathcal{F}_{B i}[\mathbf{p}]$; (d) time history of orthonormal frame error measured in $L^{\infty}$-norm; (e) time history of adaptive step size; (f) time history of the number of residual evaluations in the INK solver at each time step.}  
  \label{figs: property-test1}
  \end{center} 
\end{figure}
\begin{figure}[tbp]
  \begin{center}  
  \subfigure[Energy vs Time]{  \includegraphics[scale=.48] {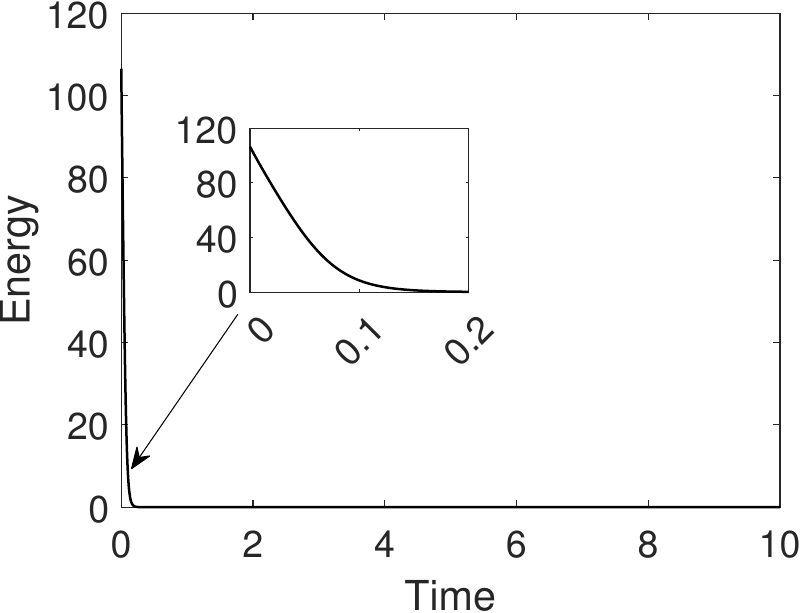}}  
  \subfigure[Final profile]{  \includegraphics[scale=.33] {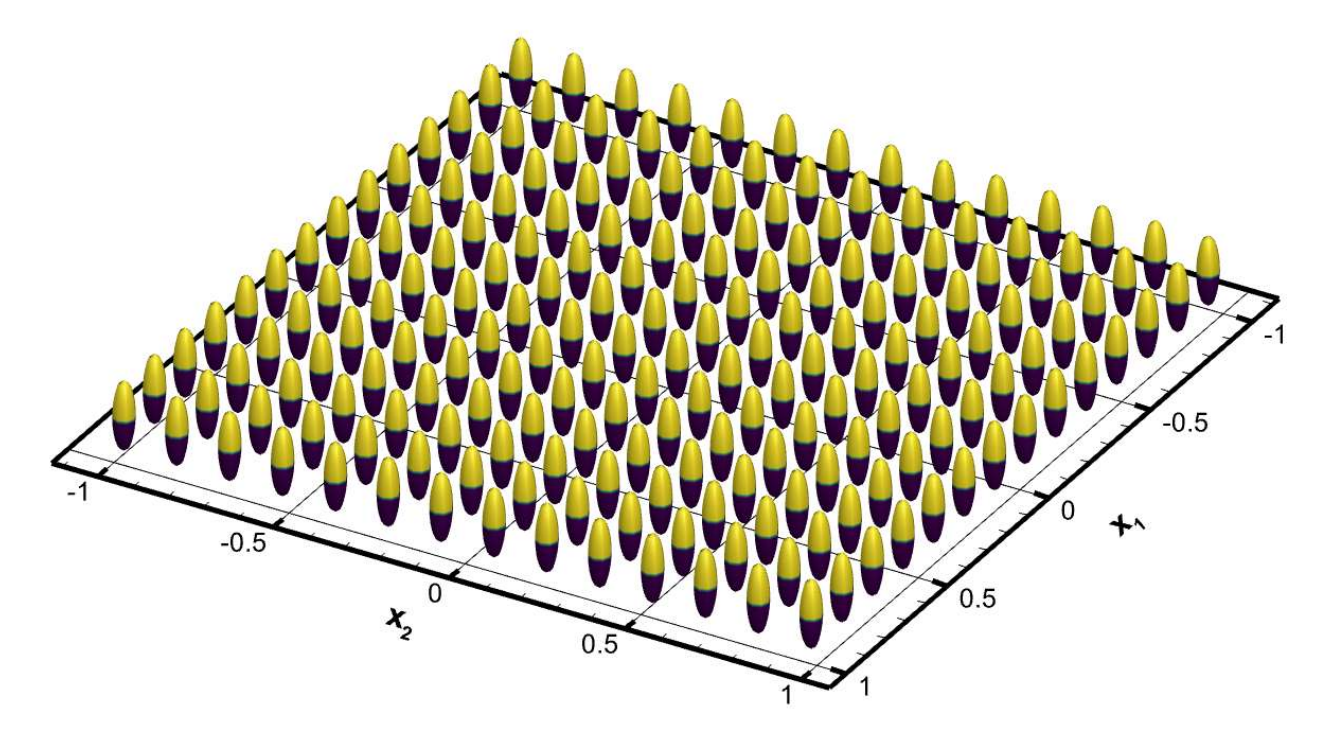}}  
  \caption{\small  Evolution of the director field with initial field distribution equals to $\bs n_1$ in equation \eqref{eq:utest1} with $K_1=K_4=K_7=K_{10}=1$ for uniaxial Oseen-Frank in equation \eqref{eq:oFrank}. (a) Time history of Oseen-Frank energy; (b) the profile of director field at equilibrium at $t=10$.}  
  \label{figs: uniaxial}
  \end{center} 
\end{figure}
Figure \ref{figs: property-test1} (a) depicts a side view (with azimuth $=120^{\circ}$, elevation $=30^{\circ}$ ) of the initial frame field distribution, visualized by colored bricks.  In Figure \ref{figs: property-test1} (c), we plot the history of energy $\mathcal{F}_{B i}[\mathbf{p}]$ as a function of time and one can observe that the energy decays rapidly from $t=0$ to $t=0.2$ and remains 0 afterwards. We plot the equilibrium state of the frame field at $t=10$ in Figure \ref{figs: property-test1} (b).  It can be seen that only the long axis representing $\bs n_1$ turn to be homogeneous, while $\bs n_2$ and $\bs n_3$ are not.
This agrees with the degenerate elastic coefficients we choose, since inhomogeneity of $\bs n_2$ and $\bs n_3$ does not lead to excess energy. 
Figure \ref{figs: property-test1} (d) plots the maximum numerical error of all entries of $\mathbf{p}\mathbf{p}^{\intercal}-\mathbf{I}$ measured in point-wise $L^\infty$-norm.
It is observed that the error curves level off at around $10^{-7}$, showing that the orthonormality constraint is preserved, up to the stopping criterion of the INK solver for solving the discrete nonlinear equation.
We depict in Figure \ref{figs: property-test1} (e) the history of time step sizes by the proposed adaptivity strategy in equation \eqref{eq: adap}.
It shows that the step sizes change adaptively in accordance with the evolution of the frame field and the time step sizes equal $\tau_{\rm max}=2\times 10^{-3}$ for most of the time steps.
Moreover, we depict in Figure \ref{figs: property-test1} (f) the time history of the number of residual evaluations at each time step to quantify the computational cost of the proposed method. The largest number of residual evaluations is less than 13, which occur between $t=0$ to $t=0.2$ to capture the rapid change of the orientation of frame field (see from Figure \ref{figs: property-test1} (c) for the rapid change in energy) and remains small for the simulations. This clearly demonstrates that when combined with the time-adaptive strategy, the nonlinear equation resulted from gRdg scheme can be solved with high efficiency. 

Since the elastic coefficients are denegerate, we would like to examine whether the above simulation result is consistent with the uniaxial Oseen-Frank gradient flow.
For comparison, we conduct the simulation for uniaxial Oseen-Frank gradient flow system with the initial distribution of the director field be $\bs n_1$ of equation \eqref{eq:utest1}.
In Figure \ref{figs: uniaxial}, we present the energy history and the equilibrium state, where we use colored rods to represent the uniaxial director. 
The energy curve is indeed identical to Figure \ref{figs: property-test1} (a), and the equilibrium director field is homogeneous along the $x_3$ direction.
Since the rotation on $\bs n_2$ and $\bs n_3$ would not contribute to the energy functional for this degenerate case, the solution to the gRdg scheme for biaxial frame model would not be unique, which poses challenges on numerical simulation. Nevertheless, it turns out that the gRdg scheme results in efficient and accurate computations without any difficulty.

In the next example, we keep the elastic coefficients $\hat{\bs K}$ but alter the initial distribution of the frame field $\mathbf{p}$ to
\begin{equation}\label{eq:utest2}
	\begin{aligned}
		\bs n_{1}(\bs x,t)=(&\sin{(\pi x_1+2\cos(\pi x_2))}, 0,\cos{(\pi x_1+2\cos(\pi x_2))})^{\intercal},  \\
		\bs n_{2}(\bs x,t)=(&\cos{(\pi x_1+2\cos(\pi x_2))}, 0, -\sin{(\pi x_1+2\cos(\pi x_2))})^{\intercal}, \\
		\bs n_{3}(\bs x,t)=(&0, 1, 0)^{\intercal}.
	\end{aligned}
\end{equation}
It can be observed from Figure \ref{figs: property-testeq3.4} that though the elastic coefficients $\hat{\bs K}$ are kept the same, a different choice of the initial distribution (see Figure \ref{figs: property-testeq3.4} (a)) of the frame field leads to a distinct evolution process of the frame field.
In Figure \ref{figs: property-testeq3.4} (b), we depict the time history of the energy curve and a multi-stage evolution character can be observed, presented in \ref{figs: property-testeq3.4} (e)-(h) for stages A to D. 
At stage A (\ref{figs: property-testeq3.4} (e)), the energy decays rapidly, and the frame field turns to be more uniform along $x_2$-direction.
Then as Figure \ref{figs: property-testeq3.4} (f) plots, the frame field becomes uniform along $x_3$-direction and rotates in the $x_1$-$x_3$ plane in stage B.
At stage C, the rotation in the $x_1$-$x_3$ plane is still on-going but the direction of $\bs n_1$ gradually tends to be homogeneous, as shown in Figure \ref{figs: property-testeq3.4} (g).
Figure \ref{figs: property-testeq3.4} (h) depicts the equilibrium stage and it shows that only the long axis representing $\bs n_1$ turn to be homogeneous,  which is similar with that of the previous example.
Moreover, we examine the efficiency of the proposed method by the time step sizes and the number of evaluations of the residual function (see Figure \ref{figs: property-testeq3.4} (c)-(d)). Again, the time step sizes remain at $\tau_{\rm max}=2\times 10^{-3}$ for most of the time steps and the number of residual evaluations are small, even between transition stages.

\begin{figure}[tbp]
  \begin{center}  
  \subfigure[Initial profile]{  \includegraphics[scale=.33] {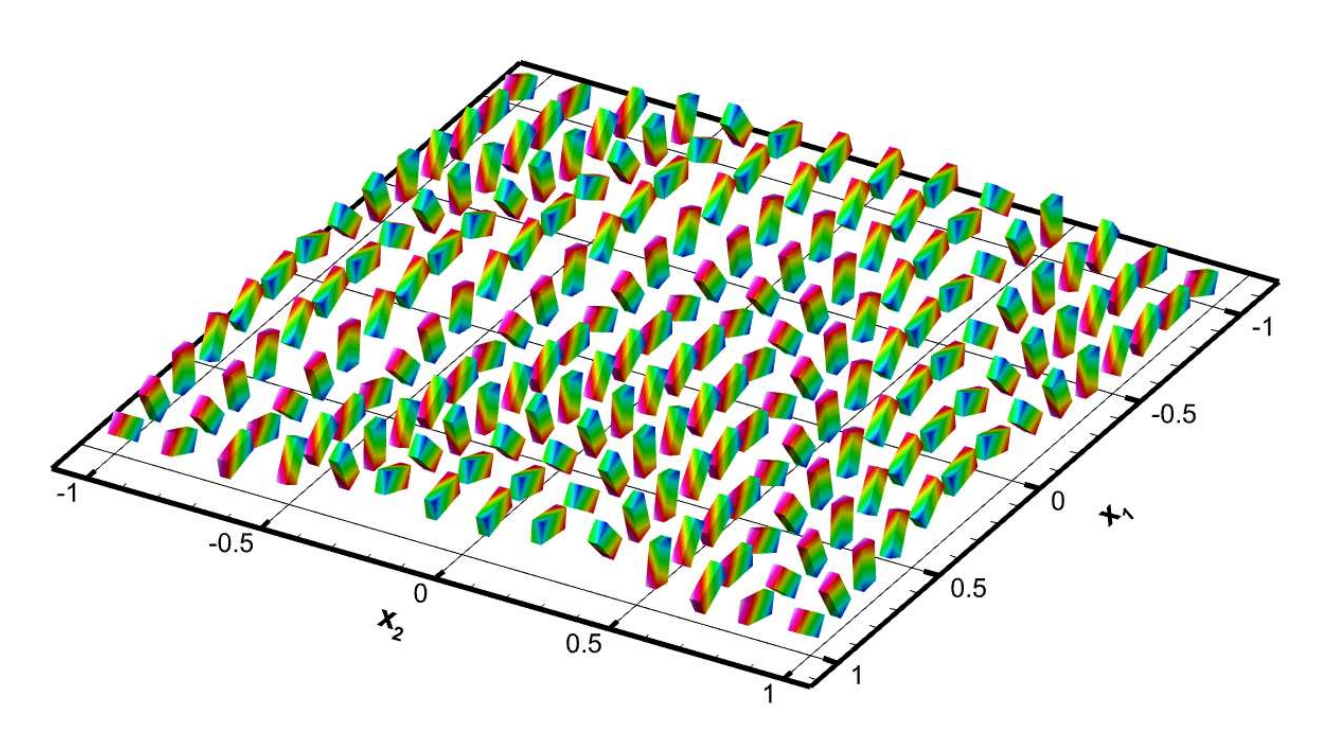}}  
  \subfigure[Energy vs Time]{  \includegraphics[scale=.42]{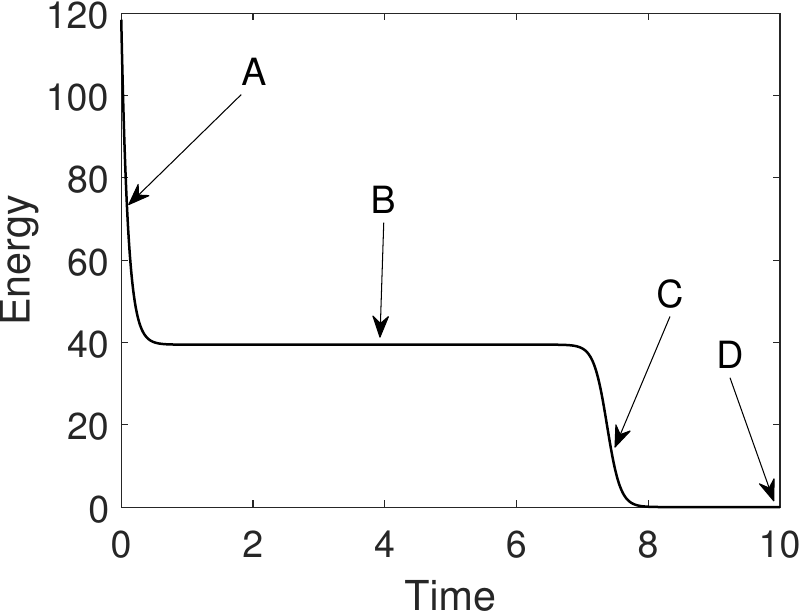}}  
  \subfigure[Step size vs Time]{  \includegraphics[scale=.42]{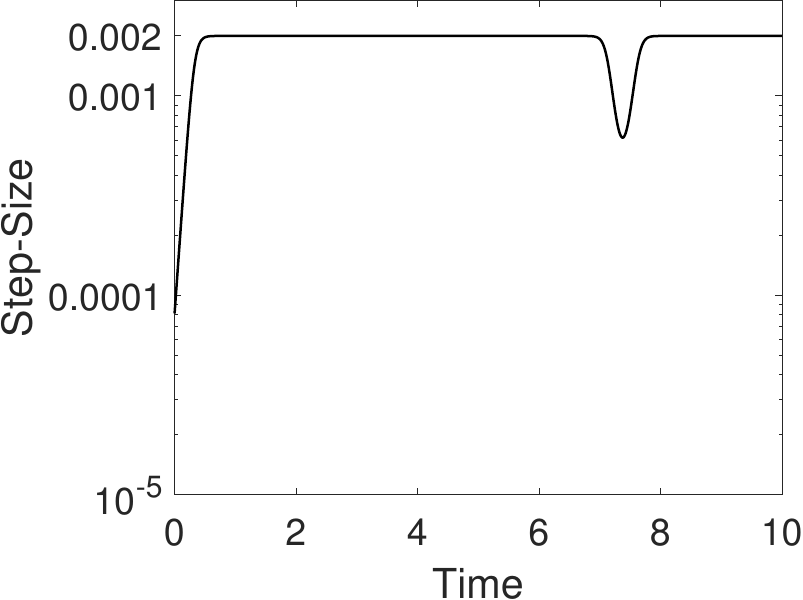}} \quad \quad \quad   \quad 
  \subfigure[Computational cost vs Time]{  \includegraphics[scale=.42]{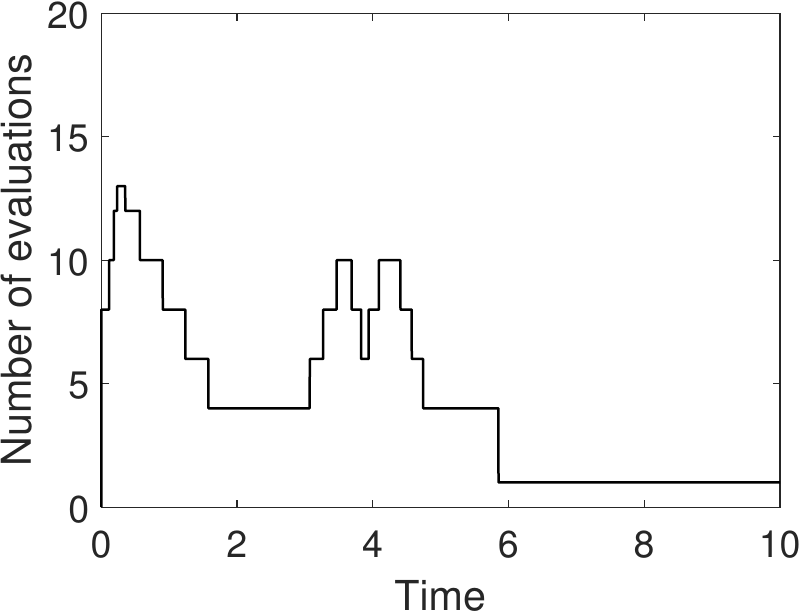}}
    \subfigure[$t=0.0828$]{  \includegraphics[scale=.33] {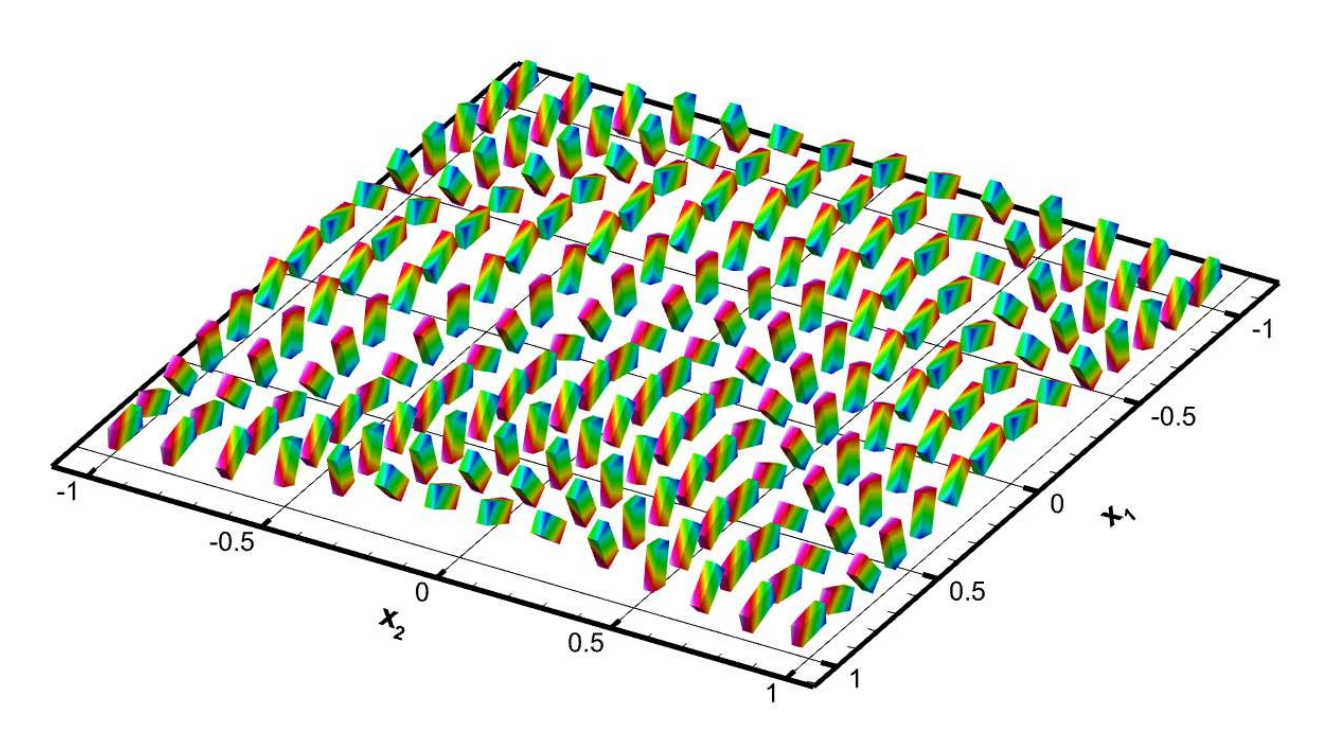}}  
  \subfigure[$t=4.0528$]{  \includegraphics[scale=.33] {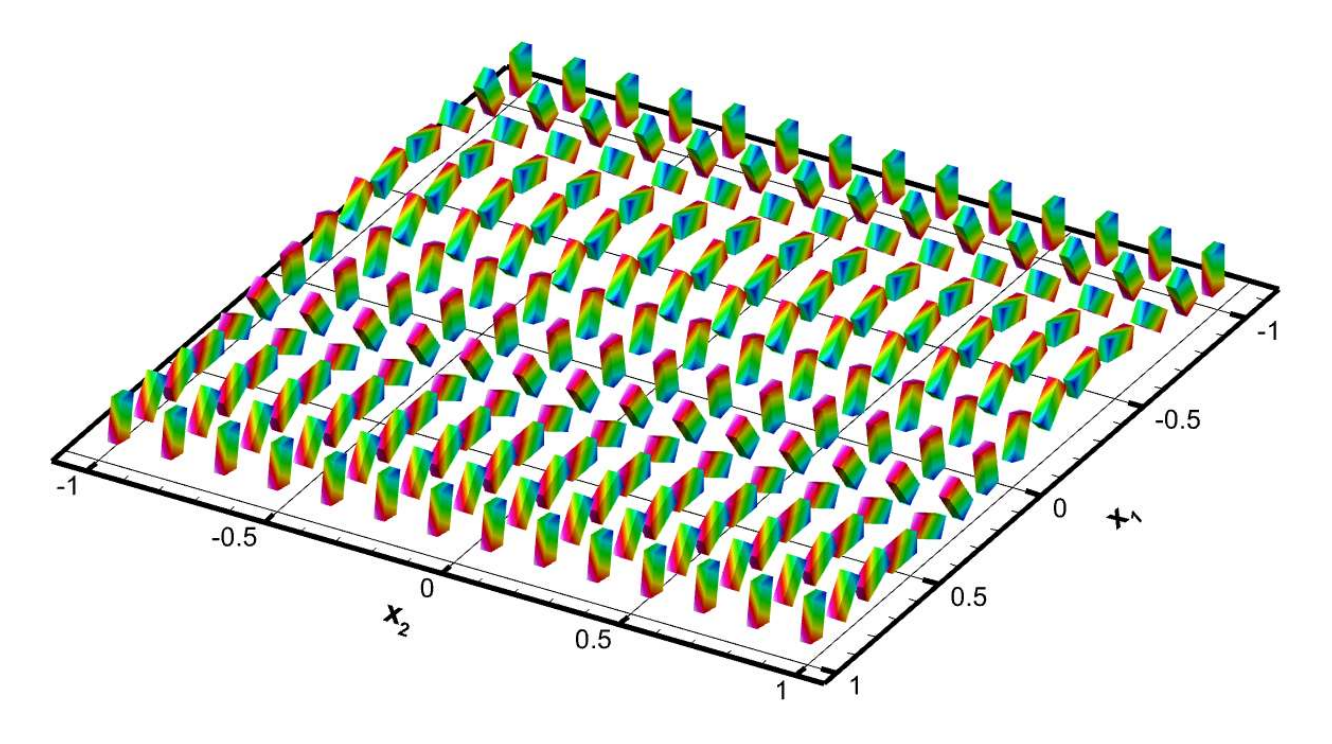}}  
  \subfigure[$t=7.4126$]{  \includegraphics[scale=.33] {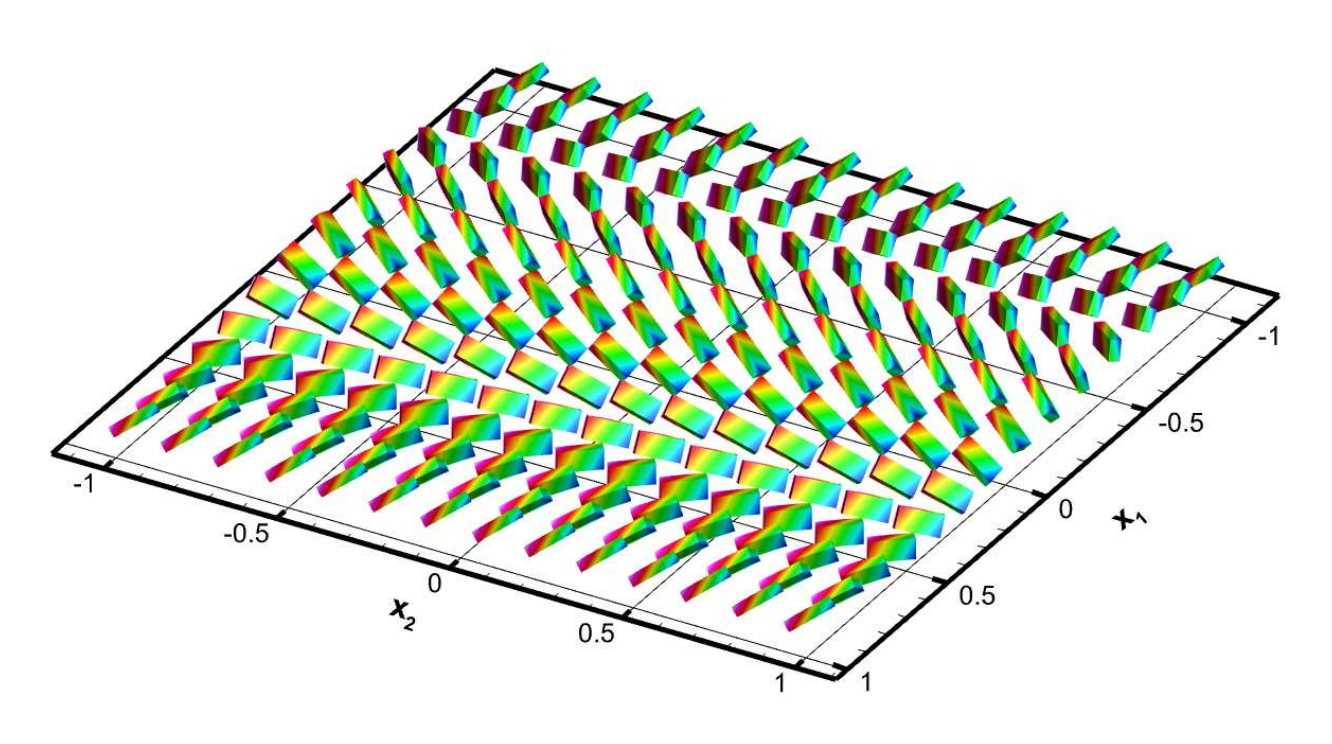}}  
  \subfigure[$t=10.1341$]{  \includegraphics[scale=.33] {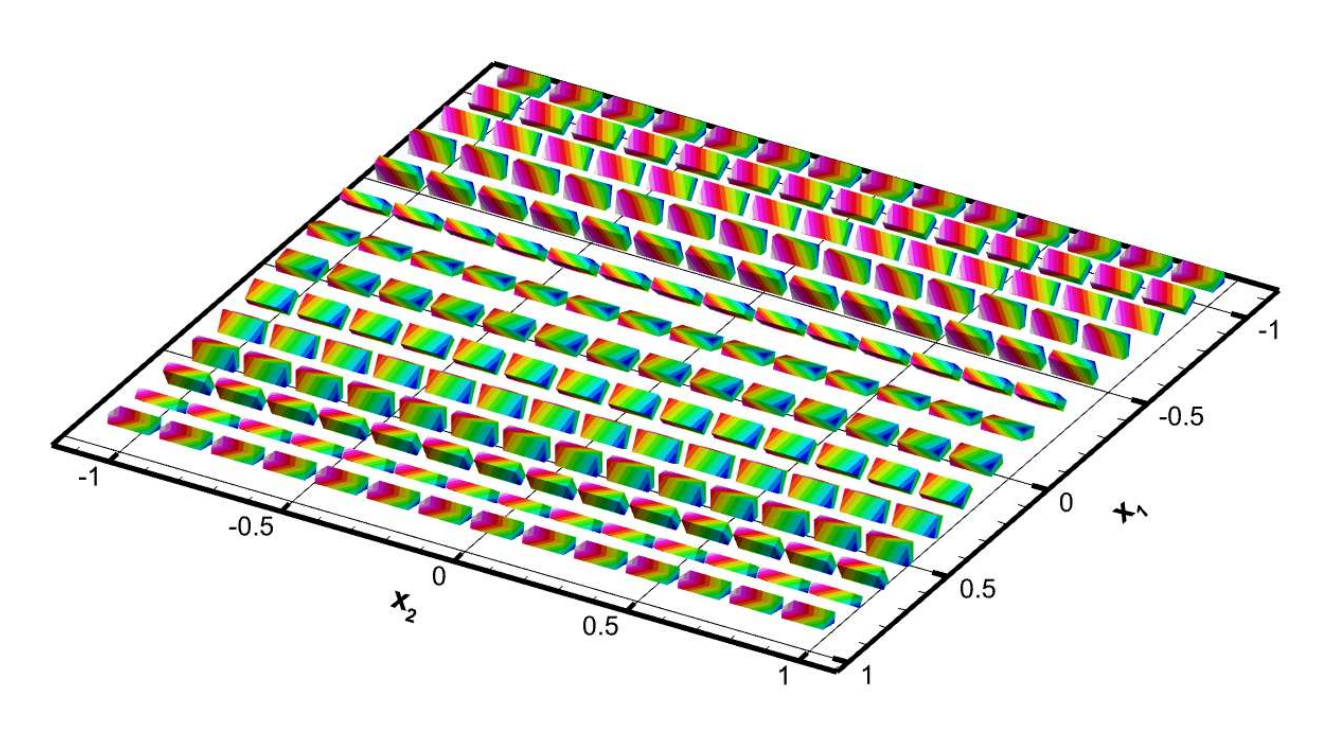}} 	
  \caption{\small Evolution of the orthonormal frame field $\mathbf{p}$ with initial frame field in equation \eqref{eq:utest2} and the elastic coefficients $\hat{\bs K}=(1, 0, 0, 1, 0, 0, 1, 0, 0, 1, 0, 0)^{\intercal}$. (a) Initial frame field distribution; (b) time history of energy $\mathcal{F}_{B i}[\mathbf{p}]$; (c) time history of adaptive step size; (d) time history of the number of residual evaluations in the INK solver at each time step; Time snapshots of frame field at (e) $t=0.0828$; (f) $t=4.0528$; (g)  $t=7.4126$; (h) $t=10.1341$.
  }  
  \label{figs: property-testeq3.4}
  \end{center} 
\end{figure}

We then fix the initial field distribution defined in equation \eqref{eq:utest1} and use the elastic coefficients $\hat{\bs K}=(1, 1, 1, 1, 1, 1, 1, 1, 1, 1, 1, 1)^{\intercal}$. 
In Figure \ref{figs: property-test2snap} (e), one can observe that the energy first decays rapidly at state A, and keep relatively stable at state B, then decays rapidly again at state C and keeps almost zero at state D. In accordance with the decay of the energy curve, we depict in Figure \ref{figs: property-test2snap} (a)-(d) the profiles of the frame field at these four different stages.
Figure \ref{figs: property-test2snap} (a) plots the snapshot at stage A at $t=0.0723$, indicating that the direction of $\bs n_1$  tends towards the $x_3$-direction.
Figure \ref{figs: property-test2snap} (b) plots the snapshot at $t=3.0445$.
We observe that it looks quite the same with the equilibrium stage in Figure \ref{figs: property-test1} (b), where the direction of $\bs n_1$ are homogeneous while the directions of $\bs n_2$ and $\bs n_3$ are still nonuniform.
Figure \ref{figs: property-test2snap} (c) plots the snapshot at $t=3.8908$, where the directions of $\bs n_2$ and $\bs n_3$ are still nonuniform along $x_1$-direction and the directions of $\bs n_1$ are almost lying on the $x_1-x_2$ plane.
Figure \ref{figs: property-test2snap} (d) plots the snapshot of equilibrium stage D, where the frame field is uniformly in every direction.
We plot in Figure \ref{figs: property-test2snap} (f) the maximum numerical error of all entries of $\mathbf{p}\mathbf{p}^{\intercal}-\mathbf{I}$ measured in point-wise $L^\infty$-norm and find that the error curves level off between $10^{-6}$ and $10^{-7}$, demonstrating the preservation of orthonormality constraint up to the tolerance of the INK solver.

\begin{figure}[tbp]
  \begin{center}  
    \subfigure[$t=0.0723$]{  \includegraphics[scale=.33] {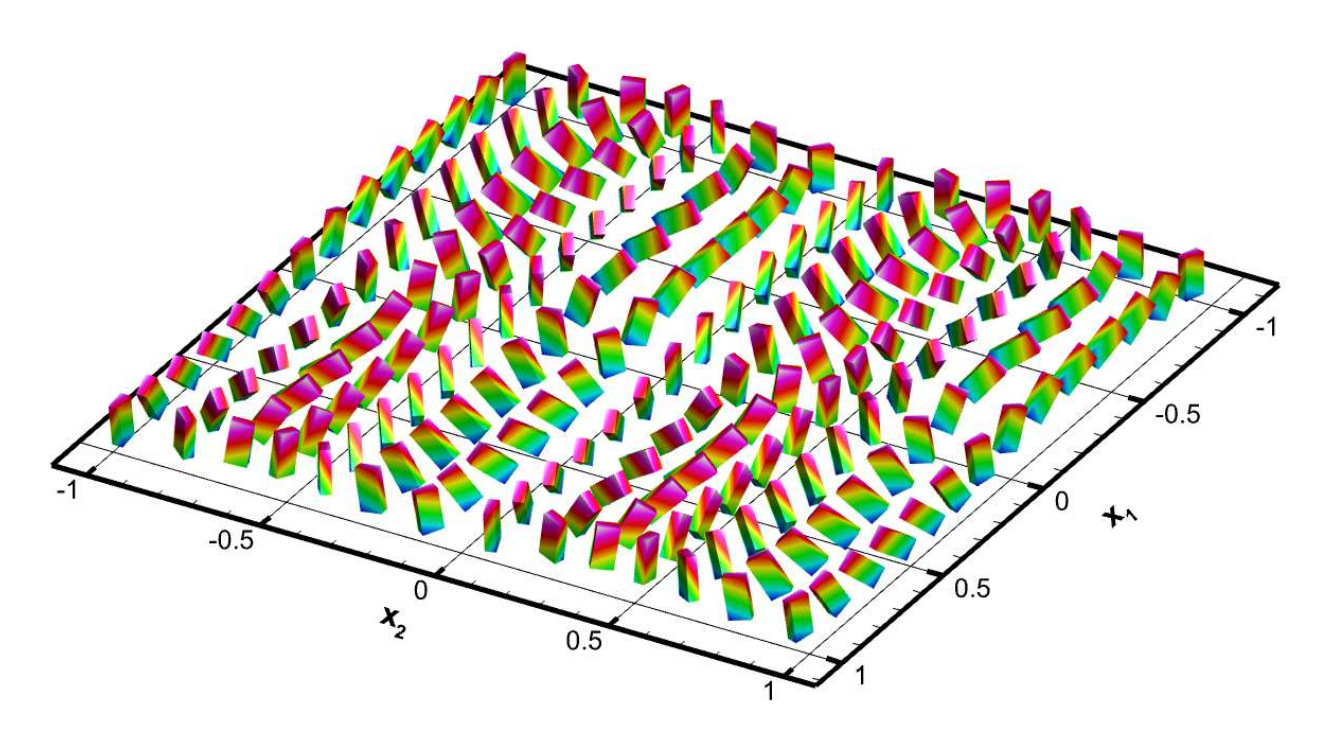}}  
    \subfigure[$t=3.0445$]{  \includegraphics[scale=.33] {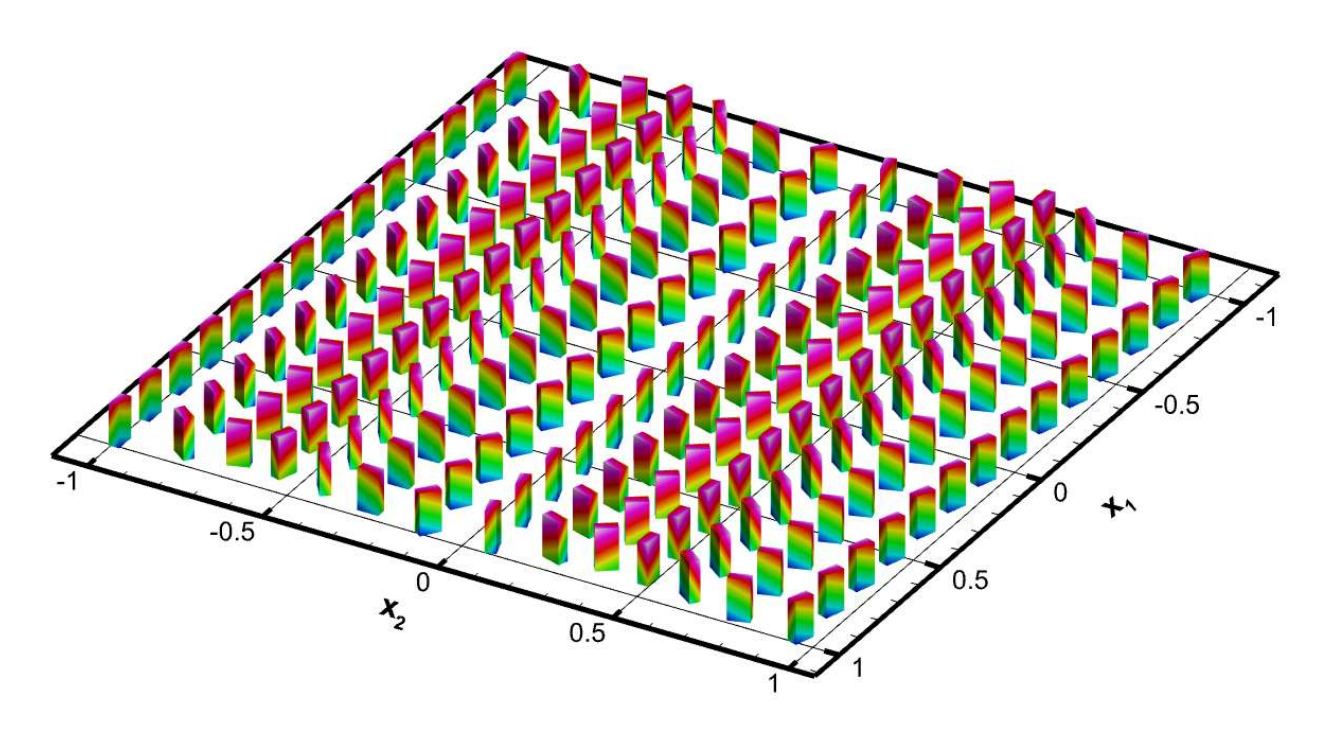}}  
    \subfigure[$t=3.8908$]{  \includegraphics[scale=.33] {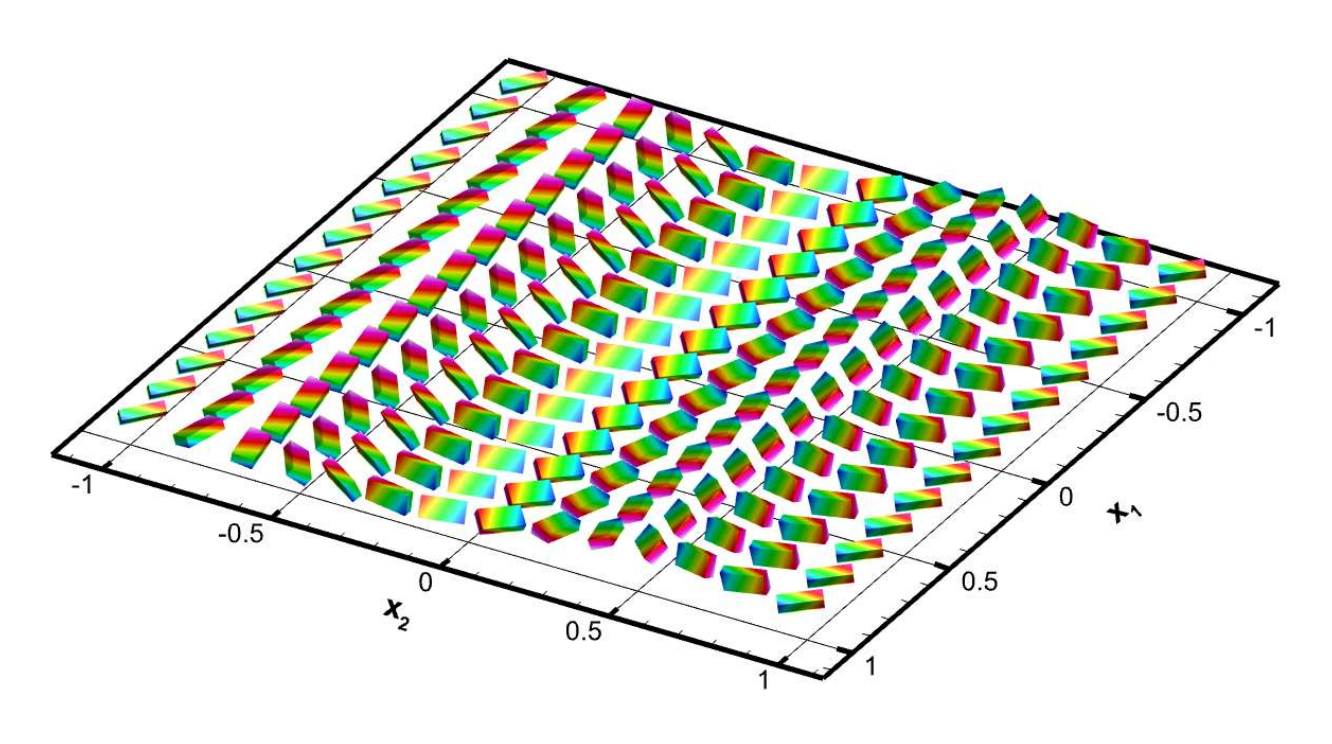}}  
    \subfigure[$t=9.8209$]{  \includegraphics[scale=.33] {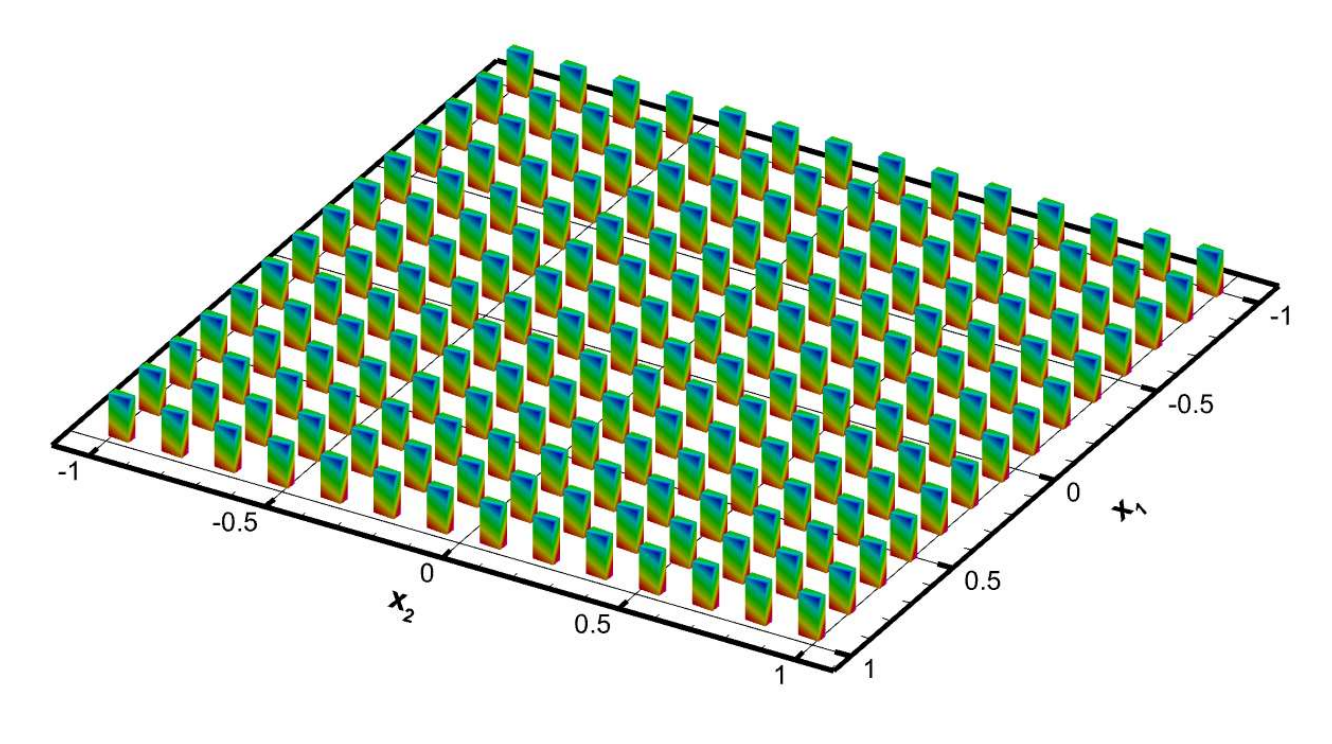}}
    \subfigure[Energy vs Time]{  \includegraphics[scale=.48]{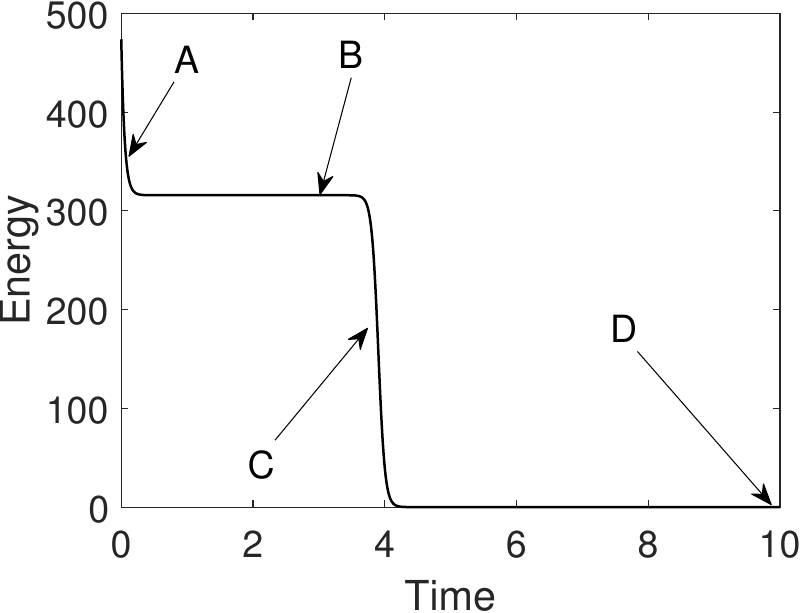}} \quad 
  \subfigure[Error vs Time]{  \includegraphics[scale=.48]{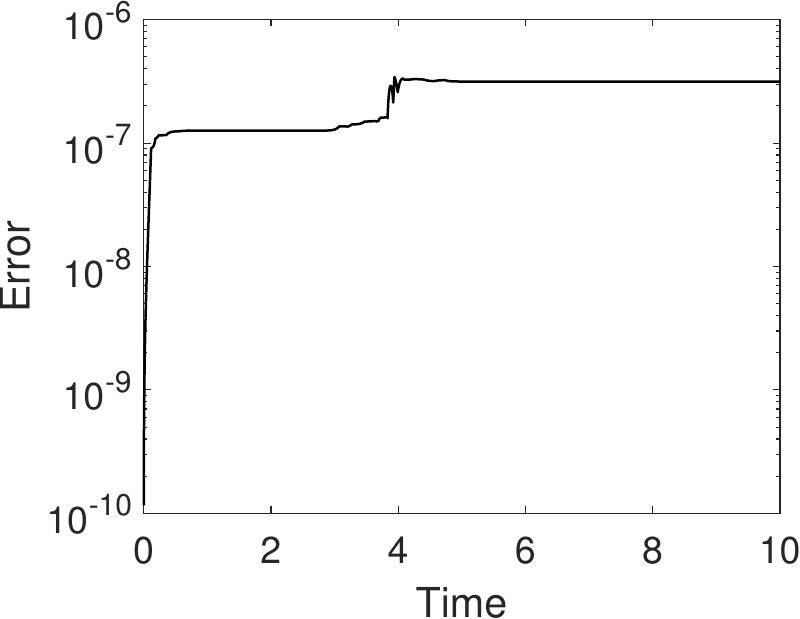}}
    \caption{\small Evolution of the orthonormal frame field $\mathbf{p}$ with initial frame field distribution in equation \eqref{eq:utest1} and the elastic coefficients $\hat{\bs K}=(1,1,1,1, 1,1,1,1,1,1,1,1)^{\intercal}$. Time history of snapshots of the frame field at  (a) $t=0.0723$; (b) $t=3.0445$; (c)  $t=3.8908$; (d) $t=9.8209$; (e) time history of energy $\mathcal{F}_{B i}[\mathbf{p}]$; (f) time history of orthonormal frame error measured in point-wise $L^{\infty}$-norm.}     
    \label{figs: property-test2snap}
  \end{center} 
\end{figure}

\subsection{Dynamics against anisotropic elasticity}
In this subsection, we challenge our numerical algorithm with highly anisotropic elasticity and adopt the elastic coefficients to be 
\begin{equation*}
\hat{\bs K}= (0.05, 0.45, 3.75, 0.15, 0.35, 1.75, 5.55, 2.25, 3.955, 0.255, 1.955, 1.55)^{\intercal},
\end{equation*}
which is obtained from molecular parameters of bent-core molecules (see \cite{xu2018calculating}). In the following simulations, we enlarge the number of Fourier collocation points in one direction to $60$ and fix the rest of the simulation setting the same as above. 

Let us prescribe the initial profile of the frame field defined in the equation \eqref{eq:utest1}. We depict in Figure \ref{figs: bicase1} (g) the history of energy. It can be observed that  the energy  decays  rapidly in stage A; and keep relatively stable in stage B; then decays rapidly again in stage C; the rate of energy decay turns to be slower in stage D; then the energy decays dramatically again in stage E;  the  energy decays to nearly zero and keeps stable in stage F.  Figures \ref{figs: bicase1} (a)-(e) plot the snapshots of the profile of frame field in each stages. Figure \ref{figs: bicase1} (a) shows the snapshot of the frame field at $t=0.0426$, and one can see the direction of $\bs n_1$ tends towards the $x_3$-axis, compared with the initial profile. Figure \ref{figs: bicase1} (b) depicts the profile at $t=0.0425$, and it can be found that the trend that stage A has shown is more obvious in stage B. Then in stage C, such a trend does not keep, while there is a rotation around $x_3$-axis, see Figure \ref{figs: bicase1} (c). Figure \ref{figs: bicase1} (d) plots the profile at $t=0.8191$, and there is little difference from stage B and stage C. The snapshot at $t=0.9821$ is shown in Figure \ref{figs: bicase1} (e). One can observe that the direction turns to be nearly homogeneous in the $x_1$-axis. Finally, in Figure \ref{figs: bicase1} (f), we plot the frame field at $t=9.991$. The directions of $\bs n_1$, $\bs n_2$ and $\bs n_3$ are all homogeneous, and there is a clear angle between the direction of $\bs n_1$ and $x_1$-$x_2$ plane.
We also plot the history of adaptive time step size, and we can see that the step size is relatively small when energy decays rapidly and keeps as large as $\tau_{max}$ for most of the simulation, which demonstrates that the gRdg scheme using the proposed  time-adaptivity strategy is efficient for simulations of dynamics against anisotropic elasticity.

\begin{figure}[tbp]
  \begin{center}  
    \subfigure[$t=0.0426$]{  \includegraphics[scale=.33] {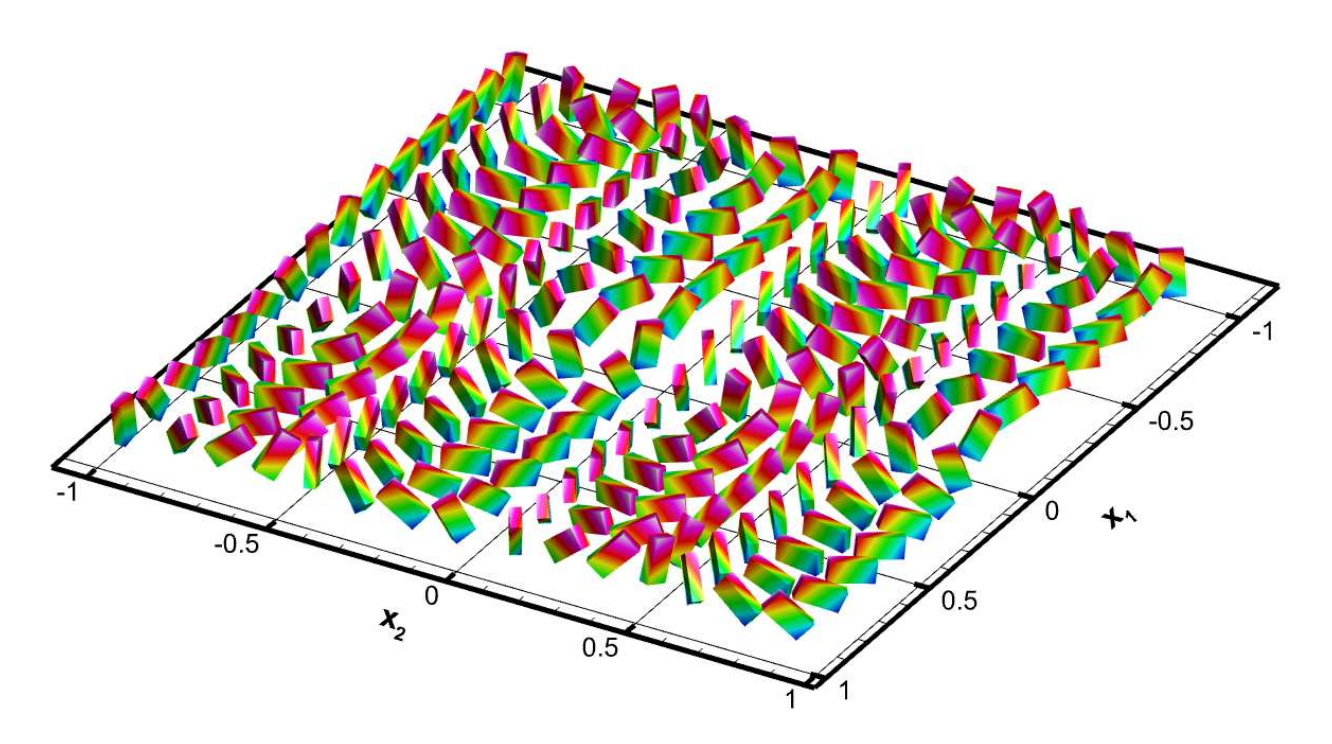}}  
    \subfigure[$t=0.4625$]{  \includegraphics[scale=.33] {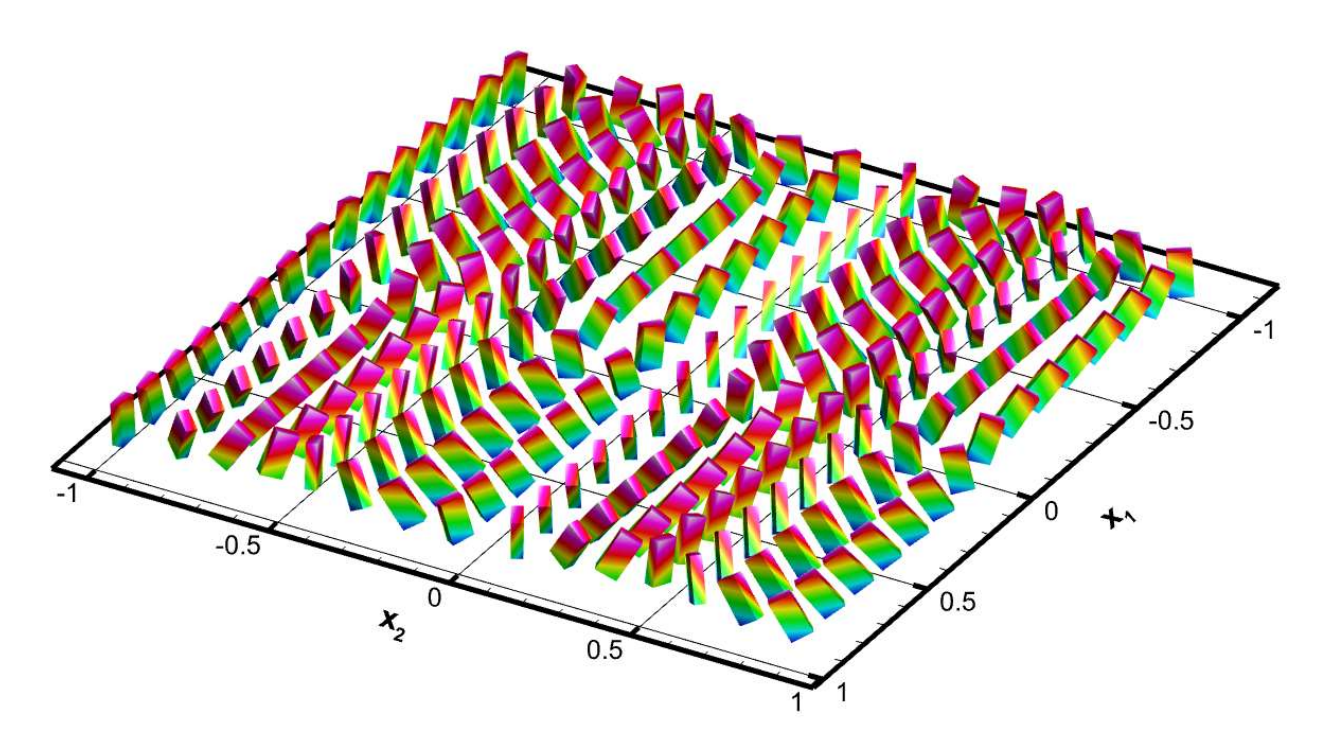}}  
    \subfigure[$t=0.7694$]{  \includegraphics[scale=.33] {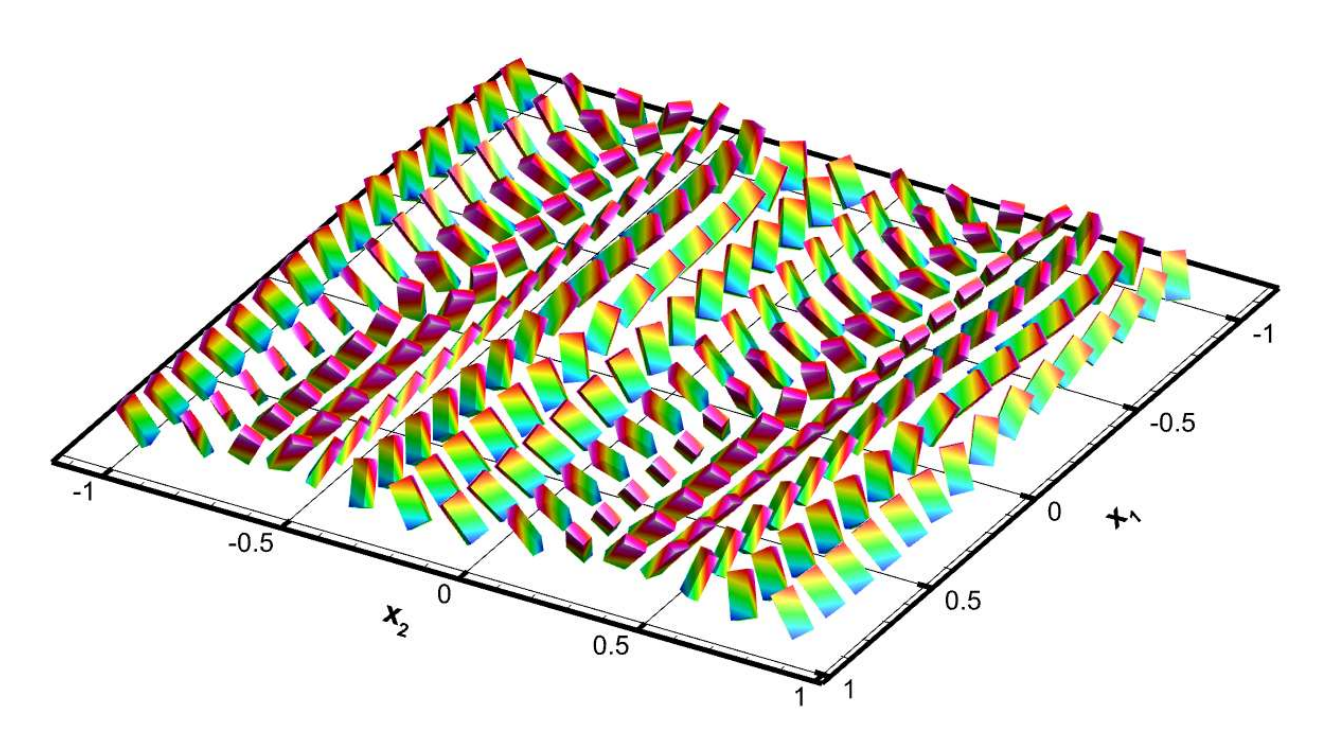}}  
    \subfigure[$t=0.8191$]{  \includegraphics[scale=.33] {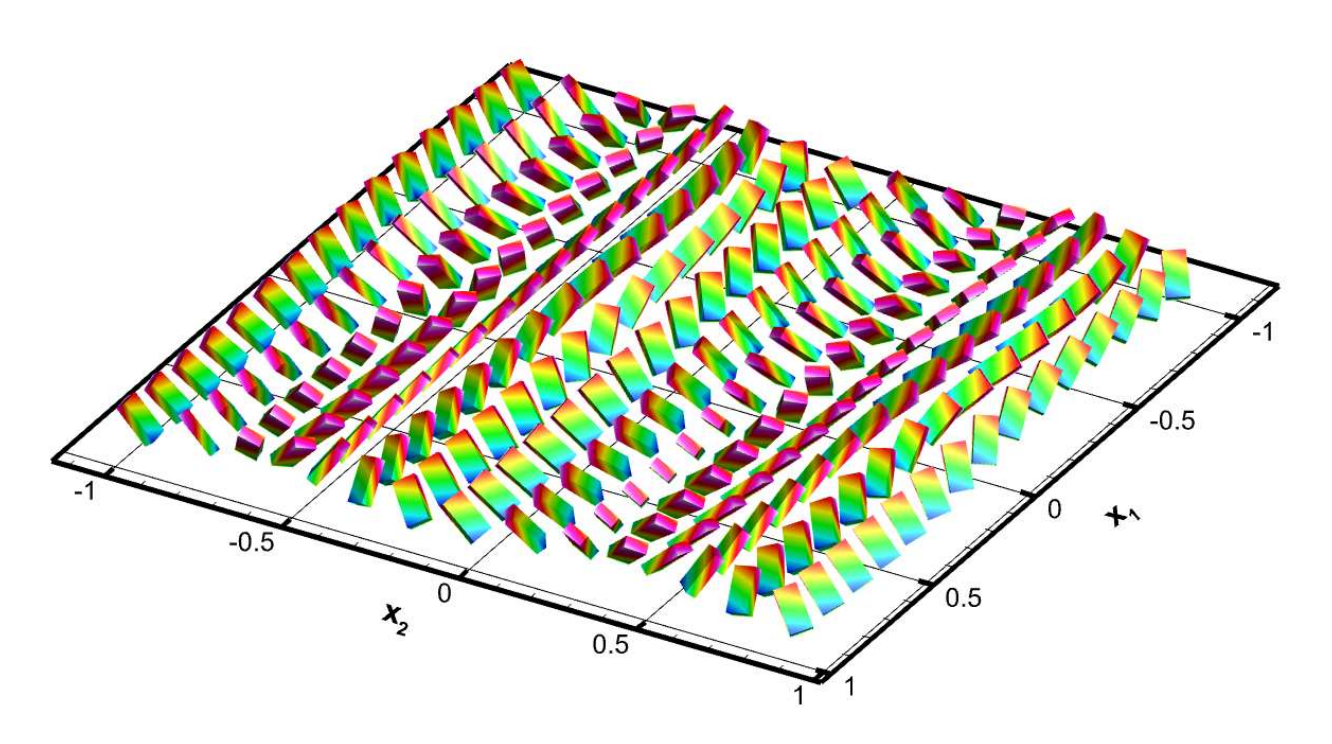}}
    \subfigure[$t=0.9821$]{  \includegraphics[scale=.33] {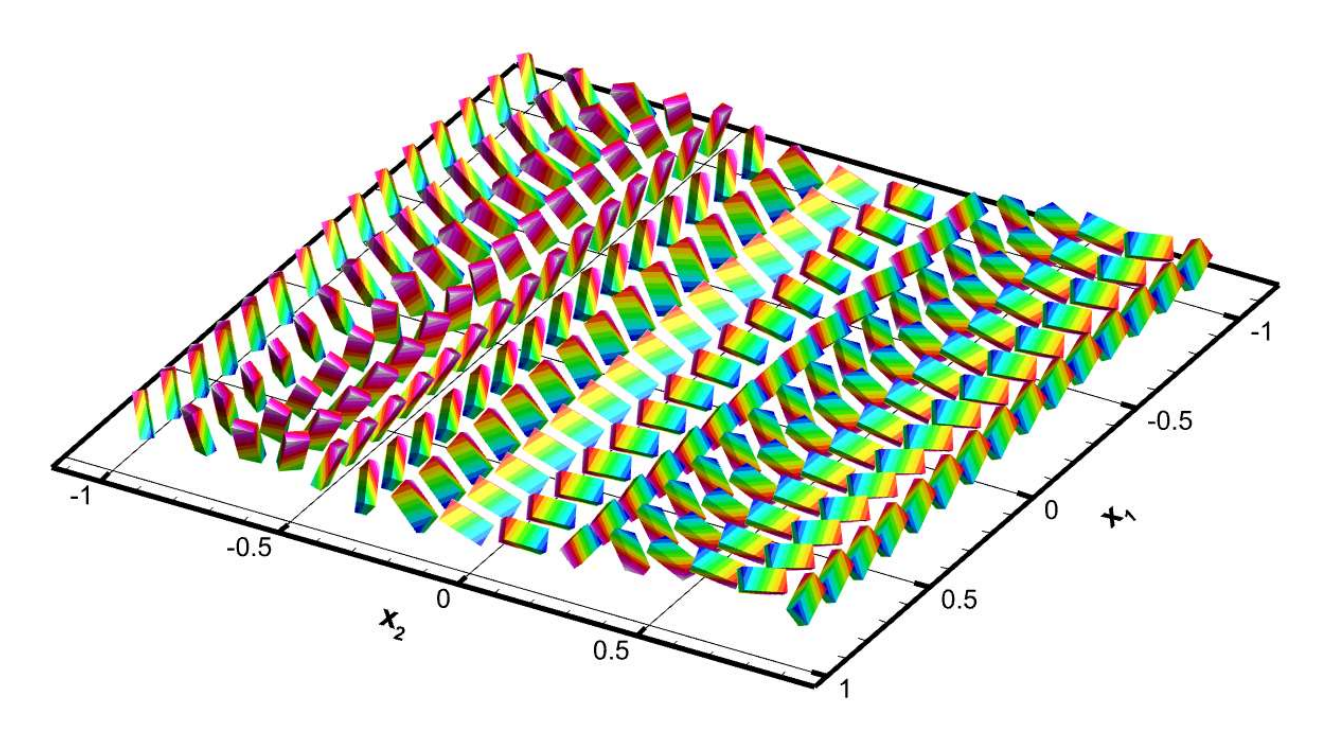}}  
    \subfigure[$t=9.991$]{  \includegraphics[scale=.33] {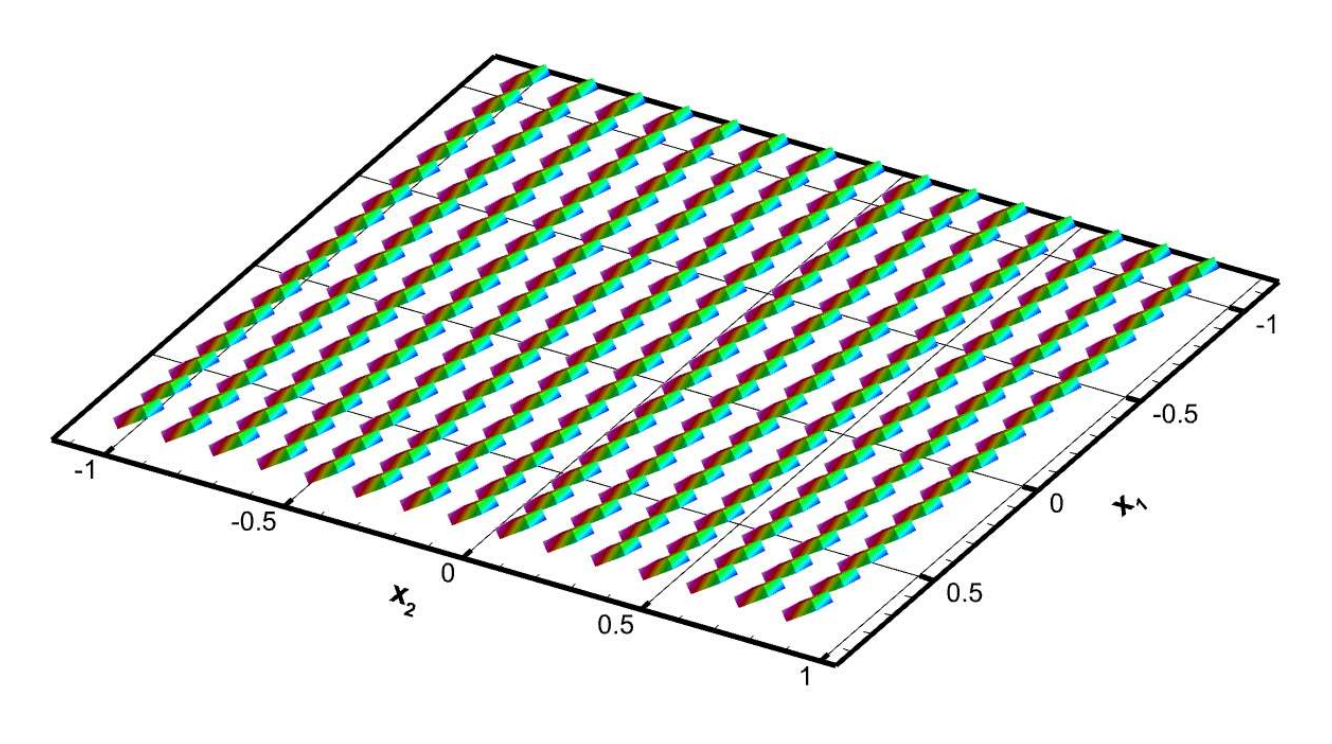}}
      \subfigure[Energy vs Time]{  \includegraphics[scale=.42]{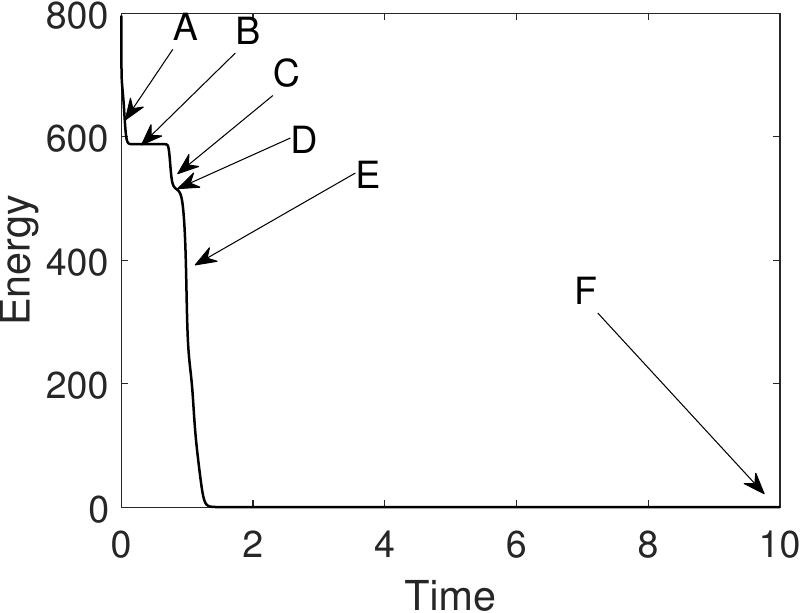}} \quad \quad \quad \
  \subfigure[Step size vs Time]{  \includegraphics[scale=.42]{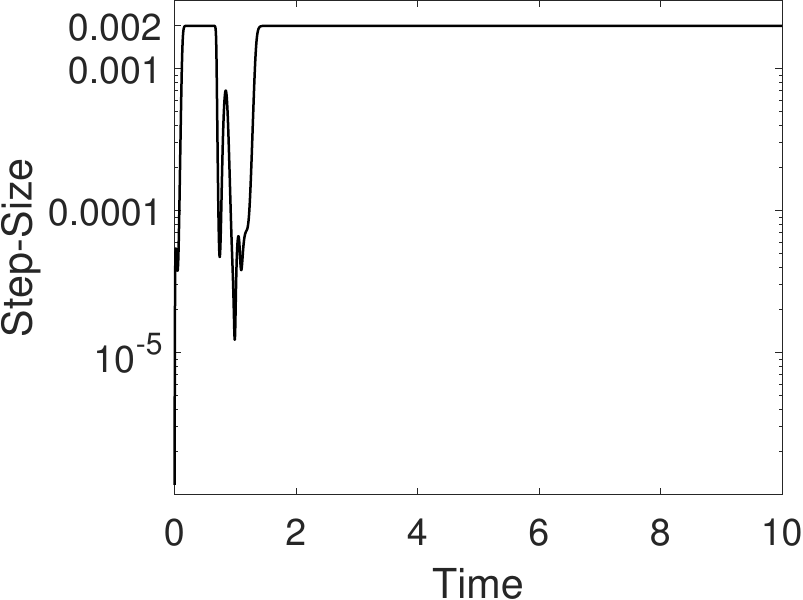}}
    \caption{\small Evolution of the orthonormal frame field $\mathbf{p}$ with initial frame field distribution \eqref{eq:utest1} and $\hat{\bs K}=(0.05, 0.45, 3.75, 0.15, 0.35, 1.75, 5.55, 2.25, 3.955, 0.255, 1.955, 1.55)^{\intercal}$. Time history of snapshots of the frame field at (a) $t=0.0426$, (b) $t=0.4625$, (c)  $t=0.7964$; (d) $t=0.8191$; (e)  $t=0.9821$; (f) $t=9.991;$ (g) time history of energy $\mathcal{F}_{B i}[\mathbf{p}]$; (h) time history of adaptive step size.
}     
    \label{figs: bicase1}
  \end{center} 
\end{figure}

Then, we take the initial profile of frame field defined in equation \eqref{eq:utest2}.
Figure \ref{figs: bicase2} (c) shows the time history of energy, and one can observe that it decays rapidly for a while and then keeps stable. In fact, we have continued  calculation until $t=300$ but only to see that this stable stage still remains.
Figure \ref{figs: bicase2}  (a) and (b) depict the snapshots of these two stages, and there is only a slight rotation in the $x_1$-$x_3$ plane.
We further introduce a  perturbation to the profile of the frame field by rotating it with $\pi/2$ angle around $x_3$-axis within the region $x_1^2+x_2^2 < 0.04$.
We depict in Figure \ref{figs: bicase2} (e) the snapshot of simulation, initialized with the perturbation. Noticeable differences in the central region can be found, when compared with the profile in Figure \ref{figs: bicase2} (b).
The time history of energy after perturbation  is shown in Figure \ref{figs: bicase2} (d), and one can see that the energy is quite larger than the energy before perturbation. One can also find that the energy decays rapidly into stable stage again and the energy of this new stable stage is almost equal to the energy of stable stage without perturbation.
The associated frame field is plotted in Figure \ref{figs: bicase2} (f), which is almost the same with Figure \ref{figs: bicase2} (b). Several different perturbations have been examined but they all lead to the same stable stage B in Figure \ref{figs: bicase2} (c), which implies that stage B is a local minimizer.
It would be an interesting problem to investigate what kind of local minimizers like the above one can be found, which calls for further studies. 

\begin{figure}[tbp]
  \begin{center}  
  \subfigure[$t=0.0192$]{  \includegraphics[scale=.33]{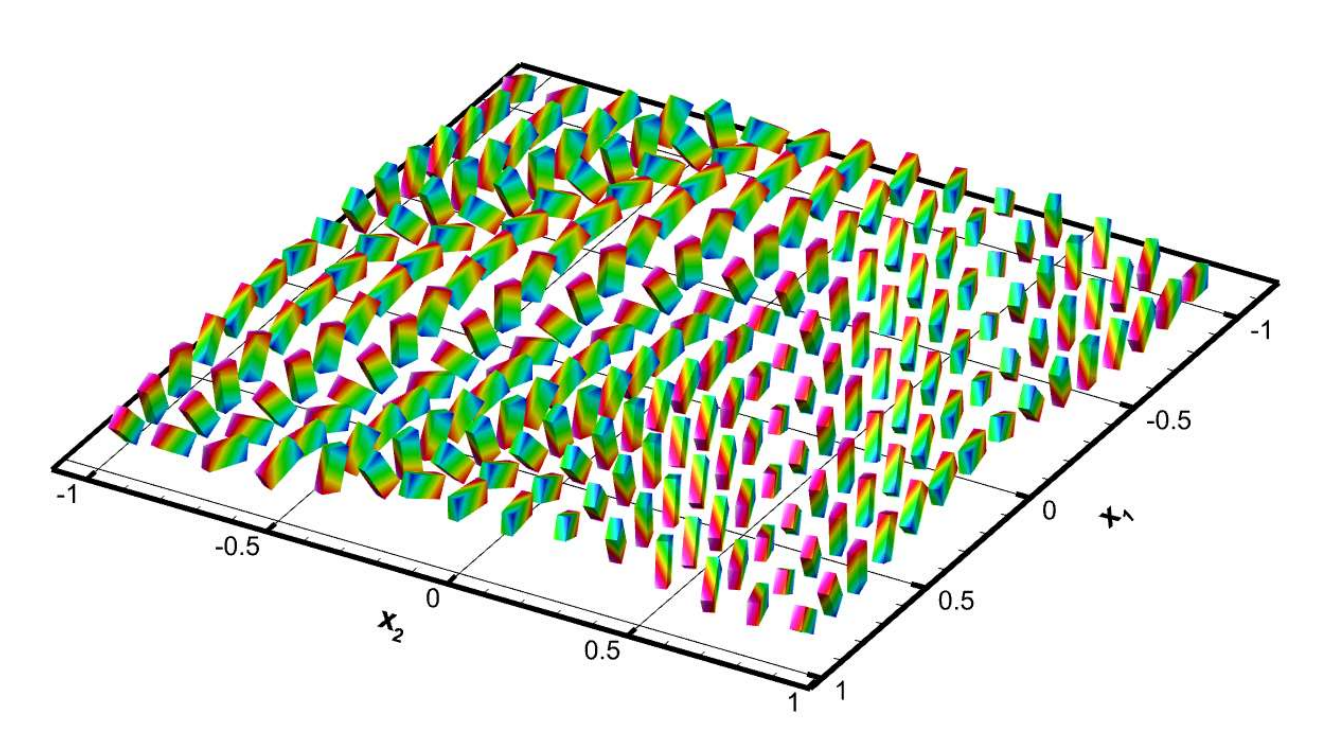}}
  \subfigure[$t=9.9822$]{  \includegraphics[scale=.33]{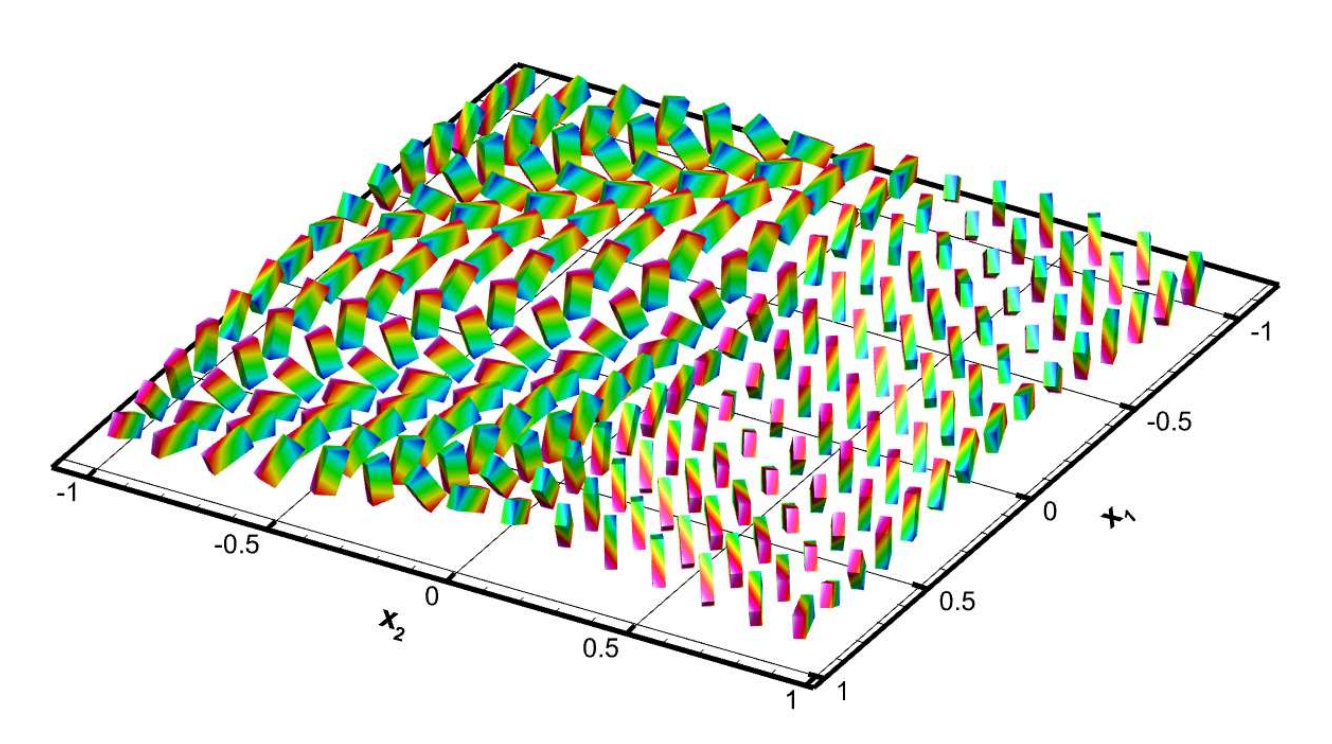}}
    \subfigure[Energy vs Time]{  \includegraphics[scale=.42]{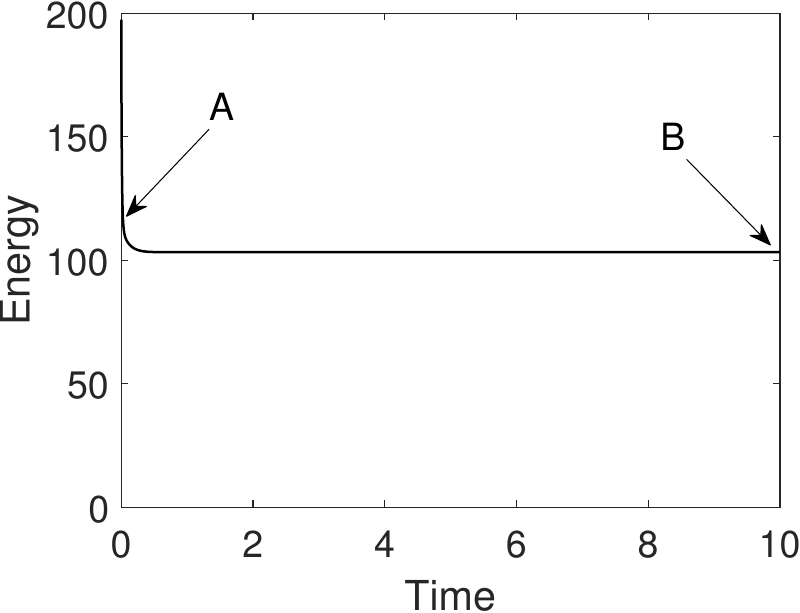}} \quad \quad \quad \
  \subfigure[Energy vs Time with Perturbation]{  \includegraphics[scale=.42]{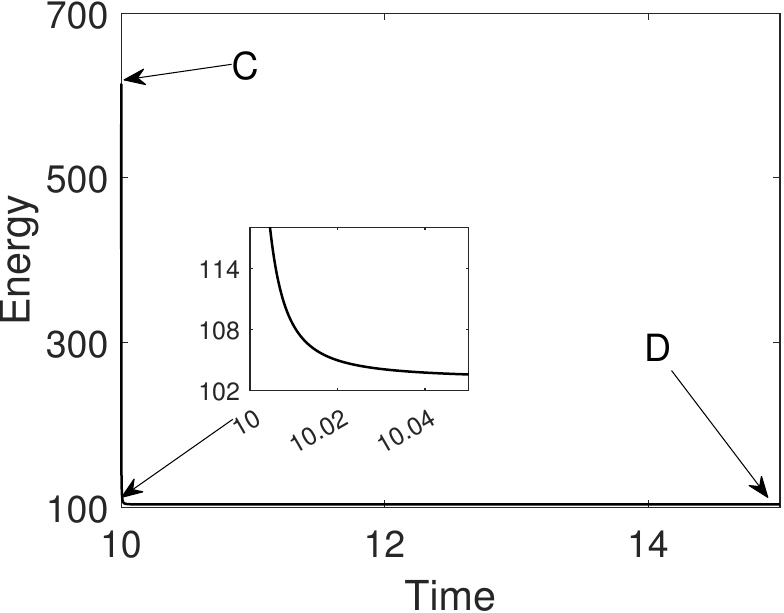}}
  \subfigure[$t=10.000$]{  \includegraphics[scale=.33]{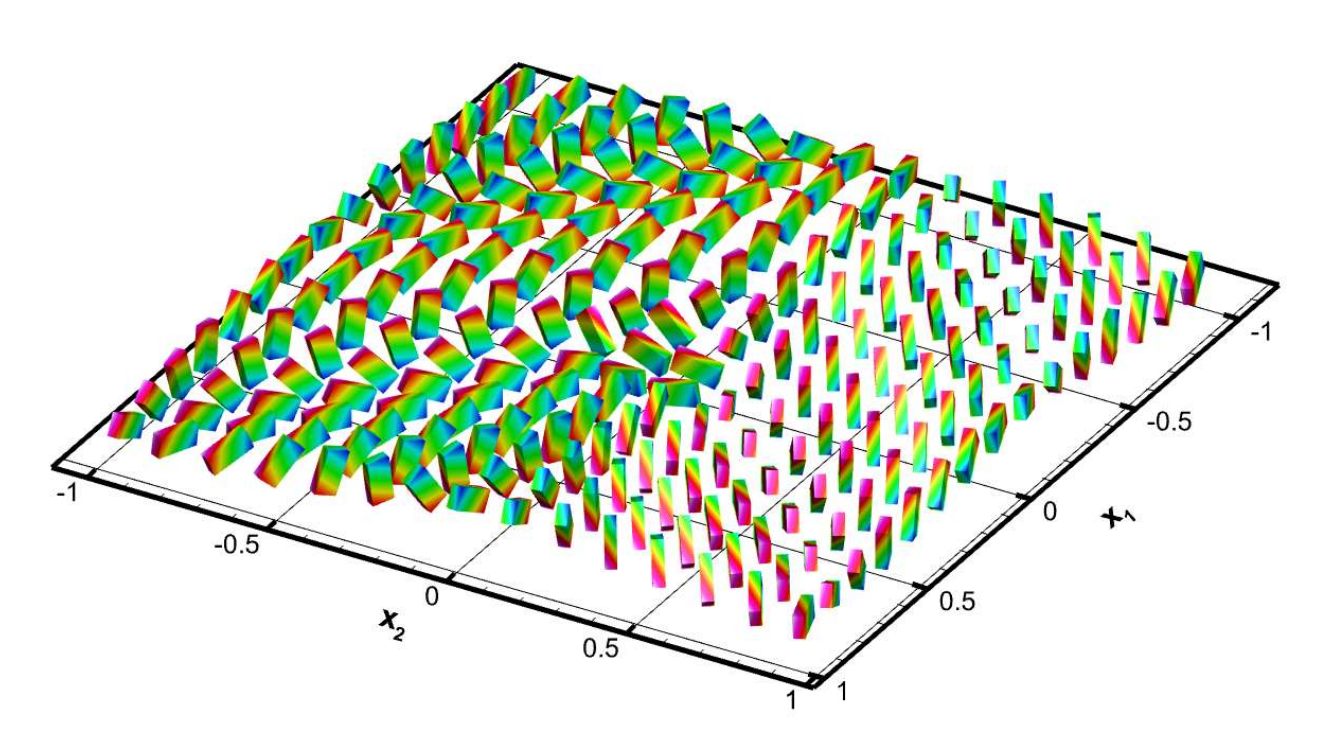}}
  \subfigure[$t=14.9292$]{  \includegraphics[scale=.33]{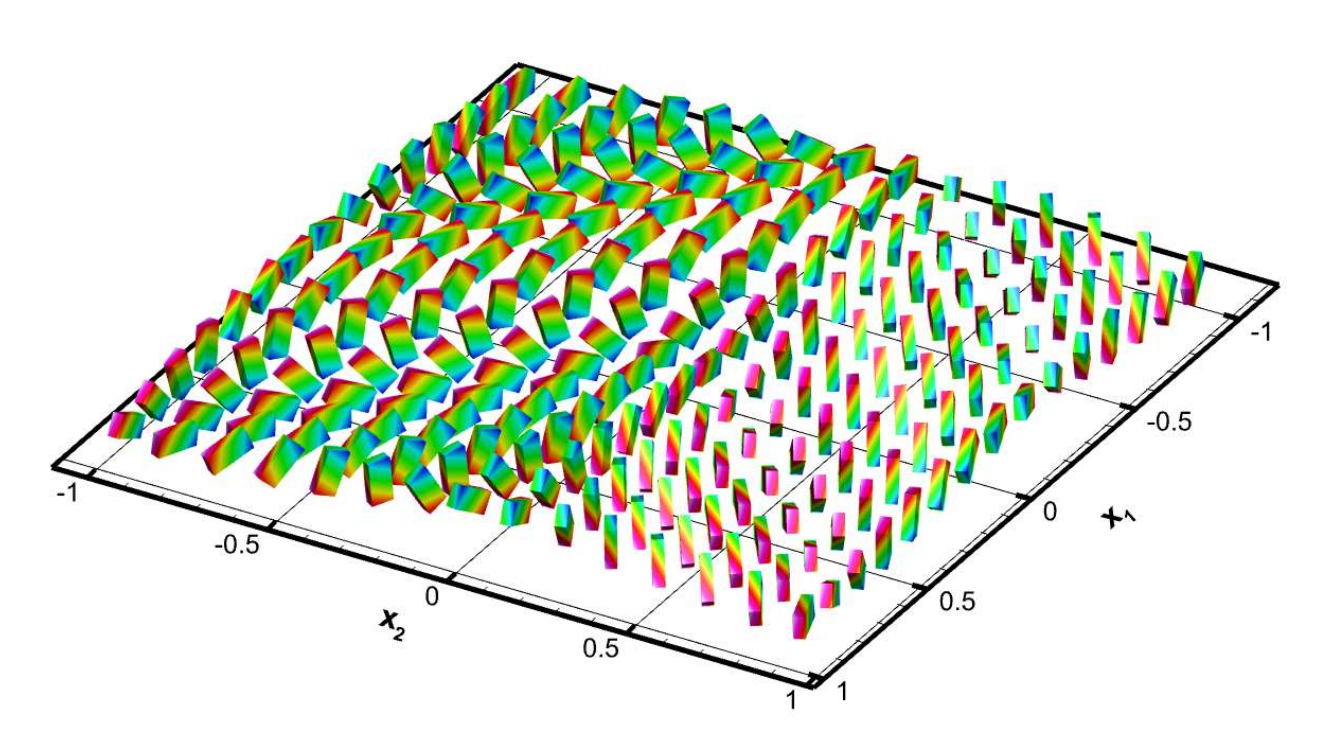}}
    \caption{\small Evolution of the orthonormal frame field $\mathbf{p}$ with initial frame field distribution  \eqref{eq:utest2} and $\hat{\bs K}=(0.05, 0.45, 3.75, 0.15, 0.35, 1.75, 5.55, 2.25, 3.955, 0.255, 1.955, 1.55)^{\intercal}$ with perturbation. Time history of snapshots of frame field at (a) $t=0.0192$ and (b) $t=9.9822$; (c) time history of energy $\mathcal{F}_{B i}[\mathbf{p}]$; (d) time history of energy $\mathcal{F}_{B i}[\mathbf{p}]$ with a rotation perturbation of the frame field at $t=10$; (e) profile of the perturbed frame field (f) profile of the equilibrium of the frame field at $t=14.9292$. } 
  \label{figs: bicase2}
\end{center} 
\end{figure}

\section{Concluding Remarks}\label{sect: remark}
In this paper, we have developed a novel gRdg scheme for orthonormal frame gradient flow system modeling the quasi-static evolution of biaxial liquid crystals. The proposed scheme is of second-order accuracy,  unconditionally energy-stable and $SO(3)$-preserving. Ample numerical experiments demonstrate its efficiency and accuracy and ability to compute highly anisotropic biaxial nematics. Since the orthonormal frame gradient flow system characterizes the dissipation of the frame field in the biaxial hydrodynamic models, the proposed method is expected to be a cornerstone for efficient and robust schemes of these problems.

\section{Acknowledgement}

J. Xu acknowledges the support from the NSFC (Nos. 12288201, 12001524 and 12371414). Z. Yang acknowledges the support from the NSFC (No. 12101399), the Shanghai Sailing Program (No. 21YF1421000), the Strategic Priority Research Program of Chinese Academy of Sciences (No. XDA25010402) and the Key Laboratory of Scientific and Engineering Computing (Ministry of Education). 

\appendix
\section{proof of equation \eqref{eq: TpSO(3)}}\label{sect: app}
\renewcommand{\theequation}{A.\arabic{equation}}
\begin{proof}
Let $\mathbf{p}_0\in SO(3)$ and define a curve passing $\mathbf{p}_0$ by
\begin{equation}\label{eq: auxiphi}
\phi(\eta)=\mathbf{p}_0{\rm exp}({\eta\mathbf{\Theta}}),\;\; \forall \mathbf{\Theta}\in so(3).
\end{equation}
We first show that $\phi(\eta)\in SO(3)$. Note that
\begin{align*}
& {\rm exp}({\eta\mathbf{\Theta}})=\mathbf{I}+\eta\mathbf{\Theta}+\frac{\eta^2}{2!}\mathbf{\Theta}^2+\frac{\eta^3}{3!}\mathbf{\Theta}^3+\cdots, \\ 
& {\rm exp}({\eta\mathbf{\Theta}})^{\intercal}=\mathbf{I}-\eta\mathbf{\Theta}+\frac{\eta^2}{2!}\mathbf{\Theta}^2-\frac{\eta^3}{3!}\mathbf{\Theta}^3+\cdots={\rm exp}({-\eta\mathbf{\Theta}}),
\end{align*}
where $\mathbf{\Theta}+\mathbf{\Theta}^{\intercal}=0$ has been used. With the help of the above equation, the fact that $\mathbf{p}_0\in SO(3)$ and the identity ${\rm det}({\rm exp}(\mathbf{A}))={\rm exp}({\rm tr}(\mathbf{A}))$ for any  $\mathbf{A}\in \mathbb{C}^{n\times n}$, we directly verify
\begin{align}
& \phi(\mathbf\eta)^{\intercal}\phi(\eta)=\mathbf{p}_0^{\intercal}{\rm exp}({-\eta\mathbf{\Theta}}){\rm exp}({\eta\mathbf{\Theta}})\mathbf{p}_0=\mathbf{I}, \label{eq: ppt}\\
&{\rm det}(\phi(\eta))={\rm det}(\mathbf{p}_0){\rm exp}({\text{tr}(\eta\mathbf{\Theta})})=1,\label{eq: det1}
\end{align}
from where we conclude that $\phi(\eta)\in SO(3)$.  Then, by a direct calculation 
\begin{equation}
\left.\frac{d \phi(\eta)}{d\eta}\right|_{\eta=0}=\mathbf{p}_0\mathbf{\Theta},
\end{equation}
we show $\mathbf{p}_0\mathbf{\Theta}$ belongs to the tangential plane to $SO(3)$ at $\mathbf{p}_0$, i.e. $\left\{\mathbf{p}_0 {\mathbf{\Theta}}\big|\; \forall \mathbf{\Theta} \in so(3)  \right\} \subset T_{\mathbf{p}_0}SO(3).$ Next, we show  $T_{\mathbf{p}_0}SO(3) \subset \left\{\mathbf{p}_0 {\mathbf{\Theta}}\big|\, \forall \mathbf{\Theta} \in so(3)  \right\}$. By equation \eqref{eq: ppt} and the fact that $\mathbf{\Theta}+\mathbf{\Theta}^{\intercal}=0$, one can obtain
\begin{equation}
\mathbf{\Theta}^{\intercal}\mathbf{p}_0^{\intercal}\mathbf{p}_0+\mathbf{p}_0^{\intercal}\mathbf{p}_0\mathbf{\Theta}     =\left.\frac{d \phi(\eta)^{\intercal}}{d \eta}\right|_{\eta=0}\mathbf{p}_0+\mathbf{p}_0^{\intercal}\left.\frac{d \phi(\eta)}{d \eta}\right|_{\eta=0}=\mathbf{0},
\end{equation}
which implies that $\mathbf{p}_0^{\intercal} \phi_{\eta}(\eta) \big|_{\eta=0} \in so(3)$. Thus, we can obtain from the identity 
\begin{equation}
\left.\frac{d \phi(\eta)}{d\eta}\right|_{\eta=0}=\mathbf{p}_0\mathbf{p}_0^{\intercal}\left.\frac{d \phi(\eta)}{d\eta}\right|_{\eta=0}
\end{equation}
that $T_{\mathbf{p}_0}SO(3) \subset \left\{{\mathbf{\Theta}}_{\mathbf{p}_0}:=\mathbf{p}_0 {\mathbf{\Theta}},\; \forall \mathbf{\Theta} \in so(3)  \right\}$. This proves the form of $T_{\mathbf{p}} S O(3)$ in equation \eqref{eq: TpSO(3)}. Finally, from the canonical basis $\big\{ \mathbf{\Theta}_i \big\}_{i=1}^3$ for $so(3)$, where
\begin{equation}
\mathbf{\Theta}_1=
\begin{bmatrix}
0 & 1 & 0\\
-1 & 0 & 0\\
0& 0 & 0
\end{bmatrix}, \quad 
\mathbf{\Theta}_2=
\begin{bmatrix}
0 & 0 & 1\\
0 & 0 & 0\\
-1& 0 & 0
\end{bmatrix},\quad 
\mathbf{\Theta}_3=
\begin{bmatrix}
0 & 0 & 0\\
0 & 0 & 1\\
0& -1 & 0
\end{bmatrix},
\end{equation} 
one arrives at the basis for  $T_{\mathbf{p}} S O(3)$ in equation \eqref{eq: Tpbasis}, which ends the proof.
\end{proof}
\section{Various forms of biaxial elastic energy}\label{sect: app2}
\renewcommand{\theequation}{B.\arabic{equation}}
It is worthwhile to note the relation between the biaxial energy functional and the classical uniaxial Oseen-Frank energy functional. Under  periodic  boundary conditions, if we take $K_7=K_{10},K_2=K_3=K_5=K_6=K_8=K_9=K_{11}=K_{12}=0$ in the equation \eqref{eq: density}, it reduces to the well-known Oseen-Frank energy functional, 
\begin{equation}\label{eq:oFrank}
  \mathcal{F}[\boldsymbol{n_1}]=\frac{1}{2} \int_{\Omega} K_1(\nabla \cdot \boldsymbol{n_1})^2+K_4|\boldsymbol{n_1} \cdot(\nabla \times \boldsymbol{n_1})|^2+K_7|\boldsymbol{n_1} \times(\nabla \times \boldsymbol{n_1})|^2d V.
\end{equation}

Besides, in order to relate the current elastic coefficients with molecule model, we also describe another equivalent form of the biaxial energy density (see e.g. \cite{stallinga1994theory,xu2018calculating}):
\begin{equation}\label{eq: density3}
  \begin{aligned}
  f_{B_i}[\mathbf{p}]=&\frac{1}{2} K_{1111} D_{11}^2+K_{2222} D_{22}^2+K_{3333} D_{33}^2 \\
  &+K_{1212} D_{12}^2+K_{2121} D_{21}^2+K_{2323} D_{23}^2+K_{3232} D_{32}^2+K_{3131} D_{31}^2+K_{1313} D_{13}^2 \\
  &+K_{1221} D_{12} D_{21}+K_{2332} D_{23} D_{32}+K_{1331} D_{13} D_{31},	
  \end{aligned}
\end{equation}
with 
\begin{equation}
  \begin{aligned}
  &D_{11}=n_{1 i} n_{2 j} \partial_i n_{3 j},\;\; D_{12}=n_{1 i} n_{3 j} \partial_i n_{1 j},\;\; D_{13}=n_{1 i} n_{1 j} \partial_i n_{2 j}, \\
  &D_{21}=n_{2 i} n_{2 j} \partial_i n_{3 j},\;\; D_{22}=n_{2 i} n_{3 j} \partial_i n_{1 j},\;\; D_{23}=n_{2 i} n_{1 j} \partial_i n_{2 j}, \\
  &D_{31}=n_{3 i} n_{2 j} \partial_i n_{3 j},\;\; D_{32}=n_{3 i} n_{3 j} \partial_i n_{1 j},\;\; D_{33}=n_{3 i} n_{1 j} \partial_i n_{2 j},
  \end{aligned}
\end{equation}
where the summation convention on repeated indices have been used.
It is direct to verify the following identities
\begin{align*}
  &D_{22}+D_{33} =\bs{n_1} \cdot({\nabla} \times \boldsymbol{n_1}), \quad D_{33}+D_{11} =\boldsymbol{n_2} \cdot({\nabla} \times \boldsymbol{n_2}),\\
  & D_{11}+D_{22}=\boldsymbol{n_3} \cdot({\nabla} \times \boldsymbol{n_3}),\quad D_{12}  =-\boldsymbol{n_2} \cdot({\nabla} \times \boldsymbol{n_1}),\\
  & D_{23}  =-\boldsymbol{n_3} \cdot({\nabla} \times \boldsymbol{n_2}),\qquad \quad D_{31}  =-\boldsymbol{n_1} \cdot({\nabla} \times \boldsymbol{n_3}), \\
  &D_{21}  =-\boldsymbol{n_1} \cdot({\nabla} \times \boldsymbol{n_2})={\nabla} \cdot \boldsymbol{n_3}-\boldsymbol{n_2} \cdot({\nabla} \times \boldsymbol{n_1}), \\
  &D_{32}  =-\boldsymbol{n_2} \cdot({\nabla} \times \boldsymbol{n_3})={\nabla} \cdot \boldsymbol{n_1}-\boldsymbol{n_3} \cdot({\nabla} \times \boldsymbol{n_2}), \\
  &D_{13}  =-\boldsymbol{n_3} \cdot({\nabla} \times \boldsymbol{n_1})={\nabla} \cdot \boldsymbol{n_2}-\boldsymbol{n_1} \cdot({\nabla} \times \boldsymbol{n_3}).
\end{align*}
Inserting the above identities into equation \eqref{eq: density3}, we have
\begin{equation}
\begin{aligned}\label{eq: equi}
  f_{B i}(\mathbf{p}, \nabla \mathbf{p})= & \frac{1}{4}\left\{ \left(-K_{1111}+K_{2222}+K_{3333}-K_{2332}\right)\left(\nabla \cdot \boldsymbol{n}_1\right)^2\right.\\
  &+\left(K_{1111}-K_{2222}+K_{33333}-K_{1331}\right)\left(\nabla \cdot \boldsymbol{n}_2\right)^2 \\
  &+\left(K_{1111}+K_{2222}-K_{3333}-K_{1221}\right)\left(\nabla \cdot \boldsymbol{n}_3\right)^2\\
  &+\left(-K_{1111}+K_{2222}+K_{3333}\right)\left(\boldsymbol{n}_1 \cdot \nabla \times \boldsymbol{n}_1\right)^2 \\
  &+\left(K_{1111}-K_{2222}+K_{3333}\right)\left(\boldsymbol{n}_2 \cdot \nabla \times \boldsymbol{n}_2\right)^2\\
  &+\left(K_{1111}+K_{2222}-K_{3333}\right)\left(\boldsymbol{n}_3 \cdot \nabla \times \boldsymbol{n}_3\right)^2\\
  &+\left(-K_{1111}+K_{2222}-K_{3333}+2K_{1313}+K_{1331}\right)\left(\bs{n}_3 \cdot \nabla \times \bs{n}_1\right)^2\\
  &+\left(-K_{1111}-K_{2222}+K_{3333}+2K_{2121}+K_{1221}\right)\left(\bs{n}_1 \cdot \nabla \times \bs{n}_2\right)^2\\
  &+\left(K_{1111}-K_{2222}-K_{3333}+2K_{3232}+K_{2332}\right)\left(\bs{n}_2 \cdot \nabla \times \bs{n}_3\right)^2\\
  &+\left(-K_{1111}-K_{2222}+K_{3333}+2K_{1212}+K_{1221}\right)\left(\bs{n}_2 \cdot \nabla \times \bs{n}_1\right)^2 \\
  &+\left(K_{1111}-K_{2222}-K_{3333}+2K_{2323}+K_{2332}\right)\left(\bs{n}_3 \cdot \nabla \times \bs{n}_2\right)^2\\
  &\left.+\left(-K_{1111}+K_{2222}-K_{3333}+2K_{3131}+K_{1331}\right)\left(\bs{n}_1 \cdot \nabla \times \bs{n}_3\right)^2\right\}.
\end{aligned}
\end{equation}

This equivalent form is chosen to obtain the elastic coefficients from molecular coefficients (see \cite{xu2018calculating} for detailed discussions). With the help of equation \eqref{eq: equi}, we can take such elastic coefficients into consideration.


\bibliography{refpapers}

\end{document}